\documentclass[10pt,letterpaper]{amsart}

\usepackage{amssymb}
\usepackage{amsfonts}
\usepackage{amsmath}
\usepackage[all]{xy}
\usepackage[mathscr]{euscript}
\usepackage{mathrsfs}
\usepackage{stmaryrd}
\usepackage{enumitem}

\setlist[enumerate,1]{label=\textnormal{(\arabic*)}}
\setlist[enumerate,2]{label=\textnormal{(\alph*)}}
\setlist[enumerate,3]{label=\textnormal{(\roman*)}}

\let\oldmarginpar\marginpar
\renewcommand\marginpar[1]{\-\oldmarginpar[\raggedleft\footnotesize #1]%
{\raggedright\footnotesize #1}}

\numberwithin{equation}{section}

   \newtheoremstyle{example}{\topsep}{\topsep}%
     {}
     {}
     {\bfseries}
     {}
     {\newline}
     {\thmname{#1}\thmnumber{ #2}\thmnote{ #3}}

\newtheorem{theorem}{Theorem}[section]
\newtheorem{proposition}[theorem]{Proposition}
\newtheorem{lemma}[theorem]{Lemma}
\newtheorem{corollary}[theorem]{Corollary}

\theoremstyle{definition}
\newtheorem{definition}[theorem]{Definition}

\theoremstyle{remark}
\newtheorem{rmk}[theorem]{Remark}

\theoremstyle{example}

\newcommand{\diff}{\mathcal{D}}
\newcommand{\struct}{\mathcal{O}}
\newcommand{\cplx}{\mathbb{C}}

\DeclareMathOperator{\A}{\mathcal{A}}

\newcommand{\Spec}{\mathrm{Spec}}
\newcommand{\proj}{\mathbb{P}}

\newcommand{\sC}{\mathscr{C}}
\newcommand{\sM}{\mathscr{M}}
\newcommand{\tM}{\widetilde{\sM}}

\newcommand{\bfA}{\mathbf{A}}
\newcommand{\bfW}{\mathbf{W}}
\newcommand{\bfs}{\mathbf{s}}
\newcommand{\bfWa}{\mathbf{W}^{\mathrm{aff}}}
\newcommand{\orbit}{\mathscr{O}}
\newcommand{\bfw}{\mathbf{w}}
\newcommand{\bhw}{\hat{\mathbf{w}}}

\newcommand{\n}{\nabla}

\newcommand{\Waff}{W^{\mathrm{aff}}}

\newcommand{\ord}{\mathrm{ord}}
\newcommand{\Lie}{\mathrm{Lie}}

\newcommand{\Z}{\mathbb{Z}}
\newcommand{\id}{\mathrm{I}}
\newcommand{\N}{\mathbb{N}}
\newcommand{\Q}{\mathbb{Q}}
\newcommand{\C}{\mathbb{C}}

\newcommand{\glo}{C}
\newcommand{\triv}{\mathrm{triv}}
\newcommand{\ov}{\overline}

\newcommand{\tfl}{\mathfrak{t}^{\flat}}
\newcommand{\Tfl}{T^\flat}
\newcommand{\fP}{\mathfrak P}
\newcommand{\fp}{\mathfrak p}
\newcommand{\bfP}{\bar{\mathfrak P}}
\newcommand{\fI}{\mathfrak I}
\newcommand{\bfI}{\bar{\mathfrak I}}
\newcommand{\fo}{\mathfrak{o}}
\newcommand{\ft}{\mathfrak{t}}
\newcommand{\fd}{\mathfrak{d}}
\newcommand{\fs}{\mathfrak{s}}
\newcommand{\bft}{\bar{\mathfrak{t}}}
\newcommand{\fh}{\mathfrak{h}}
\newcommand{\fz}{\mathfrak{z}}
\newcommand{\fq}{\mathfrak{q}}
\newcommand{\fg}{\mathfrak g}
\newcommand{\fn}{\mathfrak{n}}
\newcommand{\fu}{\mathfrak{u}}
\newcommand{\fw}{\mathfrak{w}}

\newcommand{\sL}{\mathscr{L}}
\newcommand{\bfe}{\mathbf{e}}
\newcommand{\bL}{\bar{L}}
\newcommand{\bM}{\bar{M}}
\newcommand{\bN}{\bar{N}}
\newcommand{\bpi}{\bar{\pi}}

\newcommand{\bn}{\beta_\nu}
\newcommand{\An}{A_\nu}

\newcommand{\bbn}{\bar{\beta}_\nu}
\newcommand{\nt}{\nabla_\tau}
\renewcommand{\a}{\alpha}
\renewcommand{\b}{\beta}
\newcommand{\an}{\a_\nu}
\newcommand{\ep}{\epsilon}

\DeclareMathOperator{\Der}{\mathrm{Der}}
\DeclareMathOperator{\Res}{\mathrm{Res}}
\DeclareMathOperator{\GL}{\mathrm{GL}}
\DeclareMathOperator{\gl}{\mathfrak{gl}}
\DeclareMathOperator{\Hom}{\mathrm{Hom}}

\DeclareMathOperator{\Ad}{\mathrm{Ad}}
\DeclareMathOperator{\ad}{\mathrm{ad}}

\DeclareMathOperator{\Tr}{\mathrm{Tr}}
\DeclareMathOperator{\diag}{\mathrm{diag}}
\DeclareMathOperator{\gr}{\mathrm{gr}}
\DeclareMathOperator{\End}{\mathrm{End}}
\DeclareMathOperator{\Aut}{\mathrm{Aut}}

\DeclareMathOperator{\rank}{\mathrm{rank}}
\DeclareMathOperator{\spa}{\mathrm{span}}
\DeclareMathOperator{\slope}{\mathrm{slope}}
\DeclareMathOperator{\res}{\mathrm{res}}

\newcommand{\tmu}{\tilde{\mu}}
\newcommand{\tmug}{\tilde{\mu}_{\GL_n}}

\title[Moduli Spaces of Irregular Singular Connections]{Moduli spaces
  of Irregular Singular Connections}
\author{Christopher L.~Bremer}
\address{Department of Mathematics\\
  Louisiana State University\\
  Baton Rouge, LA 70803} 
\email{cbremer@math.lsu.edu} 
\author{Daniel
  S.~Sage} 
\email{sage@math.lsu.edu} 
\thanks{The research of the
  second author was partially supported by NSF grant~DMS-0606300 and
  NSA grant~H98230-09-1-0059.}  
\subjclass[2010]{Primary:14D24, 22E57; Secondary: 34Mxx, 53D30} 
\keywords{meromorphic connections, irregular singularities, moduli
  spaces, parahoric subgroups, fundamental stratum, regular stratum}

\begin{document}
\begin{abstract}
  In the geometric version of the Langlands correspondence, irregular
  singular point connections play the role of Galois representations
  with wild ramification.  In this paper, we develop a geometric
  theory of fundamental strata to study irregular singular connections
  on the projective line.  Fundamental strata were originally used to
  classify cuspidal representations of the general linear group over a
  local field.  In the geometric setting, fundamental strata play the
  role of the leading term of a connection.  We introduce the concept
  of a regular stratum, which allows us to generalize the condition
  that a connection has regular semisimple leading term to connections
  with nonintegral slope.  Finally, we construct a moduli space of
  meromorphic connections on the projective line with specified formal
  type at the singular points.

\end{abstract}

\maketitle
\section{Introduction}

A fundamental problem in the theory of differential equations is the
classification of first order singular differential operators up to
gauge equivalence.  An updated version of this problem, rephrased into
the language of algebraic geometry, is to study the moduli space of
meromorphic connections on an algebraic curve $C/k$, where $k$ is an
algebraically closed field of characteristic $0$.  This problem has
been studied extensively in recent years due to its relationship with
the geometric Langlands correspondence.  To elaborate, the classical
Langlands conjecture gives a bijection between automorphic
representations of a reductive group $G$ over the ad\`eles of a global
field $K$ and Galois representations taking values in the Langlands
dual group $G^\vee$.  By analogy, meromorphic connections (or, to be
specific, flat $G^\vee$-bundles) play roughly the same role in the
geometric setting as Galois representations (see \cite{F, FGa} for
more background).  Naively, one would like to find a description of
the moduli space of meromorphic connections that resembles the space
of automorphic representations of a reductive group.

A more precise statement is that the geometric Langlands data on the
Galois side does not strictly depend on the connection itself, but
rather on the monodromy representation determined by the connection.
When the connection has regular singularities, i.e., when there is a basis
in which the matrix of the connection has simple poles at each
singular point, the Riemann-Hilbert correspondence states that the
monodromy representation is simply a representation of the fundamental
group.  However, when the connection is irregular singular, the
monodromy has a more subtle description due to the Stokes phenomenon.
The irregular monodromy data consists of a collection of Stokes
matrices at each singular point, which characterize the asymptotic
expansions of a horizontal section on sectors around each irregular
singular point (see \cite{Wa}, or \cite{Wi} for a modern treatment.)

The irregular Riemann-Hilbert map from moduli spaces of connections to
the space of Stokes matrices is well understood in the following
situation.  Let $V \cong \struct_C^n$ be a trivializable rank $n$
vector bundle, and let $\nabla$ be a meromorphic connection on $V$
with an irregular singular point at $x$.  After fixing a local
parameter $t$ at $x$, suppose that $\nabla$ has the following local
description:

\begin{equation*}
\nabla = d + M_r \frac{dt}{t^{r+1}} + M_{r-1} \frac{dt}{t^r} + \ldots,
\end{equation*}
where $M_j \in \gl_n(\cplx)$ and the leading term $M_r$ has pairwise
distinct eigenvalues.  Then, we say that $\nabla$ has a regular
semisimple leading term at $x$.  Under this assumption, Jimbo, Miwa
and Ueno \cite{JMU} classify the deformations of $\nabla$ that
preserve the Stokes data by showing that they satisfy a system of
differential equations (the so-called isomonodromy equations).  In
principle, the isomonodromy equations give a foliation of the moduli
space of connections, with each leaf corresponding to a single
monodromy representation.  Indeed, the Riemann-Hilbert map has
surprisingly nice geometric properties for connections with regular
semisimple leading terms.  Consider the moduli space $\mathscr{M}$ of
connections on $\proj^1$ which have singularities with regular
semisimple leading terms at $\{x_1, \ldots, x_m\}$ and which belong
to a fixed formal isomorphism class at each singular point.  Boalch,
whose paper \cite{Boa} is one of the primary inspirations for this
project, demonstrates that $\mathscr{M}$ is the quotient of a smooth,
symplectic manifold $\widetilde{\mathscr{M}}$ by a torus action.
Moreover, the space of Stokes data has a natural symplectic structure
which makes the Riemann-Hilbert map symplectic.

However, many irregular singular connections that arise naturally in
the geometric Langlands program do not have a regular semisimple
leading term.  In \cite[Section 6.2]{Wi}, Witten considers a
connection of the form
\begin{equation*}
\nabla = d + \begin{pmatrix} 0 & t^{-r} \\ t^{-r+1} & 0 \end{pmatrix}dt,
\end{equation*}
which has a nilpotent leading term.  Moreover, it is not even locally
gauge-equivalent to a connection with regular semisimple leading term
unless one passes to a ramified cover.  A particularly important
example is described by Frenkel and Gross in \cite{FGr}.  They
construct a flat $G$-bundle on $\proj^1$, for arbitrary reductive $G$,
that corresponds to a `small' supercuspidal representation of $G$ at
$\infty$ and the Steinberg representation at $0$.  In the $\GL_n$
case, the result is a connection (due originally to Katz~\cite{Katz})
with a regular singular point at $0$ (with unipotent monodromy) and an
irregular singular point with nilpotent leading term at $\infty$.
This construction suggests that connections with singularities
corresponding to cuspidal representations of $G$, an important case in
the geometric Langlands correspondence, do not have regular semisimple
leading term in the sense above.  These examples lead us to one of our
main questions: is there a natural generalization of the notion of a
regular semisimple leading term which allows us to extend the results
of Boalch, Jimbo, Miwa and Ueno?

The solution to this problem is again motivated by analogy with 
the classical Langlands correspondence.
Suppose  that $F$ is a local field and $W$ is a ramified representation
of $\GL_n(F)$.    Let $ P \subset \GL_n (F)$ be a parahoric subgroup with a decreasing filtration 
by congruence subgroups $\{P^j\}$, and suppose that
$\beta$ is an irreducible representation of $P^r$ on which $P^{r+1}$
acts trivially. We say that $W$ contains the \emph{stratum} $(P, r, \beta)$ 
if the restriction of $W$ to $P^r$ has a subrepresentation isomorphic to $\beta$.

In the language of Bushnell and Kutzko \cite{Bu, BK, Ku}, the data
$(P,r, \beta)$ is known as a \emph{fundamental stratum} if $\beta$
satisfies a certain non-degeneracy condition (see
Section~\ref{subsec:strata}).  If we write $e_P$ for the period of the
lattice chain stabilized by $P$, an equivalent condition is that $(P,
r, \beta)$ attains the minimal value $r/e_P$ over all strata contained
in $W$ (\cite[Theorem~1]{Bu}).  It was proved independently by Howe
and Moy~\cite{HM} and Bushnell~\cite{Bu} that every irreducible
admissible representation of $\GL_n(F)$ contains a fundamental
stratum.  It was further shown in \cite{BK} and \cite{Ku} that
fundamental strata play an important role in the classification of
supercuspidal representations, especially in the case of wild
ramification.

As a tool in representation theory, fundamental strata play much the
same role as the leading term of a connection in the cases considered
above.  Therefore, we are interested in finding an analogue of the
theory of strata in the context of meromorphic connections in order to
study moduli spaces of connections with cuspidal type singularities.

In this paper, we develop a geometric theory of strata and apply it to
the study of meromorphic connections.  We introduce a class of strata
called \emph{regular strata} which are particularly well-behaved:
connections containing a regular stratum have similar behavior to
connections with regular semisimple leading term.  More precisely, a
regular stratum associated to a formal meromorphic connection allows
one to ``diagonalize'' the connection so that it has coefficients in
the Cartan subalgebra of a maximal torus $T$.  We call the
diagonalized form of the connection a $T$-formal type.  In
Section~\ref{formaliso}, we show that two formal connections that
contain regular strata are isomorphic if and only if their formal
types lie in the same orbit of the relative affine Weyl group of $T$.

The perspective afforded by a geometric theory of strata has a number
of benefits.
\begin{enumerate}

\item 
The description of formal connections obtained in terms of
  fundamental strata translates well to global connections on curves;
  one reason for this is that, unlike the standard  local classification theorem \cite[Theorem
  III.1.2]{Mal}, one does not need to pass to a ramified cover.  
  In the second half of the paper, we use regular strata to explicitly
  construct moduli spaces of irregular connections on $\proj^1$ with a
  fixed formal type at each singular point.
In particular, we obtain a concrete description of the moduli
space of connections with singularities of ``supercuspidal'' formal type.
\item Fundamental strata provide an illustration of the wild ramification case of the geometric
Langlands correspondence; specifically, in the Bushnell-Kutzko theory
\cite[Theorem 7.3.9]{BK}, refinements of 
fundamental strata correspond to induction data for admissible
irreducible representations of $\GL_n$.
\item The analysis of the irregular Riemann-Hilbert map due to Jimbo,
  Miwa, Ueno, and Boalch \cite{JMU, Boa} generalizes to a much larger
  class of connections.  Specifically, one can concretely describe the
  isomonodromy equations for families of connections that contain
  regular strata \cite{BrSa2}.

\item Since the approach is purely Lie-theoretic, it can be adapted to
  study flat $G$-bundles (where $G$ is an arbitrary reductive group)
  using the Moy-Prasad theory of minimal $K$-types \cite{MoPr}.
\end{enumerate}

Here is a brief outline of our results.  In Section~\ref{sec:strata},
we adapt the classical theory of fundamental strata to the geometric
setting.  Next, in Section~\ref{regstrata}, we introduce the notion of
regular strata; these are strata that are centralized (in a graded
sense defined below) by a possibly non-split maximal torus $T$.  The
major result of this section is Theorem~\ref{thm1}, which states that
regular strata split into blocks corresponding to the minimal Levi
subgroup $L$ containing $T$.

In Section~\ref{strataconnections}, we show how to associate a stratum
to a formal connection with coefficients in a Laurent series field $F$
or, equivalently, to a flat $\GL_n$-bundle over the formal punctured
disk $\Spec(F)$.  By Theorem~\ref{fundcontain}, every formal
connection $(V, \nabla)$ contains a fundamental stratum $(P, r,
\beta)$, and the quantity $r/e_P$ for any fundamental stratum
contained in $(V, \nabla)$ is precisely the slope.  Moreover,
Theorem~\ref{thm2} states that any splitting of $(P, r, \beta)$
induces a splitting of $(V, \nabla)$.  In particular, any connection
containing a regular stratum has a reduction of structure to the Levi
subgroup $L$ defined above.  In Section~\ref{formaltypes}, we show
that the matrix of any connection containing a regular stratum is
gauge-equivalent to a matrix in $\ft=\Lie(T)$, which we call a formal
type.  We show in Section~\ref{formaliso} that the set of formal types
associated to an isomorphism class of formal connections corresponds
to an orbit of the relative affine Weyl group in the space of formal
types.

In Section~\ref{modspace}, we construct a moduli space $\mathscr{M}$
of meromorphic connections on $\proj^1$ with specified formal type at a collection
of singular points as the symplectic reduction of a product of smooth
varieties that only depend on local data.  By Theorem~\ref{modthm},
there is a symplectic manifold $\widetilde{\mathscr{M}}$ which
resolves $\mathscr{M}$ by symplectic reduction.  Finally,
Theorem~\ref{thm3} relaxes the regularity condition on strata at
regular singular points so that it is possible to consider connections
with unipotent monodromy.  In particular, this construction contains
the $\GL_n$ case of the flat $G$-bundle described by Frenkel and
Gross.

\section{Strata}\label{sec:strata}

In this section, we describe an abstract theory of fundamental strata
for vector spaces over a Laurent series field in characteristic zero.
Strata were originally developed to classify cuspidal representations
of $\GL_n$ over non-Archimedean local fields~\cite{Bu, BK, Ku}.  We
will show that there is an analogous geometric theory with
applications to the study of flat connections with coefficients in
$F$.  In Section~\ref{regstrata}, we introduce a novel class of
fundamental strata of ``regular uniform'' type.  These strata will
play an important role in describing the moduli space of connections
constructed in section \ref{modspace}.

\subsection{Lattice Chains and the Affine Flag Variety}
Let $k$ be an algebraically closed field of characteristic zero.
Here, $\mathfrak{o} = k[[t]]$ is the ring of formal power series in a
variable $t$, $\mathfrak{p}=t\mathfrak{o}$ is the maximal ideal in
$\mathfrak{o}$, and $F = k((t))$ is the field of fractions.

Suppose that $V$ is an $n$-dimensional vector space over $F$.  An
$\mathfrak{o}$-lattice $L \subset V$ is a finitely generated
$\mathfrak{o}$-module with the property that $L \otimes_\mathfrak{o} F
\cong V$.  If we twist $L$ by powers of $t$,
\begin{equation*}\label{twist}
L(m) = t^{-m} L,
\end{equation*}
then every $L(m)$ is an $\mathfrak{o}$-lattice as well.
\begin{definition}
A \emph{lattice chain} $\mathscr{L}$ is a collection of lattices
$(L^i)_{i\in\Z}$
with the following properties:
\begin{enumerate}
\item $L^i \supsetneq L^{i+1}$; and
\item $ L^i (m) = L^{i-m e}$ for some fixed $e = e_{\mathscr{L}}$.
\end{enumerate}
\end{definition}

Notice that a shift in indexing $(L[j])^i := L^{i+j} $ produces a
(trivially) different lattice chain $\mathscr{L}[j]$.  The lattice
chain $\mathscr{L}$ is called \emph{complete} if $e =n$; equivalently,
$L^i /L^{i +1}$ is a one dimensional $k$-vector space for all $i$.

\begin{definition}\label{parahoric}
  A parahoric subgroup $P \subset \GL(V)$ is the stabilizer of a
  lattice chain $\mathscr{L}$, i.e., $P=\{g \in \GL(V)\mid g
  L^i = L^i \text{ for all } i\}$. The Lie algebra of $P$ is the
  parahoric subalgebra $\mathfrak{P} \subset \gl(V)$ consisting of
  $\mathfrak{P} = \{p \in \gl(V) \mid p L^i \subset L^i \text{ for all
  } i\}$.  Note that $\mathfrak{P}$ is in fact an associative
  subalgebra of $\gl(V)$.  An Iwahori subgroup $I$ is the stabilizer
  of a complete lattice chain, and an Iwahori subalgebra
  $\mathfrak{I}$ is the Lie algebra of $I$.
\end{definition}

There are natural filtrations on $P$ (resp. $\mathfrak{P}$) by
congruence subgroups (resp. ideals).  For $r\in \Z$, define the
$\mathfrak{P}$-module $\mathfrak{P}^r$ to consist of $X \in
\mathfrak{P}$ such that $X L^i \subset L^{i+r}$ for all $i$; it is an
ideal of $\mathfrak{P}$ for $r\ge 0$ and a fractional ideal otherwise.
The congruence subgroup $P^r \subset P$ is then defined by $P^0 = P$
and $P^r = \id_n + \mathfrak{P}^r$ for $r > 0$.
Define $e_P = e_{\mathscr{L}}$; then, $t\mathfrak{P} =
\mathfrak{P}^{e_P}$.  Finally, $P$ is \emph{uniform} if $\dim L^i /
L^{i+1} = n / e$ for all $i$.  In particular, an Iwahori subgroup $I$
is always uniform.
 \begin{proposition}[{\cite[Proposition 1.18]{Bu}}]\label{uniformprop}
   The Jacobson radical of the parahoric subalgebra $\mathfrak{P}$ is
   $\mathfrak{P}^1$.  Moreover, when $P$ is uniform, there exists an
   element $\varpi_{P} \in \mathfrak{P}$ such that $\varpi_P
   \mathfrak{P} = \mathfrak{P} \varpi_P = \mathfrak{P}^1$.
 \end{proposition}

As an example,
 suppose that $V = V_k \otimes_k F$ for a given $k$-vector space
 $V_k$.  There is a distinguished lattice $V_\mathfrak{o} = V_k
 \otimes_k \mathfrak{o}$, and an evaluation map
\begin{equation*}
\rho : V_\mathfrak{o} \to V_{k}
\end{equation*}
obtained by setting $t = 0$.  Any subspace $W \subset V_k$ determines
a lattice $\rho^{-1} (W) \subset V$.  Thus, if $\mathscr{F} = (V_k =
V^0 \supset V^1 \supset \ldots \supset V^e = \{ 0\})$ is a flag in
$V_k$, then $\mathscr{F}$ determines a lattice chain by
\begin{equation*}
\mathscr{L}_{\mathscr{F}} =  
( \ldots  \supset t^{-1}\rho^{-1} (V^{n-1}) \supset V_{\mathfrak{o}} \supset 
\rho^{-1} (V^1) \supset \ldots
\supset \rho^{-1}(V^{n-1}) \supset tV_{\mathfrak{o}} \supset \ldots).
\end{equation*}
We call such lattice chains (and their associated parahorics) {\em
  standard}.  Thus, if $\mathscr{F}_0$ is the complete flag determined
by a choice of Borel subgroup $B$, then $\rho^{-1}(B)$ is the standard
Iwahori subgroup which is the stabilizer of
$\mathscr{L}_{\mathscr{F}_0}$.  Similarly, the partial flag in $V_k$
determined by a parabolic subgroup $Q$ gives rise to a standard
parahoric subgroup $\rho^{-1}(Q)$ which is the stabilizer of the
corresponding standard lattice chain.  In particular, the maximal
parahoric subgroup $\GL_n(\mathfrak{o})$ is the stabilizer of the
standard lattice chain associated to $(V_k \supset \{0\}).$

In this situation, the obvious $\GL_n(F)$-action acts transitively on
the space of complete lattice chains, so we may identify this space
with the affine flag variety $\GL_n(F) / I$, where $I$ is a standard
Iwahori subgroup.  More generally, every lattice chain is an element
of a partial affine flag variety $\GL_n(F) / P$ for some standard
parahoric $P$.  For more details on the relationship between
affine flag varieties and lattice chains in general, see \cite{Sa00}.

For any maximal subfield $E\subset\gl(V)$, there is a unique Iwahori
subgroup $I_E$ such that $\fo^\times_E\subset I_E$.

\begin{lemma}\label{219}
Suppose that $P$ is a parahoric subgroup of $\GL(V)$ that
stabilizes a lattice chain $\mathscr{L}$.  Let $E/F$ be a degree
$n=\dim V$ field extension with a fixed embedding in $\gl(V)$ 
such that $\mathfrak{o}_E^\times \subset P$.
Then, there exists a complete lattice chain $\mathscr{L}_E \supset
\mathscr{L}$ with stabilizer $I_E$ 
such that $\mathfrak{o}_E^\times \subset I_E$ and $I_E \subset P$; it is
unique up to translation of the indexing.  In particular,
$\fo_E^\times$ is contained in a unique Iwahori subgroup.
\end{lemma}
\begin{proof}
  We may identify $V$ with $E$ as an $E$-module.  Since
  $\mathfrak{o}_E^\times \subset P$, it follows that $\mathfrak{o}_E
  \subset \mathfrak{P}$.  Therefore, we may view $\mathscr{L}$ as a
  filtration of $E$ by nonzero $\mathfrak{o}_E$-fractional ideals.
  Since $\fo_E$ is a discrete valuation ring, there is a maximal
  saturation $\mathscr{L}_E$ of $\mathscr{L}$, unique up to indexing,
  consisting of all the nonzero fractional ideals, and it is clear
  that $\mathfrak{o}_E^\times \subset I_E \subset P$.  The final
  statement follows by taking $P$ to be the stabilizer of the lattice
  $\fo_E$.
 
\end{proof}

\subsection{Duality}\label{sec:dual}
Let $\Omega^1_{F/k}$ be the space of one-forms on $F$,
and let $\Omega^\times \subset \Omega^1_{F/k}$ be the $F^\times$-torsor
of non-zero one-forms.
If $\nu \in  \mathfrak{o}^\times \frac{dt}{t^\ell} \subset \Omega^\times$,
its {\em order} is defined by $\ord(\nu) = - \ell$.
Any $\nu \in \Omega^\times$ defines
a nondegenerate invariant  symmetric $k$-bilinear form $\langle , \rangle_\nu$ on $\gl_n(F)$ by
\begin{equation*}
\langle A, B \rangle_\nu = \Res \left[ \Tr (A B) \nu \right],
\end{equation*}
where $\Res$ is the usual residue on differential forms.  In most
contexts, one can take $\nu$ to be $\frac{dt}{t}$.

Let $\mathfrak{P}$ be the parahoric subalgebra that preserves a
lattice chain $\mathscr{L}$.

\begin{proposition}[Duality]\label{duality}
Fix $\nu \in \Omega^\times$.  Then,
\begin{equation*}\label{dual1}
(\mathfrak{P}^{s})^\perp = \mathfrak{P}^{1-s - (1+  \ord(\nu)) e_P},
\end{equation*}
and, if $r\le s$,
\begin{equation*}\label{dual2}
  (\mathfrak{P}^r / \mathfrak{P}^s)^\vee \cong \mathfrak{P}^{1-s - (1+  \ord(\nu)) e_P}/\mathfrak{P}^{1-r - (1+  \ord(\nu)) e_P};
\end{equation*} 
here, the superscript ${}^\vee$ denotes the $k$-linear dual.
\end{proposition}
This is shown in Proposition 1.11 and Corollary 1.13 of \cite{Bu}.  In
particular, when $\ord(\nu) = -1$, $(\mathfrak{P}^{s})^\perp =
\mathfrak{P}^{1-s}$ and $(\mathfrak{P}^r/\mathfrak{P}^{r+1})^\vee =
\mathfrak{P}^{-r}/\mathfrak{P}^{-r+1}$. 

Observe that any element of $\fP^r$ induces an endomorphism of the
associated graded $\fo$-module $\gr(\mathscr{L})$ of degree $r$;
moreover, two such elements induce the same endomorphism of
$\gr(\mathscr{L})$ if and only if they have the same image in
$\fP^r/\fP^{r+1}$.

The following lemma gives a more precise description of the quotients
$\mathfrak{P}^r/\mathfrak{P}^{r+1}$.  Let $\bar{G}=\GL(L^0/tL^0)\cong
\GL_n(k)$ with $\bar{\mathfrak{g}}$ the corresponding Lie algebra.
Note that there is a natural map from $P\to\bar{G}$ whose image is a
parabolic subgroup $Q$; its unipotent radical $U$ is the image of
$P^1$.  Analogous statements hold for the Lie algebras
$\Lie(Q)=\mathfrak{q}$ and $\Lie(U)=\mathfrak{u}$.

\begin{lemma}\label{subquotient} \mbox{}\begin{enumerate} \item There is a
  canonical isomorphism of  $\mathfrak{o}$-modules
\begin{equation*} 
\mathfrak{P}^r/\mathfrak{P}^{r+1} \cong 
\bigoplus_{i = 0}^{e_P-1} \Hom (L^i / L^{i+1}, L^{i+r}/L^{i+r+1}).
\end{equation*}
\item In the case $r=0$, this isomorphism gives an algebra isomorphism
  between $\mathfrak{P}/\mathfrak{P}^1$ and a Levi subalgebra
  $\mathfrak{h}$ for $Lie(Q)=\mathfrak{q}$ (defined up to conjugacy by
  $U$).  Moreover, $\mathfrak{P}$
  is a split extension of $\mathfrak{h}$ by $\mathfrak{P}^1$.
\item Similarly, if $H$ is a Levi subgroup for $Q$, then $P \cong
  H \ltimes P^1$.
\end{enumerate}
\end{lemma}
\begin{proof}
  There is a natural $\mathfrak{o}$-module map $\mathfrak{P}^r \to \bigoplus_{i =0}^{e_P-1}
  \Hom (L^i / L^{i+1}, L^{i+r}/L^{i+r+1})$; it is an algebra
  homomorphism when $r=0$.  It is clear that
  $\mathfrak{P}^{r+1}$ is the kernel, since any $\mathfrak{o}$-module
  map that takes $L^i$ to $L^{i+r+1}$ for $0 \le i \le e_P-1$ must lie
  in $\mathfrak{P}^{r+1}$.

  Now, suppose that $(\phi_i) \in \bigoplus_{i = 0}^{e_P-1} \Hom (L^i
  / L^{i+1}, L^{i+r}/L^{i+r+1})$.  Let $\mathscr{F}$ be the partial
  flag in $L^0/tL^0=L^0/L^{e_P}$ given by $\{L^i/L^{e_P}\mid 0\le i\le
  e_P\}$.  We may choose an ordered basis $\bfe$ for $L^0$ that is
  compatible with $\mathscr{F}$ modulo ${L^{e_P}}$.  This means that
  there is a partition $\bfe=\bfe_0\cup\dots\cup\bfe_{e_P-1}$ such
  that $W_j=\spa(\bfe_j)\subset L^j$ is naturally isomorphic to
  $L^j/L^{j+1}$.  In this basis, the groups $\Hom (L^i/L^{i+1},
  L^{i+r}/L^{i+r+1})$ appear as disjoint blocks in $\mathfrak{P}^r$
  (with exactly one block in each row and column of the array of
  blocks), so it is clear that we can construct a lift $\tilde{\phi}
  \in \mathfrak{P}^r$ that maps to $(\phi_i)$.

  Note that when $r=0$, the image of this isomorphism is a Levi
  subalgebra $\mathfrak{h}$ for the parabolic subalgebra
  $\mathfrak{q}$.  The choice of basis gives an explicit embedding
  $\mathfrak{h}\cong\gl(W_1)\oplus\dots\oplus\gl(W_{e_P-1})\subset\GL(V)$,
  so the extension is split.  The proof in the group case is similar.
\end{proof}

\begin{rmk}
  
  The same proof gives an isomorphism
  \begin{equation}\label{priso}\mathfrak{P}^r/\mathfrak{P}^{r+1} \cong \bigoplus_{i = m}^{m+e_P-1}
    \Hom (L^i / L^{i+1}, L^{i+r}/L^{i+r+1})
\end{equation}
for any $m$.  However, if $\Hom (L^i / L^{i+1}, L^{i+r}/L^{i+r+1})$
and $\Hom (L^j / L^{j+1}, L^{j+r}/L^{j+r+1})$ for $i\equiv j\mod e_P$
are identified via homothety, the image of an element of
$\mathfrak{P}^r/\mathfrak{P}^{r+1}$ is independent of $m$ up to cyclic
permutation.  Indeed, this follows immediately from the observation
that if $m=s e_P+j$ for $0\le j< e_P$, then
$t^{s+1}\bfe_0\cup\dots\cup t^{s+1}\bfe_{j-1}\cup t^s\bfe_j\cup
t^s\bfe_{e_P-1}$ is a basis for $L^m$.  In particular,
$\mathfrak{P}/\mathfrak{P}^{1}$ is isomorphic to a Levi subalgebra in
$\gl(L^m/tL^m)$

\end{rmk}

\begin{rmk} \label{levi} Any element
  $\bar{x}\in\mathfrak{P}/\mathfrak{P}^{1}$ determines a canonical
  $GL(L^0/tL^0)$-orbit in $\gl(L^0/tL^0)$, and similarly for $P/P^1$.
  To see this, note that any choice of ordered basis for $L^0$
  compatible with $\mathscr{L}$ maps $\bar{x}$ onto an element of a
  Levi subalgebra of $\gl(L^0/tL^0)$; a different choice of compatible
  basis will conjugate this image by an element of $Q$.  In fact, this
  orbit is also independent of the choice of base point $L^0$ in the
  lattice chain.  Indeed, an ordered basis for $L^0$ compatible with
  $\mathscr{L}$ gives a compatible ordered basis for $L^m$ by
  multiplying basis elements by appropriate powers of $t$ and then
  permuting cyclically.  Using the corresponding isomorphism $L^0\to
  L^m$ to identify $\gl(L^0/tL^0)$ and $\gl(L^m/tL^m)$, the images of
  $\bar{x}$ are the same.  Accordingly, it makes sense to talk about
  the characteristic polynomial or eigenvalues of $\bar{x}$.

\end{rmk}

\subsubsection*{Notational Conventions}

Let $\nu \in \Omega^\times$, and let $V$ be an $F$-vector space.
Suppose that $P \subset \GL(V)$ is a parahoric subgroup with Lie
algebra $\mathfrak{P}$ that stabilizes a lattice chain $(L^i)_{i \in \Z}$.  
We will use the following conventions
throughout the paper:
\begin{enumerate}
\item $\bar{L}^i = L^i/L^{i+1}$.
\item $\bar{P} = P/P^1$, $\bar{\mathfrak{P}} = \mathfrak{P}/\mathfrak{P}^1$.
\item $\bar{P}^\ell = P^\ell / P^{\ell+1}$, $\bar{\mathfrak{P}}^\ell =
\mathfrak{P}^\ell/
\mathfrak{P}^{\ell+1}$.
\item If $X\in\fP^s$, then $\bar{X}$ will denote its image in $\bfP^s$
  and the corresponding degree $s$ endomorphism of $\gr(\mathscr{L})$.
\item If $\a\in(\fP^r)^\vee$, then $\an\in\gl(V)$ with denote an
  element such that $\a=\langle \an,\cdot \rangle_\nu$.
\item If $\beta \in (\bar{\mathfrak{P}}^r)^{\vee}$, then $\beta_\nu$
  will denote an element of $\mathfrak{P}^{r - (1+\ord(\nu)) e_P}$
  such that $\bar{\beta}_\nu \in \bar{\mathfrak{P}}^{r - (1+\ord(\nu))
    e_P}$ is the coset determined by the isomorphism in
  Proposition~\ref{duality}. 
\item Let $X \in \mathfrak{P}^{s}$.  Then, 
$\delta_X : \bar{\mathfrak{P}}^{i} \to \bar{\mathfrak{P}}^{i+s}$
is the map induced by $\ad (X)$.
\item Let $\mathfrak{a} \subset \End (V)$ be a subalgebra.  We define
  $\mathfrak{a}^i=\mathfrak{a} \cap \mathfrak{P}^i$ and 
$\bar{\mathfrak{a}}^i = \mathfrak{a}^i /
\mathfrak{a}^{i+1}.$ 
\item If $A \subset \GL(V)$ is a subgroup, define $A^i=A\cap P^i$ and 
$\bar{A}^i = A^i / A^{i+1}$.
\end{enumerate}

\subsection{Tame Corestriction}\label{sec:cores}

In this section, we first suppose that $\mathscr{L}$ is a complete
(hence uniform) lattice chain in $V$ with corresponding Iwahori
subgroup $I$.  By Proposition \ref{uniformprop}, $\mathfrak{I}^1$ is a
principal ideal generated by $\varpi_I$; similarly, each fractional
ideal $\mathfrak{I}^\ell$ is generated by $\varpi_I^\ell$.  Choose an
ordered basis $(e_0, \ldots, e_{n-1})$ for $V$ indexed by $\Z_n$, so
that $e_{n+i} = e_i$ for all $i$.  Furthermore, we may choose the
basis to be compatible with $\mathscr{L}$: if $r = q n -s$, $0 \le s
<n$, then $L^r$ is spanned by $\{t^{q-1} e_0, \ldots, t^{q-1}
e_{s-1},t^{q} e_{s}, \ldots, t^{q} e_{n-1}\}$.  Thus, if we let
$\bar{e}_i$ denote the image of $e_i$ in $L^0/t L^0$, then
$\mathscr{L}$ corresponds to the full flag in $L^0/t L^0$ determined
by the ordered basis $(\bar{e}_0, \ldots \bar{e}_{n-1})$.

In this basis, we may take
\begin{equation}\label{varpip}
\varpi_I= 
\begin{pmatrix}
0 & 1 & \cdots & 0\\
\vdots & \ddots & \ddots & \vdots \\
0 & \ddots & 0 & 1  \\
t  & 0 & \cdots& 0 
\end{pmatrix}.
\end{equation}
Notice that the characteristic polynomial of $\varpi_I$ is equal to
$\lambda^n - t$, which is irreducible over $F$. Thus, $F[\varpi_I]$ is
a degree $n$ field extension isomorphic to $F[t^{1/n}]$.

\begin{rmk}\label{compatiblebasis}  If $P$ is a parahoric subgroup stabilizing the lattice
  chain $\sL$, we say that a basis for $L^0$ is compatible with
  $\sL$ if it is a basis that is compatible as above for any complete lattice
  chain extending $\sL$.  Note that any pullback of a compatible basis for the induced
  partial flag in $L^0/tL^0$ is such a basis.
\end{rmk}

We first examine the kernel and image of the map 
$\delta_{\varpi_I^r}:\bfI^l\to\bfI^{l-r}$.  Let $\nu \in
\Omega^\times$ have order $-1$.  We define $\psi_\ell (X) = \langle X,
  \varpi_{I}^{-\ell}\rangle_\nu$.  Note that $\psi_\ell
(\mathfrak{I}^{\ell+1}) = 0$; we let $\bar{\psi}_\ell$ be the induced
functional on $\bar{\mathfrak{I}}^\ell$.

Let $\mathfrak{d}\subset\gl_n(k)$ be the subalgebra of diagonal
matrices.  By Lemma~\ref{subquotient}, the Iwahori subgroup and
subalgebra have semidirect product decompositions:
$I=\mathfrak{d}^*\rtimes {I}^1$ and $\mathfrak{I}$
is a split extension of $\mathfrak{d}$ by $\mathfrak{I}^1$.
Accordingly, any coset in $\mathfrak{I}^\ell/\mathfrak{I}^{\ell+1}$
has a unique representative $x\varpi_{I}^\ell$ with $x=\diag(x_0, x_1,
\dots, x_{n-1})\in\mathfrak{d}$.

\begin{lemma}\label{adequation} 
  The image of $\delta_{\varpi_{I}^{-r}}$ in
  $\bar{\mathfrak{I}}^{\ell-r}$ is contained in
  $\ker(\bar\psi_{\ell-r})$, and the kernel of
  $\delta_{\varpi_{I}^{-r}}$ in
  $\bar{\mathfrak{I}}^{\ell}$ contains the
  one-dimensional subspace spanned by $\ov{\varpi_{I}^\ell}$.
  Equality in both cases happens if and only if $\gcd(r, n) =1$.
\end{lemma}

\begin{proof}

Take $X=x\varpi_{I}^\ell$ with $x\in\mathfrak{d}$ as above. By direct
calculation, 
\begin{equation}\label{directcalc} [X,\varpi_{I}^{-r}] = x'
  \varpi_{I}^{\ell-r},\text{ with }
x' = \diag(x_0 - x_{-r}, \dots, x_{n-1} - x_{n-1-r} ).
\end{equation}

Therefore, 
\begin{equation*}\begin{aligned}
\psi_{\ell-r} (x' \varpi_{I}^{\ell-r}) & = 
\Res (\Tr (x' \varpi_{I}^{\ell -r} \varpi^{r-\ell}_I) \nu)\\
&= \Res(\nu) \Tr(x')  \\
& = \Res(\nu) \sum_{i = 0}^{n-1} (x_i - x_{i-r}) = 0.
\end{aligned}
\end{equation*}
It follows that $\delta_{\varpi_I^{-r}}
(\bar{\mathfrak{I}}^{\ell}) \subset \ker
(\bar{\psi}_{\ell-r}).$

The kernel of $\delta_{\varpi_{I}^{-r}}$ satisfies the equations
$x_i - x_{i-r} = 0$ for $0 \le i \le n-1$.  If we set $x_0 = \alpha$,
then $x_{-r} = x_{-2 r} = x_{-3 r} = \dots = \alpha$.  When $\gcd(r,n)
= 1$, $j\equiv -mr\pmod n$ is solvable for any $j$, and it follows that
$x_j=\alpha$ for all $j$.  Therefore, the kernel is just the span of
$\ov{\varpi_{I}^\ell}$.  Otherwise, the dimension of the kernel is at
least $2$.  This implies that the image of
$\delta_{\varpi_{I}^{-r}}$ has codimension $1$ if and only if
$\gcd(r, n) = 1$.
\end{proof}
For future reference, we remark that there is a similar formula to
\eqref{directcalc} for $\Ad$.  Any element of $\bar{I}\cong
\mathfrak{d}^*$ is of the form $\bar{p}$ for $p = \diag(p_0, \dots,
p_{n-1})$.

Then,
\begin{equation}\label{Adcalc}
\Ad (\bar{p}) (\ov{\varpi_I^{-r}})  = \ov{p' \varpi_I^{-r}} \in
\bar{\mathfrak{I}}^{-r}, \text { where }
p'  = \diag(\frac{p_0}{p_{-r}}, \frac{p_1}{p_{1-r}}, \ldots, \frac{p_{n-1}}{p_{n-1-r}}).
\end{equation}
In particular, when $\gcd(r, n) = 1$, every generator of
$\mathfrak{I}^{-r}$ lies modulo $\mathfrak{I}^{-r+1}$ in the
$\Ad(I)$-orbit of $a\varpi_I^{-r}$ for some $a\in
k^*$.

Next, we consider more general uniform parahorics.  Let $E/F$ be a
degree $m$ extension; it is unique up to isomorphism.  Now, identify
$V \cong E^{n/m}$ as an $F$ vector space.  We will view $E$ as a
maximal subfield of $\gl_m(F)$: if we define $\varpi_E = \varpi_I
\in\gl_m(F)$ as in \eqref{varpip}, then $E$ is the centralizer of
$\varpi_E$, which is in fact a uniformizing parameter for $E$.  Since
$m|n$, define a Cartan subalgebra $\mathfrak{t} \cong E^{n/m}$ in
$\gl(V)$ as the block diagonal embedding of $n/m$ copies of $E \subset
\gl_m(F)$.  Let $\gcd(r, m) = 1$, and take $\xi = (a_1 \varpi_E^{-r},
\ldots, a_{n/m} \varpi_E^{-r})$ with the $a_i$'s pairwise unequal
elements of $k$.  This implies that $\xi$ is regular semisimple with
centralizer $\mathfrak{t}$.

Let $\mathscr{L}_E$ be the complete lattice chain in $F^m$ stabilized
by $\mathfrak{o}_E$; we let $I_E$ be the corresponding Iwahori
subgroup.  We define a lattice chain $\mathscr{L} =
\bigoplus^{n/m}_{i=1} \mathscr{L}_E$ in $V$ with associated parahoric
subgroup $P$.  It is clear that $P$ is uniform with $e_P = m$.
Moreover, elements of $\mathfrak{P}^{\ell}$ are precisely those
$n/m\times n/m$ arrays of $m\times m$ blocks with entries in
$\mathfrak{I}^{\ell}_E$; in particular, we can take
$\varpi_P=(\varpi_E,\dots,\varpi_E)$.  Note that $\mathfrak{t} \cap
\gl(L) \subset \mathfrak{P}$ for any lattice $L$ in
$\mathscr{L}$.
\begin{proposition}[Tame Corestriction]\label{cores}
There is a morphism of $\mathfrak{t}$-bimodules  $\pi_{\mathfrak{t}} : \gl(V) \to \mathfrak{t}$ 
satisfying the following properties:
\begin{enumerate}[label={\textnormal{(\arabic*)}},ref={\theproposition(\arabic*)}]
\item \label{cores1} $\pi_{\mathfrak{t}}$ restricts to the identity on $\mathfrak{t}$;
\item  \label{cores2} $\pi_{\mathfrak{t}} (\mathfrak{P}^\ell) \subset \mathfrak{P}^\ell$;
\item\label{cores3} the kernel of the induced map
\begin{equation*}
\bar{\pi}_{\mathfrak{t}} : (\mathfrak{t} + \mathfrak{P}^{\ell-r})/ \mathfrak{P}^{\ell-r+1} \to
\mathfrak{t}/ (\mathfrak{t} \cap \mathfrak{P}^{\ell-r+1})
\end{equation*}
is given by the image of $\ad(\mathfrak{P}^\ell) (\xi)$ modulo $\mathfrak{P}^{\ell-r+1}$;
\item \label{cores4} if $z \in \mathfrak{t}$ and $X \in \gl(V)$, then
$\langle z, X\rangle_\nu = \langle z, \pi_{\mathfrak{t}}
(X)\rangle_\nu$;
\item \label{cores5} $\pi_\ft$ commutes with the action of the normalizer $N(T)$ of $T$.
\end{enumerate}
\end{proposition}

\begin{proof}
  First, take $\nu = \frac{dt}{t}$.  Let $\epsilon_i \in \gl(V)$ be the
  identity element in the $i^{th}$ copy of $E$ in $\mathfrak{t}$ and
  $0$ elsewhere.  Define $\psi_s^i (X) =
  \frac{1}{m}\langle \varpi_E^{-s} \epsilon_i, X \rangle_{\nu}$ and
\begin{equation*}
\pi_{\mathfrak{t}} (X) = 
\sum_{s = -\infty}^\infty  \sum_{i = 1}^{n/m} \psi_s^i (X) \varpi_E^s \epsilon_i.
\end{equation*}
It is easily checked that for $s \ll 0$, $\psi_s^i (X) = 0$ and that
$\pi_\mathfrak{t}$ is a $\mathfrak{t}$-map.  A direct calculation
shows that for $n\in N(T)$, $\pi_\ft(\Ad(n)X)=\Ad(n)\pi_\ft(X)$.  Since
$\pi_{\mathfrak{t}}$ is defined using traces, it is immediate that it
vanishes on the off-diagonal blocks $\epsilon_i \gl(V) \epsilon_j$
for $i\ne j$.  Moreover,
$\pi_{\mathfrak{t}}$ is the identity on $\mathfrak{t}$,
since $\psi_s^i (\varpi_E^j \epsilon_k) = 1$, if $j = s$ and $i = k$,
and equals $0$ otherwise.

We note that $\pi_\mathfrak{t}(\mathfrak{P}^\ell)\subset
\mathfrak{P}^\ell \cap \mathfrak{t}$, so the induced map
$\bar{\pi}_{\mathfrak{t}}$ makes sense.  Let $V_{i j}^{\ell-r} =
\epsilon_i \mathfrak{P}^{\ell-r} \epsilon_j$ so that $\bar{V}_{i
  j}^{\ell-r} \subset \ker (\bar{\pi}_\mathfrak{t})$ for $i\ne j$.  By
regularity, $\delta_\xi : \bar{V}_{i j}^{\ell} \to \bar{V}_{i
  j}^{\ell-r}$ is an isomorphism whenever $i \ne j$.  This proves that
the off-diagonal part of $\ker(\bar{\pi}_\mathfrak{t})$ is of the
desired form.  We may now reduce without loss of generality to the case of
a single diagonal block, i.e., $\mathfrak{t} = E$ and $\mathfrak{P}=\mathfrak{I}$.

Since $ \pi_{\mathfrak{t}}$ is the identity on ${\mathfrak{t}}$, the
kernel of $\bar{\pi}_{\mathfrak{t}}$ is contained in
$\bar{\mathfrak{I}}^{\ell-r} = \varpi_I^{\ell-r} \bar{\mathfrak{I}}$.
Notice that when $s < \ell - r$, $\varpi_I^{-s} \varpi_I^{\ell-r}
\mathfrak{I} \subset \mathfrak{I}^1$; therefore, $\psi_s
(\mathfrak{I}^{\ell-r}) = 0$.  It is trivial that $\psi_s (X)
\varpi_E^s \in \mathfrak{I}^{\ell-r+1}$ for $s > \ell-r$.
It follows that $\ker(\bar{\pi}_{\mathfrak{t}}) =
\ker(\bar{\psi}_{\ell-r}).$ By Lemma \ref{adequation},
$\ad(\mathfrak{I}^\ell)(\varpi_{E}^{-r}) = \ker(
\bar{\psi}_{\ell-r})$.  This completes the proof of the third part of
the proposition.

Finally, for arbitrary $\nu' = f \nu$, $\langle z, X \rangle_{\nu'} =
\langle z, f X \rangle_{\nu}$.  Since $f \in F \subset T$, and
$\pi_{\mathfrak{t}}$ is a $\mathfrak{t}$-map, it suffices to prove the
fourth part when $\nu = \frac{dt}{t}$.  Although $z\in\mathfrak{t}$ is
an infinite sum of the form $\sum_{s\ge q}\sum_{i=1}^{n/m} a_{s
  i}\varpi_E^s\epsilon_i$ for some $a_{s i}\in k$, only a finite
number of terms contribute to the inner products.  Hence, it suffices
to consider $z=\varpi_E^s\epsilon_i$.  Observing that $\langle  \varpi_E^s \epsilon_i, \varpi_E^{-r}
  \epsilon_j \rangle_\nu = m \delta_{i j} \delta_{r s}$, we see
that
\begin{equation*}
\langle  \varpi_E^s \epsilon_i, X \rangle_\nu =  m \psi_{-s}^i  (X) = 
\langle  \varpi_E^s \epsilon_i, \psi_{-s}^i(X) \varpi_E^{-s}  \epsilon_i\rangle_\nu
= \langle  \varpi_E^s \epsilon_i, \pi_{\mathfrak{t}} (X) \rangle_\nu,
\end{equation*}
as desired.

\end{proof}

\begin{rmk}\label{unifvarpi}
  Suppose that $P \subset \GL(V)$ is a uniform parahoric that
  stabilizes a lattice chain $\mathscr{L}$.  Let $H \subset P$ be the
  Levi subgroup that splits $P \to P / P^1$ as in
  Lemma~\ref{subquotient}, and let $\mathfrak{h}\subset \fP$ be the
  corresponding subalgebra.  We will show that there is a generator
  $\varpi_P$ for $\fP^1$ that is well-behaved with respect to $H$,
  akin to $\varpi_{I} \in \fI^1$.  In the notation used in the proof
  of Lemma~\ref{subquotient}, $\mathfrak{h}$ is determined by an
  ordered basis $\bfe$ for $L^0$ partitioned into $e_P$ equal
  parts: $\bfe=\bigcup_{j=0}^{e_P-1}\bfe_j$.  Setting
  $W_j=\spa{\bfe_j}\cong\bL^j$, we have  $\mathfrak{h} = \bigoplus_{j = 0}^{e_P-1}
  \gl(W_j)$ and $H = \prod_{j = 0}^{e_P -1} \GL(W_j)$.

  Now, let $\mathscr{L}'$ be the complete lattice chain determined by
  the ordered basis for $L^0$ given above, and let $I$ be the
  corresponding Iwahori subgroup.  If $\varpi_I$ is the generator of
  $\fI^1$ constructed in \eqref{varpip}, define $\varpi_P =
  \varpi_I^{m}$, where $m=n/e_P$.  This matrix is an $e_P\times e_P$
  block matrix of the same form as \eqref{varpip}, but with scalar
  $m\times m$ blocks.  Evidently, $\varpi_P (L^i) = L^{i+1}$, so
  $\varpi_P$ generates $\mathfrak{P}^1$.  Furthermore, $\varpi_P (W_j)
  = W_{j+1}$ for $0 \le j < e_P-1$, and $\varpi_P (t^{-1}W_{e_P-1}) =
  W_0$.  It follows that $\varpi_P$ normalizes $H$, and
  $\Ad(\varpi_P)(\mathfrak{h}) \subset \mathfrak{h}$.  In fact, if
  $A=\diag(A_0,\dots,A_{e_P-1})\in\fh$, then
  $\Ad(\varpi_P^r)(A)=\diag(A_r,\dots,A_{r+e_P-1})$, with the indices
  understood modulo $e_P$.
  
\end{rmk}

\subsection{Strata}\label{subsec:strata}

For the remainder of Section~\ref{sec:strata}, $\nu\in\Omega^\times$
will be a fixed one-form of order $-1$. 

\begin{definition}\label{stratdef}
Let $V$ be an $F$ vector space.
A \emph{stratum} in $\GL(V)$ is a triple $(P, r, \beta)$ consisting of
\begin{itemize}
\item $P \subset \GL(V)$  a parahoric subgroup;
\item $r \in \Z_{\ge 0}$;
\item $\beta \in (\bar{\mathfrak{P}}^r)^\vee$.
\end{itemize}
\end{definition}

Proposition \ref{duality} states that $(\bar{\mathfrak{P}}^r)^\vee =
\bar{\mathfrak{P}}^{-r}$.  Therefore, we may choose a representative
$\beta_\nu \in \mathfrak{P}^{-r}$ for $\beta$.  Explicitly, a stratum
is determined by a triple $(\mathscr{L}, r, \beta_\nu)$, where
$\mathscr{L}$ is the lattice chain preserved by $P$, and $\beta_\nu$
is a degree $-r$ endomorphism of $\mathscr{L}$.  The triples
$(\mathscr{L}, r, \beta_\nu)$ and $(\mathscr{L'}, r', \beta'_\nu)$
give the stratum if and only if $r=r'$, $\mathscr{L'}$ is a translate
of $\mathscr{L}$, and $\beta_\nu$ and $\beta'_\nu$ induce the same
maps on $\gr(\mathscr{L})$, i.e., $\bar{\beta}_\nu=\bar\beta_\nu$.

We say that $(P, r, \beta)$ is \emph{fundamental} if 
 $\beta_\nu + \mathfrak{P}^{- r+1}$ contains no nilpotent
elements of $\gl_n(F)$.  
By \cite[Lemma 2.1]{Bu}, a stratum is non-fundamental if and only 
if $(\beta_\nu)^m \in \mathfrak{P}^{1 - r m}$ for some $m$. 
\begin{rmk}\label{asgrd}
  A stratum $(P, r, \beta)$ is fundamental if and only if
  $\bar{\beta}_\nu\in\End(\gr(\mathscr{L}))$ 
  is non-nilpotent in the usual sense.  In particular, if
  $\beta_\nu(L^i) = L^{i-r}$ for all $i\in\Z$, then $(P, r, \beta)$ is
  necessarily fundamental.
\end{rmk}

\begin{definition}\label{stratared}
  Let $(P, r, \beta)$ be a stratum in $\GL(V)$.  A reduction of $(P,
  r, \beta)$ is a $\GL(V)$-stratum $(P', r', \beta')$ with the
  following properties: $\left(\beta_\nu' + (\mathfrak{P}')^{1-r'}
  \right) \cap \left( \beta_\nu + \mathfrak{P}^{1-r}\right) \ne
  \emptyset $, $\beta_\nu + \mathfrak{P}^{1-r} \subset
  (\mathfrak{P'})^{-r'}$, and there exists a lattice $L$ that lies in
  both of the associated lattice chains $\mathscr{L}$ and
  $\mathscr{L}'$.
\end{definition}

Let $(P', r', \beta')$ be a reduction of $(P, r, \beta)$.  The first
property allows one to choose $\beta_\nu \in \gl_n(F)$ to represent
both $\beta$ and $\beta'$.  The second implies that any representative
$\beta_\nu$ for $\beta$ determines an element of
$(\bar{\mathfrak{P}}')^{r'}$.  Note that it is
possible to have two different reductions $(P', r', \beta'_1)$ and
$(P', r', \beta'_2)$ with the same $P'$ and $r'$, if
$\mathfrak{P}^{1-r} \nsubseteq (\mathfrak{P}')^{1-r'}$.

An important invariant of a stratum $(P, r, \beta)$ is its \emph{slope},
which is defined by
$\slope(P, r, \beta) = r/e_P$.  The following theorem
describes the relationship between slope and fundamental strata.

\begin{theorem}\label{theorem:Bu1}
Suppose that $r\ge 1$.  Then $(P, r, \beta)$ is a non-fundamental stratum if and only if
there is a reduction $(P', r', \beta')$ with
$\slope(P, r, \beta)<  \slope(P', r', \beta')$. 
\end{theorem}
This is proved in Theorem 1 and Remark 2.9 of \cite{Bu}.
\begin{definition}\label{uniform} 
  A stratum $(P, r, \beta)$ is called \emph{uniform} if it is
  fundamental, $P$ is a uniform parahoric subgroup, and $\gcd(r, e_P)
  = 1$.  The stratum is \emph{strongly uniform} if it is uniform and
  $\beta_\nu (L^i) = L^{i-r}$ for all $L^i \in \mathscr{L}$.
\end{definition}

\begin{rmk} \label{sug} A uniform stratum $(P, r, \beta)$ is strongly
  uniform if and only if the induced maps
  $\bar{\beta}_\nu:\bL^i\to\bL^{i-r}$ are isomorphisms for each $i$.
  The forward implication follows since the $\bL^i$'s have the same
  dimension.  For the converse, note that if the $\bar{\beta}^i_\nu$'s
  are isomorphisms, then, in particular,
  $L^{i-r+j}=\beta_\nu(L^{i+j})+L^{i-r+j+1}$ for $0\le j<e_P$.
  Substituting gives
  $L^{i-r}=\sum_{j=0}^{e_P-1}\beta_\nu(L^{i+j})+L^{i-r+e_P}=\beta_\nu(L^i)+tL^{i-r}$,
  so $\beta_\nu(L^i)=L^{i-r}$ by Nakayama's Lemma.
\end{rmk}

Any  fundamental stratum has a reduction with $\gcd(r, e_P) = 1$.
\begin{lemma}\label{gcdlemma}
If $(P, r, \beta)$ is a fundamental stratum, there is a fundamental reduction
$(P', r', \beta')$ with the property that $\gcd(r', e_P) = 1$.
\end{lemma}
\begin{proof}

  Let $g = \gcd (r, e_P)$ and $r' = r/g$.  Let $(L')^j = L^{j g}$, and
  set $\mathscr{L}' = ((L')^j)$.  This is the sub-lattice chain of
  $\mathscr{L}$ consisting of all lattices of the form $L^{a e_P + b
    r}$ with $a,b \in \Z$.  If we choose a representative $\beta_\nu$
  for $\beta$, $\beta_\nu ((L')^j) \subset (L')^{j - r'}$.  Thus,
  $\beta_\nu + \mathfrak{P}^{1-r} \subset (\mathfrak{P}')^{-r'}$.  Let
  $\beta'\in((\bfP')^{r})^\vee$ be the
  functional determined by the image of $\beta_\nu$ in
  $(\bfP')^{-r'}$.  If $\beta_\nu^N
  \in (\mathfrak{P}')^{-Nr' + 1},$ then $\beta_\nu^N \in
  \mathfrak{P}^{-Nr + 1}$: if $\beta_\nu (L')^j \subset (L')^{j-Nr' +
    1},$ then for any $j\in\Z$ and $0<m<g$, $\beta_\nu (L^{jg+m})
  \subset\beta_\nu (L^{jg}) \subset L^{jg - Nr + g}\subset L^{jg - Nr
    +m+ 1}$.  Thus, if $(P, r, \beta)$ is fundamental, so is $(P', r',
  \beta')$.

\end{proof}

\subsection{Split Strata}\label{section:splittings}
We now generalize the notion of a `split stratum' given in
\cite[Section 2]{Ku} and \cite[Section 2.3]{BK} to the geometric
setting.  Suppose that $(P, r, \beta)$ is a stratum in $V$ and that
$\mathscr{L} = (L^i)_{i\in\Z}$ is the lattice chain stabilized by $P$.
Let $V = V_1 \oplus V_2$ with $V_1, V_2 \ne \{0\}$.  Define $L^i_j =
L^i \cap V_j$ for $j = 1, 2$.  Note that $L^i_j$ has maximal rank in
$V_j$, so it is indeed a lattice.  Let $\mathscr{L}_j$ be the lattice
chain consisting of $(L^i_j)$ omitting repeats.  We denote the
parahoric associated to $\mathscr{L}_j$ by $P_j$.  Note that if $L^i =
L^i_1 \oplus L^i_2 $ for all $i$, then each $L^i_j$ is automatically a
lattice in $V_j$.

\begin{definition}\label{lvlr}
We say that $(V_1, V_2)$ \emph{splits} $P$ if
\begin{enumerate}
\item $L^i = L^i_1 \oplus L^i_2 $ for all $i$, and
\item $\mathscr{L}_1$ is a uniform lattice chain
with $e_{P_1} = e_P$.
\end{enumerate}
In addition, $(V_1, V_2)$ splits $\beta$  at level $r$ if
$\beta_\nu (L^i_j) \subset L^{i-r}_j + L^{i-r+1}$.
\end{definition}
Note that the above definition is independent of the choice of
representative $\beta_\nu$. However, it is possible to choose a
`split' representative for $\beta_\nu$.  Let $\pi_j : V \to V_j$ be
the projection, $\iota_j : V_j \to V$ the inclusion, and $\epsilon_j =
\iota_j \circ \pi_j$.  Set $\beta_{j\nu}=\pi_j \circ \beta_\nu \circ
\iota_j$ and $\beta'_{\nu}=\beta_{1\nu}\oplus\beta_{2\nu}$.  Whenever
$(V_1, V_2)$ splits $\beta$ and $P$, $\beta'_\nu\in \beta_\nu +
\mathfrak{P}^{1-r}$.  Thus, by replacing $\beta_{\nu}$ by
$\beta'_\nu$, we may assume without loss of generality that the
representative $\beta_{\nu}$ is ``block-diagonal'', i.e., it satisfies
$\beta_\nu' (L_j^i) \subset L_j^{i-r}$.  If $\beta_j$ is the
functional induced by $\beta_{j\nu}$, then $(P_j, r, \beta_j )$ is a
stratum in $\GL(V_j)$.

\begin{rmk}\label{Vsplit}
  If $P$ is uniform and $(V_1,V_2)$ splits $P$, then $P_2$ is also
  uniform with $e_{P_2} = e_P$, since $\bar{L}^i \cong \bar{L}^i_1
  \oplus \bar{L}^i_2$.  Furthermore, if $(P, r, \beta)$ is strongly
  uniform, $(V_1, V_2)$ splits $\beta$, and the first part of the
  splitting condition for $P$ is satisfied, then $(V_1, V_2)$ splits
  $P$.  Since $\gcd(r, e_P) = 1$ and $\beta_\nu (L^i) = L^{i-r}$ for
  all $i$, we may choose integers $a$ and $b$ such that $\alpha_\nu =
  t^a \beta_\nu^b$ generates the $\fP$-module $\mathfrak{P}^1$.  Thus,
  if we choose $\beta_\nu$ such that $\beta_\nu (V_j) \subset V_j$ as
  above, it is clear that $\alpha_\nu (L^i_1) = L^{i+1}_1$ and
  $\bar{\alpha}_\nu (\bar{L}^i_1) = \bar{L}^{i+1}_1$.  Accordingly,
  the subquotients $\bar{L}_1^i$ have the same dimension for all $i$
  and there are no repeats in the lattice chain.  This implies that
  $\mathscr{L}_1$ is uniform and $e_{P_1} = e_P$.  In fact, we see
  that  $(P_1, r, \beta_1 )$ is strongly uniform.
  \end{rmk}

Define  $V_{1 2}$ to be the vector space $\Hom_F(V_2, V_1)$, and
let $\partial_{\beta_\nu}$ be the operator
\begin{equation*}\begin{aligned}
\partial_{\beta_\nu} : V_{1 2} & \to V_{1 2}\\
x & \mapsto \beta_{1\nu} x - x \beta_{2\nu}.
\end{aligned}
\end{equation*}
We remark that if we embed $V_{12}$ in $\gl(V)$ in the obvious way and
assume that $\beta_\nu$ is block-diagonal, then
$\partial_{\beta_\nu} (x) = [\beta_\nu, x] = \delta_{\bn} (x)$.
 
The map $\partial_{\beta_\nu}$ is a degree $-r$ endomorphism of the 
$\mathfrak{o}$-lattice chain $\mathscr{M} = (M^j)$ defined by
\begin{equation*}
M^j = \{ x \in V_{1 2} \mid  x L^i_2 \subset L^{i +j}_1\text{ for all
} i\}.
\end{equation*}
By \cite[Lemma 2.2]{Ku}, $\mathscr{M}$ is a uniform lattice chain with
period $e_P$.  The functional on $\bfP$ induced by
$\partial_{\beta_\nu}$ is independent of the choice of representative;
we denote the corresponding stratum by $(P_{1 2}, r, \partial_\beta)$.

Recall from Remark~\ref{levi} that any element of $\fP/\fP^1$
determines a conjugacy class in $\gl_n(k)$.  Accordingly, if $(P, 0,
\beta)$ is a stratum with $r = 0$, it makes sense to refer to the
`eigenvalues' of $\bar{\beta}_\nu$.  If the stratum splits at level
$0$, then the eigenvalues of the diagonal blocks $(\bar{\beta}_1)_\nu$
and $(\bar{\beta}_2)_\nu$ are well-defined.

\begin{definition}\label{splitstrata}
We say that $(V_1, V_2)$ splits the fundamental stratum $(P, r, \beta)$ if
\begin{enumerate}[label={(\arabic*)},ref={\thedefinition(\arabic*)}]
\item\label{splitpart} $(V_1, V_2)$ splits $P$ and $\beta$ at level $r$;
\item\label{partialpart} $(P_1, r,  \beta_1)$ and $(P_{1 2}, r,
  \partial_\beta)$ are strongly uniform; and
\item\label{regpart} when $r = 0$, the eigenvalues of
  $\bar{\beta}_{1\nu}$ are distinct from the eigenvalues of
  $\bar{\beta}_{2\nu}$ modulo $\Z$.
\end{enumerate}
\end{definition}

\begin{rmk}\label{splitp1}

  The congruence subgroup $P^1$ acts on the set of splittings of
  $(P,r,\b)$, i.e., if $g\in P^1$ and $(V_1,V_2)$ splits $(P,r,\b)$,
  then so does $(gV_1,gV_2)$. First, note that $P^1$ stabilizes this
  stratum.  Next, given $g\in P^1$, it is clear that $L^i=g L^i=
  g L^i_1 \oplus g L^i_2$ and that $g\sL_1$ is uniform
  with the same period as $\sL_1$.  Thus, $(gV_1,gV_2)$ splits $P$; it
  also splits $\b$ at level $r$, since $g x\in x+ L^{i+1}$ for
  any $x\in L^i$.  It is obvious that the induced strata on $gV_1$ and
  $gV_{12}$ are strongly uniform.  Finally, note that when viewed as
  subalgebras of $\bfP$ is the natural way, $\bfP_j$ and
  $g\fP_j/g\fP_j^1$ are the same.  It follows that the eigenvalues of
  $g\beta_{j\nu}+g\fP_j^1$ and $\bar{\beta}_{j\nu}$ are the same, so
  the last condition also holds.
\end{rmk}

If $P_1$ is a uniform parahoric, then $P_{1 2}$ is as
well, with the same period.  To see this, note that there is an isomorphism
\begin{equation}\label{P12uniform}
  \bar{M}^j \to \bigoplus_{\ell = 0}^{e_P-1} \Hom (\bar{L}_2^\ell, \bar{L}_1^{\ell+j}).
\end{equation}
Since $\dim_k(\bar{L}_1^{\ell}) = \dim_k(L^0_1/t L^0_1)/e_{P_1}$ for all $\ell$, 
\begin{equation*}
\dim_k(\bar{M}^j) = \dim_k(L^0_1/t L^0_1) \dim_k(L^0_2/ t L^0_2) / e_{P_1},
\end{equation*}
and $P_{12}$ is uniform.  Furthermore, since $\dim_k (M^j/tM^j) = \dim_k (L^0_1/t L^0_1) \dim_k
(L^0_2/t L^0_2)$, it follows that $tM^j=M^{j+e_{P_1}}$, i.e.,
$e_{P_{12}} = e_{P_1}=e_P$.

Let $V_{2 1} = \Hom_F(V_1, V_2)$.  Define a lattice chain $\mathscr{N}
= \{N^i\}$ in $V_{2 1}$ in the same way as for $\mathscr{M}$.  An
argument similar to that given above shows that $\mathscr{N}$ is
uniform with period $e_P$ and that the operator
$\partial'_{\beta_\nu}$ on $V_{2 1}$ defined by $\partial'_{\beta_\nu} (x) =
\beta_{2\nu} x - x \beta_{1\nu}$ is an endomorphism of $\mathscr{N}$
of degree $-r$.  We let $(P_{2 1}, r, \partial_\beta')$ be the
associated stratum.

\begin{lemma}\label{spstratalemma}
The stratum $(P_{2 1}, r, \partial_\beta')$ is strongly uniform.
\end{lemma}
\begin{proof}
It only remains to show that $\partial'_\beta (N^i) = N^{i-r}$ for all
$i$.  First, observe that there is a natural injection
$\bM^i\hookrightarrow\bfP^i$, so by Proposition~\ref{duality}, we have
a surjection $\bfP^{-i}\cong(\bfP^i)^\vee\to(\bM^i)^\vee$.  Since the
kernel of this map consists of the image of the ``block upper
triangular'' matrices, we see that $(\bar{M}^i)^\vee \cong
\bar{N}^{-i}$.  Next, if $x \in M^i$
and $y \in N^{-i+r}$,
\begin{equation*}\begin{aligned}
\langle \beta_1 x - x \beta_2, y \rangle_\nu & = 
\Res (\Tr(\beta_1 x y)\nu - \Tr(x \beta_2 y)\nu) \\
& = \Res (\Tr( x y \beta_1)\nu - \Tr(x \beta_2 y)\nu) \\
& = -\langle x, \beta_2 y- y \beta_1 \rangle_\nu .
\end{aligned}
\end{equation*}
Since $\bar{\partial}_{\beta_\nu}:\bM^i\to\bM^{i-r}$ is an isomorphism
by Remark~\ref{sug}, it follows that $\bar{\partial}'_\beta (\bar{N}^{-i+r})
= \bar{N}^{-i} $, so $\partial_\beta (N^i) = N^{i-r}$ by the same
remark.

\end{proof}

Suppose that $(V_1, V_2)$ splits a uniform stratum $(P, r, \beta)$ as above.
By Remark~\ref{Vsplit}, $e_{P_1} = e_{P_2} = e_P$.
Thus, it is never the case that $L^i_j = L^{i+1}_j$ using the indexing
convention in Definition~\ref{lvlr}, and indeed $\mathscr{L}_j = (L^i_j)_{i \in \Z}$.
In this setting, it makes sense to think of $(P, r, \beta)$
as the direct sum of $(P_1, r, \beta_1)$ and $(P_2, r, \beta_2)$.

In the following, let $J$ be a finite  indexing set, and suppose
$V_J = \bigoplus_{j \in J} V_j$,
with each $V_j \ne \{0\}$.
Let $(P_j, r, \beta_j)$ be a stratum in 
$\GL(V_j)$ corresponding to a uniform parahoric $P_j$,
and let $e_{P_j} = e_{P_k}$ for all $j, k$.  
Define $L^i_J = \bigoplus_{j \in J} L^i_j$ 
and $\mathscr{L}_J = (L^i_J)_{i \in \Z}$, and let $P_J \subset \GL(V)$ 
be the parahoric subgroup that stabilizes $\mathscr{L}_J$.
Finally, let $\beta_J = \bigoplus_{j \in J} \beta_j$.
\begin{definition}\label{directsum}
Under the assumptions of the previous paragraph:
\begin{enumerate}
\item When $J = \{1, 2\}$, we say that $(P_J, r, \beta_J)=(P_1, r,
  \beta_1) \oplus (P_2, r, \beta_2)$ if $(V_1,V_2)$ splits $(P_J, r,
  \beta_J)$.
\item When $J = \{1, \ldots, m\}$, we define the direct
  sum recursively by
\begin{equation*}\label{sumeq}
\bigoplus_{j\in J} (P_j, r, \beta_j) = 
(P_1, r, \beta_1) \oplus \left( (P_2, r, \beta_2) \oplus \left( \ldots \oplus \left(P_m, r, \beta_m
\right) \right) \right).
\end{equation*}
Note that if we set $J_\ell = \{\ell, \ldots, m \}$ for $\ell\in J$,
then $(V_\ell, V_{J_{\ell+1}})$ must split $(P_{J_\ell}, r,
\beta_{J_\ell})$ for all $\ell$.
\item We say that a uniform stratum $(P, r, \beta) \in \GL(V)$ splits
  into the direct sum $\bigoplus_{j\in J} (P_j, r, \beta_j)$ if there
  is an isomorphism $V \cong V_J$ under which $(P, r, \beta)$ and
  $(P_J, r, \beta_J)$ are equivalent.
\end{enumerate}
\end{definition}
\begin{rmk}
  It is clear that the direct sum operation is associative.  It is not
  symmetric because Definition~\ref{partialpart} implies that $(P_j, r,
  \beta_j)$ is strongly uniform whenever $j \ne m$.  The definition is
  easily modified to make it symmetric, but we will not do so here.
\end{rmk}

We can determine if a stratum has a splitting by considering the characteristic polynomial
of $(P, r, \beta)$.  Fix a parameter $t \in F$ and let 
$g = \gcd(r, e_{P})$.  Define an element
\begin{equation}\label{ybeta}
y_\beta = \beta_\nu^{e_P/g} t^{r/g} + \mathfrak{P}^1\in\bfP.
\end{equation}
 Recall from Remark~\ref{levi} that $y_\beta$ determines a conjugacy
 class in $\gl_n(k)$.
\begin{definition}
We define the characteristic polynomial $\phi_\beta\in k[X]$
of the stratum $(P, r, \beta)$ to be the characteristic polynomial of
$y_\beta$.
\end{definition}

The local field version of the following proposition is in \cite[Proposition 3.4]{Ku}.
\begin{proposition}\label{splittingprop}
Suppose that $\gcd(r, e_P) = 1$ and $r > 0$.  The stratum $(P, r, \beta)$ is fundamental
if and only if $\phi_\beta (X)$ has a non-zero root.  
If $(P, r,\beta)$ is fundamental, it splits if
$\phi_\beta (X) = g(X) h(X)$ for
$g, h \in k[X]$ relatively prime of positive degree.  
\end{proposition}
\begin{rmk}
Given any fundamental stratum $(P, r, \beta)$, one can always find a reduction 
that satisfies the condition $\gcd(r, e_P) = 1$ by Lemma \ref{gcdlemma}.
\end{rmk}
\begin{proof}
  Note that $\bar{\beta}_\nu$ and $y_\beta=t^r\bar{\beta}^{e_P}_\nu$,
  viewed as endomorphisms of $\gr(\mathscr{L})$, are either
  simultaneously nilpotent or not.  Using the identification of $\bfP$
  and a Levi subalgebra of $\gl_n(k)$, we see that the latter is
  nilpotent if and only its characteristic polynomial
  $\phi_\beta(X)$ equals $X^n$.  Since $(P, r, \beta)$ is not fundamental if
  and only if $\bar{\beta}_\nu$ is nilpotent, we see that $(P, r,
  \beta)$ is fundamental if and only if $\phi_\beta(X)$ has a nonzero
  root.

  Let $\tilde{\phi}_\beta\in F[X]$ be the characteristic polynomial of
  $\tilde{y} = t^r\beta_\nu^{e_P}$.  Then, $\tilde{\phi}_\beta$
  necessarily has coefficients in $\mathfrak{o}$ and
  $\tilde{\phi}_\beta \equiv \phi_\beta \pmod{\mathfrak{p}}$.
  Hensel's lemma 
  states that
  $\tilde{\phi}(X) = \tilde{g}(X) \tilde{h}(X)$, where $\tilde{h}
  \equiv h \pmod{\mathfrak{p}}$ and $\tilde{g} \equiv g
  \pmod{\mathfrak{p}}$.

We take $V_1= \ker(g(\tilde{y}))$ and $V_2 = \ker(h(\tilde{y}))$.  By
Lemmas 3.5 and 3.6 of \cite{Ku}, $(V_1, V_2)$ splits $P$ and $\beta$
at level $r$, $\beta_1(L^i_1) = L^{i-r}_1$ for all $i$, and
$\partial_\beta(M^j) = M^{j-r}$ for all $j$.  Therefore, $(V_1,V_2)$
splits $(P, r, \beta)$.
\end{proof}

\begin{corollary}\label{nilsplitting}
Suppose that $(P, r, \beta)$ is a uniform stratum that is not strongly uniform.  Then, 
$(P, r, \beta)$ splits into the direct sum of two strata $(P_1, r, \beta_1)$
and $(P_2, r, \beta_2)$, where
$(P_1, r, \beta_1)$ is strongly uniform and $(P_2, r, \beta_2)$
is non-fundamental.  
\end{corollary}
\begin{proof}
Factor $\phi_\beta(X) = g (X) h (X)$ so that
$h(X) = X^m$ and $g(0) \ne 0$, and let $V_1$
and $V_2$ be the subspaces from the proof of the previous
proposition.  Since  $(P, r, \beta)$ is fundamental, $\deg(g)>0$ and
$V_1$ is nontrivial.   Moreover, Remark~\ref{sug} implies that $V_1\ne V$; if not,
$y_\beta=t^r\bar{\beta}_\nu^{e_P}$ (and hence $\bar{\beta}_\nu$) would
be invertible endomorphisms of $\gr(\mathscr{L})$, contradicting the
fact that  $(P, r, \beta)$ is not strongly uniform.

Since $y_\beta$ (and hence $\bar{\beta}_\nu$) restricts to an
automorphism of $\gr(\mathscr{L}_1)$, the same remark shows that
$(P_1, r, \beta_1)$ is strongly uniform.  On the other hand, since
$y_\beta$ restricts to a nilpotent endomorphism of
$\gr(\mathscr{L}_2)$, Remark~\ref{asgrd} shows that $(P_2, r,
\beta_2)$ is not fundamental.
\end{proof}

\section{Regular Strata}\label{regstrata}
In this section, we make precise the notion of a stratum with regular
semisimple ``leading term''. We introduce the concept of a regular
stratum; this is a stratum which is ``graded-centralized'' by a maximal
torus.  Regular strata do not appear in the theory of strata for local
fields.  However, they play an important role in the geometric
theory.

\subsection{Classification of regular strata}

Consider a stratum $(P,r,\b)$.  Recall that the congruence subgroups
$P^i$ act on $\beta_\nu$ by the adjoint action.  In particular,
\begin{equation*}
\Ad(P^i)(\beta_\nu) \subset \beta_\nu +\mathfrak{P}^{i-r},
\end{equation*}
since $p \in P^i$ implies that $p$ and $p^{-1}$ act trivially on $L^j/
L^{i+j}$.  Define $Z^i(\beta_\nu) \subset \bar{P}^i$ to be the
stabilizer of $\beta_\nu \pmod{\mathfrak{P}^{i-r+1}}$.  Notice that
this is independent of the choice of representative $\beta_\nu$ for
$\beta$: if $\beta'_\nu = \beta_\nu + \gamma$, for some $\gamma \in
\mathfrak{P}^{-r+1}$, then $ \Ad(P^i) (\gamma) \subset \gamma +
\mathfrak{P}^{i-r+1}$.

Let $T\subset\GL(V)$ be a maximal torus.  The corresponding Cartan
subalgebra $\ft\subset\gl(V)$ is the centralizer of a regular
semisimple element and is therefore an associative subalgebra.  In
particular, since $\mathfrak{t}$ must have the structure of a
commutative semisimple algebra, $\mathfrak{t}$ is the product of field
extensions of $F$: $\mathfrak{t} = E_1 \times E_2 \times \ldots \times
E_\ell$ and $T = E_1^\times \times E_2^\times \times \ldots \times
E_\ell^\times$.  We let $\fo_j$ be the ring of integers of $E_j$ and
$\fp_j$ its maximal ideal.  Let $s_j=[E_j:F]$.  The field $E_j$
contains a uniformizer which is an $s_j^{th}$ root of $t$; we let
$\omega_j\in\ft$ denote this uniformizer supported on the $j^{th}$
summand.  We will also denote the identity of the $j^{th}$ Wedderburn
component of $\ft$ by $\chi_j$.

There is a map $N_T : T \to (F^\times)^\ell$ obtained by
taking the norm on each summand.  Define $T(\mathfrak{o}) = N_T^{-1}
(\mathfrak{o}^\times)^\ell=\prod_j \fo_j^\times$.  Similarly, there is
a trace map $\Tr_{\ft}:\ft\to F^\ell$, and we set
$\ft(\fo)=\Tr_\ft^{-1}(\fo)=\prod_j \fo_j$. We also define a
finite-dimensional $k$-toral subalgebra and $k$-torus:
$\tfl\subset\ft(\fo)$ is the $k$-linear span of the $\chi_j$'s and
$\Tfl\overset{\mathrm{def}}{=}(\tfl)^\times\subset T(\fo)$.  Of
course, $\Tfl\cong(k^\times)^\ell$ and $\tfl\cong k^\ell$.

We will be concerned with tori which are compatible with a given
parahoric subgroup in the sense that $T(\fo)\subset P$ or equivalently
$\ft(\fo)\subset\fP$. 

\begin{lemma}\label{comptor}  If $T(\fo)\subset P$, then $T\cap P^i=T(\fo)\cap P^i$
  and  $\ft\cap \fP^i=\ft(\fo)\cap \fP^i$ for all $i\ge 0$.
\end{lemma}
\begin{proof} It suffices to show that $T\cap P=T(\fo)$ and
  $\ft\cap\fP=\ft(\fo)$; moreover, the first statement follows from
  the second by taking units.  Since the central primitive idempotents
  $\chi_j$ are contained in $\ft\cap\fP$, it is enough to check that if $x\chi_j\in E_j\cap
  \fP$, then $x\in\fo_j$.  Suppose $x\notin\fo_j$, so that
  $x\chi_j=\omega_j^q f$ for some $f\in
  \fo_j^\times$ and $q<0$.   Since $\fo_j^\times\chi_j\in P\subset\fP$,
  we see that $\omega_j^q\in\fP$.  This implies that
  $t^q\chi_j=\omega_j^{qm}\in\fP$.  We deduce that $t^s\chi_j\in\fP$ for
  all $s\in\Z$, which is absurd.
\end{proof}

\begin{definition}\label{regstratum}
  A uniform stratum $(P, r, \beta)$ is called \emph{regular} if there
  exists a maximal torus $T$ (possibly non-split) with the following
  properties:
\begin{itemize}
\item $T(\mathfrak{o}) \subset P$;
\item $\bar{T}^i = Z^i(\beta_\nu)$ for all $i$;
\item $y_\beta \in \bar{P}$ (defined as in \eqref{ybeta}) is semisimple;
\item in the case $r = 0$ (and thus $e_p = 1$), 
the eigenvalues of $\bar{\beta}_\nu \in \gl(\bL^0)$ are distinct
modulo $\Z$.
\end{itemize}
 We say that
$T$ centralizes $(P, r, \beta)$.  If $T \cong E^\times$ for some field
extension $E/F$, the stratum  is called \emph{pure}.
\end{definition}

\begin{rmk}\label{regularaction} Suppose that $(P,r,\b)$ is a regular stratum
  centralized by $T$,  and $L$ is a lattice with $P\subset \GL(L)$.
  Then, for any $g\in \GL(L)$, $(gPg^{-1}, r, \Ad^*(g)\b)$ is a regular
  stratum centralized by $gTg^{-1}$.
\end{rmk}
\begin{rmk}\label{Tcentral}
  If $T$ centralizes a regular stratum $(P, r, \beta)$, then any
  conjugate of $T$ by an element of $P^1$ also centralizes $(P, r,
  \beta)$.  Thus, $T$ is not unique.
\end{rmk}

It will be useful to have a variation of Definition~\ref{regstratum}
in terms of the graded action of $\mathfrak{t}$ on $\beta$.  Define
$\mathfrak{z}^i (\beta_\nu) \subset \bar{\mathfrak{P}}^i$ to be the
image of $\{z \in \mathfrak{P}^i \mid \ad(z) (\beta_\nu) \in
\mathfrak{P}^{-r+i+1}\}$.

\begin{proposition}\label{liecen}
  Let $P$ be a uniform parahoric, and let $(P, r, \beta)$ be a regular
  stratum centralized by the torus $T$.  Then $\bar{T}^i = Z^i
  (\beta_\nu)$ if and only if $\bar{\mathfrak{t}}^i = \mathfrak{z}^i
  (\beta_\nu)$ for each $i\ge 0$. 
\end{proposition}
\begin{proof} First, we take $i=0$.  Suppose $\bar{\mathfrak{t}}^0 =
  \mathfrak{z}^0 (\beta_\nu)$.  Given $z\in P$, it is clear that $z\bn
  z^{-1}-\bn\in\fP^{-r+1}$ if and only if $\ad(z)\bn\in \fP^{-r+1}$.
  This immediately gives $\bar{T}^0\subset Z^0 (\beta_\nu)$.  It also
  implies that if $zP^1\in Z^0 (\beta_\nu)$, then
  $z+\fP^1\in\mathfrak{z}^0 (\beta_\nu)$.  By assumption, this means
  that there exists $s\in\ft(\fo)$ such that $z-s\in\fP^1$, and we
  obtain $s\in P\cap \ft(\fo)=T(\fo)$.  Consequently, $s^{-1}z\in
  P^1$, i.e., $zP^1\in\bar{T}^0$.

  Next, suppose that $Z^0 (\bn) = \bar{T}^0$.  Recall that a
  finite-dimensional $k$-algebra is spanned by its units.  (Let $A$ be
  such an algebra with Jacobson radical $J$.  Since $1+x\in A^\times$
  for $x\in J$, $J$ is in the span of $A^\times$.  Moreover, $a\in
  A^\times$ if and only if $\bar{a}\in (A/J)^\times$.  The result now
  follows because $A/J$ is a product of matrix algebras, and hence is
  spanned by its units.)  We show that
  $(\fz^0(\bn))^\times=(\bft^0)^\times$.  Suppose
  $y\in(\fz^0(\bn))^\times\subset\bfP^\times$.  Since $\fP^1$ is the
  Jacobson radical of $\fP$, any $z\in\fP$ lifting $y$ is invertible,
  hence lies in $P$.  The argument above shows that $zP^1\in
  Z^0(\bn)$, so we can assume $z\in T$, i.e., $y\in(\bft^0)^\times$.
  A similar argument (using the fact that a lift of $y\in
  (\bft^0)^\times$ to $\ft(\fo)$ actually lies in $T(\fo)$) gives the
  reverse inclusion.  We conclude that
  $\fz^0(\bn)=\spa\left((\fz^0(\bn))^\times\right)=\spa\left((\bft^0)^\times\right)=\bft^0$.

Now suppose $i >0$.  There is an isomorphism $\bar{\fP}^i \to
\bar{P}^i$ induced by $X \mapsto 1 + X$.  Since $\Ad (1+X) (\beta_\nu)
\in \beta_\nu + \ad (X) (\beta_\nu) + \mathfrak{P}^{-r+i+1}$ for $X
\in \mathfrak{P}^i$, it is clear that this map restricts to give an
isomorphism between $\mathfrak{z}^i (\beta)$ and $Z^i (\beta)$.  Since
this same map takes $\bar{\mathfrak{t}}^i$ to $\bar{T}^i$, the proof
is complete.

\end{proof}

\begin{rmk}\label{repint}
  If $T$ centralizes $(P, r, \beta)$, then in fact
  $\bar{\mathfrak{t}}^i = \mathfrak{z}^i (\beta_\nu)$ for all
  $i\in\Z$.  For $i\ge 0$, this has been shown in the proposition.  On
  the other hand, if $i<0$ and $s$ is any integer such that $i+s
  e_P\ge 0$, then the result follows because multiplication by $t^s$
  induces isomorphisms $\bft^i\cong\bft^{i+se_P}(\bn)$ and
  $\fz^i\cong\fz^{i+se_P}(\bn)$.  Since $\bbn\in\fz^{-r}(\bn)$,
  $\bn\in\ft\cap\fP^{-r}+\fP^{-r+1}$; it follows that we can always
  choose the representative $\bn\in\ft\cap\fP^{-r}$.
\end{rmk}

\begin{corollary}\label{liecencor}
Let $(P, r, \b)$ be a regular stratum centralized by $T$ and let $X \in \fP^\ell$.
If $\bn \in \ft^{-r} $ is a representative for $\b$, and $\ad(X) (\bn) \in \fP^{-r+j}$,
then $X \in \ft^\ell + \fP^j$.
\end{corollary}
\begin{proof}
When $\ell \ge j$, there is no content.  
We note that the case $j = \ell+1$ follows from Proposition~\ref{liecen}
and Remark~\ref{repint}.  By induction on $j > \ell$, 
suppose that the statement is true for
$j -1$.  Let $Y \in \fP^{j-1}$ satisfy $X - Y \in \ft$.  Then, 
$X \in \ft^\ell + \fP^j$ if and only if $Y \in \ft^{j-1} + \fP^j$.  The latter
statement follows from the base step above.
\end{proof}

The main goal of this section is to give a structure theorem for
regular strata.
\begin{theorem}\label{thm1} Let $(P, r, \beta)$ be a regular stratum.\begin{enumerate}\item If $(P,r,\b)$ is pure, then
    $e_P=\dim V$.
\item
  \begin{enumerate}
\item If $(P, r, \beta)$  is strongly uniform, then it splits into a
  direct sum of pure strata (necessarily of the same dimension).
\item If $(P, r, \beta)$ is not strongly uniform, then $e_P = 1$ and
  $(P, r, \beta)$ is the direct sum of a regular, strongly uniform
  stratum and a non-fundamental stratum of dimension $1$.
\end{enumerate}
\noindent In each case, the splitting coincides with the splitting induced by a
$P^1$-conjugate of $T$.
\end{enumerate}
\end{theorem}

\begin{corollary}\label{uniformsplitting}
  Let $(P, r, \beta)$ be a regular stratum centralized by $T$.  Take
  $E/F$ to be the unique (up to isomorphism) field extension of degree
  $e_P$.  Then, $T \cong (E^\times)^{n / e_P}$.  Moreover, the maps
  $\Tfl \to \bar{T}^0$ and $\tfl \to \bar{\ft}^0$ are isomorphisms.
\end{corollary}
\begin{rmk}
We note that if $(P, r, \beta)$ is not strongly uniform,
Theorem~\ref{thm1} implies that $e_P = 1$.  By the corollary, this
can only happen when $T$ is totally split. 
\end{rmk}

The following proposition allows us to make sense of what it means
for an element $\bn$ to have regular semisimple leading term.

\begin{proposition}\label{lemrss} If 
  $(P, r, \beta)$ is regular, then every representative $\beta_\nu$
  for $\beta$ is regular semisimple.
\end{proposition}

By Remark~\ref{repint}, we may choose $\bn \in \ft \cap \fP^{-r}$.
Corollary~\ref{uniformsplitting} implies that $\bn$ is a block
diagonal matrix with entries in $F[\varpi_E]^\times$.  Then, the
\emph{leading term} $\bn'$ is the matrix consisting of the degree
$-r/e_P$ terms from each block diagonal entry in $\bn$ (after
identifying $\varpi_E$ with $t^{1/e_P}$).

It suffices to check that $\bn - \bn' \in \fP^{-r+1}$, in which case
Proposition~\ref{lemrss} implies that $\bn'$ is regular semisimple.
If $e_P = 1$, this is clear.  When $e_P > 1$, we may assume without
loss of generality that the splittings for $T$ and $(P,r, \beta)$ are
induced by the same splitting of $V$.  In particular, the $j^{th}$
block-diagonal entry $\beta_{j \nu}$ is a representative for the
$j^{th}$ summand of $(P, r, \beta)$.  Therefore, by
Theorem~\ref{thm1}, we may reduce to the case where $(P, r, \beta)$ is
pure and $e_P = \dim V$.  In this case, $\varpi_E$ generates
$\fP^{1}$, so it is clear that $\bn - \bn' \in \fP^{-r+1}$.

We call a maximal torus \emph{uniform} if it isomorphic to
$(E^\times)^\ell$ for some field extension $E$.  Given a fixed lattice
$L$ and a uniform maximal torus $T$ with $T(\fo)\subset \GL(L)$, we
can associate a corresponding parahoric subgroup $P_{T,L}\subset
\GL(L)$ containing $T(\fo)$ as follows. The isomorphism $\ft\cong
E^{\ell}$ induces splittings $V=\oplus V_j$ and $L=\oplus L_j$.
Lemma~\ref{219} states that there is a unique complete lattice chain
$(L_j^i)_{i \in \Z}$ in $V_j$ up to indexing; we normalize it so that
$L_j^0=L_j$. Let $\mathscr{L}_T$ be the lattice chain with $L^i=\oplus
L_j^i$, and let $P_{T,L}$ be its stabilizer.  Since $L_0=L$, we have
$P_{T,L}\subset \GL(L)$ as desired.  It is obvious that $T(\fo)\subset
P_{T,L}$.  Note that $e_{P_{T,L}}=n/\ell=[E:F]$.

Given a uniform torus $T$, there is a canonical $\Z$-grading on its
Cartan subalgebra $\ft$; the $i^{th}$ graded piece is given by
$\bigoplus k\varpi_E^i\chi_i$, where $\varpi_E$ is a uniformizer in
$E$ which is an $[E:F]^{th}$ root of $t$.  We denote the corresponding
filtration by $\ft^{(i)}=\bigoplus \fo_E\varpi_E^i\chi_i$.  There is a
corresponding canonical $\N$-filtration on $T(\fo)$ given by
$T^{(0)}=T(\fo)$ and $T^{(i)}=1+\ft^{(i)}$.


 \begin{proposition}\label{uniquedet} Let $L$ be a fixed lattice, and
   let $T\cong (E^\times)^\ell$ be a uniform maximal torus with
   $T(\fo)\subset \GL(L)$.\begin{enumerate}
\item If $P \subset \GL(L)$ is a parahoric subgroup
   for which there exists a regular stratum $(P, r, \b)$
  centralized by $T$, then $P=P_{T,L}$.
\item
  \label{uniquedetp1} Let $r\ge 0$ satisfy $(r,n/\ell)=1$, and suppose
  $x\in\ft^{-r}$ has regular semisimple leading term.  Then, there is
  a unique regular stratum $(P,r,\beta)$ with $P \subset \GL(L)$ which
  has $x$ as a representative.
\item The canonical filtrations on $\ft$ and $T(\fo)$ coincide with
  the filtrations induced by $P_{T,L}$, i.e., $\ft^{(i)}=\ft\cap
  \fP^i_{T,L}$ and $T^{(i)}=T\cap P_{T,L}^i$.

\end{enumerate}
\end{proposition}




\begin{proof}  
  Let $(P,r,\beta)$ be a regular stratum as in the first statement.
  Since $P \subset \GL(L)$, we may take $L^0= L$ in the
  corresponding lattice chain $\mathscr{L}$.  By
  Corollary~\ref{uniformsplitting}, $\ft\cong E^{n/e_{P}}$, so
  $e_{P}=n/\ell$. Theorem~\ref{thm1} implies that the splitting
  $V=\oplus V_j$ induced by $\ft\cong E^{n/{e_P}}$ splits $(P, r,
  \beta)$ into a direct sum of pure strata when $e_{P} > 1$ and a sum
  of one-dimensional strata (with at most one non-fundamental summand)
  when $e_{P} = 1$.  In either case, each $V_j$ is an $E$-vector space
  of dimension one, and $L^i=\oplus(L^i\cap V_j)$.  However,
  Lemma~\ref{219} states that there is a unique complete lattice chain
  $(L_j^i)_{i \in \Z}$ in $V_j$ up to indexing, and we know that
  $L^0_j=L \cap V_j$.  By definition, $\mathscr{L}=\mathscr{L}_T$,
  so $P=P_{T,L}$.

The uniqueness part of the second statement is now immediate.  For
existence, it is clear from the construction of $P_{T,L}$ that  
$(a_1 \varpi_E^{-r}, \ldots,
  a_{\ell} \varpi_E^{-r}) \in E^{\ell} \cong \ft$ determines a regular
  stratum with parahoric subgroup $P_{T,L}$ if $a_i \ne a_j$ whenever $i \ne
  j$ and $r$ is coprime to $n/\ell$.  Since the leading term of $x$ is
  regular semisimple, we obtain a regular stratum $(P_{T,L},r,\beta)$ with
  the leading term of $x$, and hence $x$ itself, as a representative.

Finally, $\ft|_{V_j} \cong E$.  It follows that $L_j^i =
\varpi_E^j L_j^0$.  We deduce that 
$\ft \cap \fP_{T, L}^i = \bigoplus (\ft|_{V_j} \cap \fP^i_{T_{V_j}, L_j^0}) =
\bigoplus \fo_E \varpi_E^i \chi_j = \ft^{(i)}$.  The fact that
$T^{(i)}=T\cap P_{T,L}^i$ is an immediate consequence when $i\ge 1$.
The $i=0$ case is obtained by taking units in $\ft^{(0)}=\ft\cap\fP_{T,L}$.

\end{proof} 

If $(P,r,\b)$ is a regular stratum centralized by $T$, then the
proposition shows that $\{\ft^i\}$ and $\{T^i\}$ are actually the
canonical filtrations.

\begin{rmk}\label{standardtorus}  Given a fixed $F$-isomorphism $V\overset{\sim}{\to}F^n$, we can choose a
  standard representative of each conjugacy class of uniform maximal
  tori.  Indeed, if the torus is isomorphic to
  $(F((t^{1/e}))^\times)^{n/e}$, then under the identification
  $\GL(V)\cong \GL_n(F)$, we can choose a block diagonal
  representative $T$ (and $\ft$) with each uniformizer $t^{1/e}$
  mapping to the $e\times e$ matrix $\varpi_I$ from \eqref{varpip} in
  the corresponding block.  In this case, $P_T$ is the standard
  uniform parahoric subgroup that is `block upper-triangular modulo $t$'.
\end{rmk}

\subsection{Lemmas and proofs}

We now give proofs of the results described above.  We also include
some lemmas that will be needed later.  We remark that this section is
largely technical in nature.

\begin{lemma}\label{410}
  The homomorphism $\Tfl \to \bar{T}^0$ is an injection, and if $U$ is
  the unipotent radical of $\bar{T}^0$, then the induced map $\Tfl \to
  \bar{T}^0 /U$ is an isomorphism.  Similarly, the map $\tfl \to
  \bar{\mathfrak{t}}^0$ is an injection which induces an isomorphism
  $\tfl \cong \bar{\mathfrak{t}}^0 / \mathfrak{n}$, where
  $\mathfrak{n} \subset \bar{\mathfrak{t}}^0$ is the Jacobson
  radical.
\end{lemma}
\begin{proof}

  It is immediate from the definitions that $\Tfl \subset P$ and $\Tfl
  \cap P^1=\{1\}$, so $\Tfl \to \bar{T}^0$ is injective.  Moreover,
  the unipotent radical $U$ of $\bar{T}^0$ is the image of
  $\prod_j(1+\fp_j)$, whence the isomorphism $\Tfl \to \bar{T}^0 /U$.
  A similar proof works for $\tfl$.

\end{proof}

\begin{lemma}\label{nonfund}
  Suppose that $P$ is a uniform parahoric and $(P, r, \beta)$ is a
  non-fundamental stratum in an $F$-vector space $V$.  If $(P, r,
  \beta)$ satisfies the first three conditions of
  Definition~\ref{regstratum} and $\gcd (r, e_P) = 1$, then $V$ must
  have dimension one.
\end{lemma}
\begin{proof}
  Since $(P, r, \beta)$ is non-fundamental, it follows that there is a
  minimal $m > 0$ such that $\beta_\nu^{m} \in \mathfrak{P}^{-r m
    +1}$.  Without loss of generality, we may assume $\beta_\nu^m =
  0$.  Indeed, after choosing a basis, Lemma~\ref{subquotient} shows
  that we may take the representative $\beta_\nu$ to be the product of
  $\varpi_P^{-r}$ with an element $D \in \mathfrak{h}$.  By Remark
  \ref{unifvarpi}, we may assume $\Ad(\varpi_P) (\mathfrak{h}) \subset
  \mathfrak{h}$.  Therefore, $(D \varpi_P^{-r})^m = D' \varpi_P^{-r
    m}$ for some $D' \in \mathfrak{h}$.  Since $D' \varpi_P^{-rm} \in
  \mathfrak{P}^{-rm+1}$, $D' \in \mathfrak{P}^1 \cap \mathfrak{h} =
  \{0\}$. 

  First, we claim that $\beta_\nu$ is regular nilpotent.  Let
  $\mathfrak{z}$ be the centralizer of $\beta_\nu$ in $\fP$.  Note
  that $\mathfrak{z}$ is a free $\mathfrak{o}$-module of rank equal to
  the dimension of the centralizer of $\beta_\nu$ in $\gl(V)$, hence
  is at least $n$.  Since Nakayama's lemma implies that
  $\rank(\fz)=\dim_k(\mathfrak{z}/t \mathfrak{z})$, to show that $\bn$
  is regular, it now suffices to show that $\dim_k(\mathfrak{z}/t
  \mathfrak{z})\le n$.

  By Lemma~\ref{comptor}, $\ft^0 /\ft^{e_P}=\ft(\fo)/t\ft(\fo)$, which clearly has
  $k$-dimension $n$.  Also, recalling our convention that
  $\bar{\fz}^i$ is the projection of $\fz$ in $\bfP^i$, we have
  $\bar{\fz}^i\subset \mathfrak{z}^i (\beta)$.  It follows that
\begin{equation*}
\dim_k (\mathfrak{z}/t \mathfrak{z}) \le \sum_{i = 0}^{e_P-1} 
\dim_k (\mathfrak{z}^i (\beta)) = \sum_{i = 0}^{e_P - 1} \dim_k (\bar{\mathfrak{t}}^i)
 = \dim_k (\mathfrak{t}(\mathfrak{o}) / t \mathfrak{t} (\mathfrak{o}))=n
\end{equation*}
as desired.  Note that this argument actually shows that any coset
representative $\bn$ is regular.

Since the index of nilpotency of a regular nilpotent matrix is $n$, we
have $m=n$.  In particular, since $y_\beta = t^{r} \beta_\nu^{e_P} +
\mathfrak{P}^1\in\bfP$ is nilpotent, $y_\beta$ is semisimple only if
$t^r\bn^{e_P}\in\fP^1$.  This implies that
$\bn^{e_P}\in\fP^{-re_P+1}$, so $n\ge e_P\ge m$, i.e., $e_P = n$.
Thus, $P = I$ is an Iwahori subgroup, and $\gcd(n, r) = 1$.

In the notation of Section~\ref{sec:cores}, we write $\beta_\nu = x
\varpi_I^{-r}$ where $x = \diag (x_0,  \dots, x_{n-1}) \in
\mathfrak{d} $.  Define $\sigma^q(x) = (x_{-qr}, x_{1 - q r}, \ldots,
x_{n-1 - q r})$ to be the cyclic shift of the coefficients of $x$ by
$-q r$ places (with indexing in $\Z_n$).
It is immediate from \eqref{directcalc} that
$\Ad(\varpi_I^q)(x)=\sigma^q (x)$.  Therefore,
\begin{equation*}
\bn^s= x  \Ad (\varpi_I^{-r}) (x) \ldots \Ad (\varpi_I^{-(s-1) r}) (x) \varpi_I^{-rs}
 = x  \sigma^{1}(x) \ldots \sigma^{s-1} (x).
\end{equation*}
By assumption, $\bn^{n-1} \ne 0$.  Thus, one of the components
$x'_j = x_j x_{j -r} \ldots x_{j - (n-2)r}$ of $x' = x \sigma^{1} (x) \ldots \sigma^{n-2} (x)$ 
is non-zero;  since $\gcd (r, n) = 1$, $x'_j$ is the product of all but one
of the components of $x$.
Moreover, $\bn^{n} = 0$, so $x'_j x_{j - (n-1) r} = 0$ and
we conclude that exactly one component of $x$ is equal to $0$.  

Without loss of generality, assume that $x_0 = 0$.  Then, if $\bar{p}
\in Z^0 (\beta),$ we may choose a representative $p = \diag (p_0,
\ldots, p_{n-1}) \in \mathfrak{d}^*$.  Equation \eqref{Adcalc} shows
that $p_i = p_{i-r}$ for all $i$; again, since $\gcd (r, n) = 1$, this
implies that $p_0 = p_1 = \dots = p_{n-1}$.  It follows that
$Z^0(\beta)$ has dimension $1$.  Lemma~\ref{410} now implies that
$\Tfl$ also has dimension $1$, so $T = E^\times$, where $E/F$ is a
field extension of degree $n$.

By Lemma~\ref{219}, $\mathfrak{L}$ is a saturated chain of
$\mathfrak{o}_E$-lattices, so may assume that $\varpi_I$ is a
uniformizer in $E$.  Applying \eqref{directcalc}, we see that $\ad (x)
(\varpi_I) = x' \varpi_I$ where $x' = \diag (x_0 - x_1, \dots)$.
However, $x_0 = 0$ and $x_1 \ne 0$ whenever $n > 1$, so $E^\times$
only centralizes $(P, r, \beta)$ when $n = 1$.
\end{proof}
\begin{lemma}\label{stronglyunif}

  Let $(P, r, \beta)$ be a regular stratum centralized by a torus $T$.
  If $(V_1, V_2)$ is a splitting of $(P, r, \beta)$, then $(P_1, r,
  \beta_1)$ is regular, and $(P_2, r, \beta_2)$ is either regular or
  non-fundamental.  In the latter case, $V_2$ has dimension $1$ and
  $e_P = 1$.  Moreover, there exists $p\in P^1$ such that $(V_1,V_2)$
  splits the torus $pTp^{-1}$ (which also centralizes $(P, r, \beta)$) into
  $T_1\times T_2$; here, $T_1$ centralizes $(P_1, r, \beta_1)$ and
  $T_2$ centralizes $(P_2, r, \beta_2)$ when this stratum is regular.
\end{lemma}
\begin{proof}
  Let $Z^i_j (\beta_\nu) = Z^i(\beta_{j\nu})$, the centralizer of
  $\beta_{j\nu}$ in $\bar{P}^i_j$; similarly, let $\mathfrak{z}^i_j
  (\bn)= \mathfrak{z}^i (\beta_{j\nu})$.  First, we claim that $Z^i
  (\bn) = Z_1^i(\bn) \times Z_2^i (\bn)$, which is embedded in
  $\bar{P}^i$ by diagonal blocks.  It is clear that $Z_1^i
  (\beta_{\nu}) \times Z_2^i (\beta_{\nu}) \subset Z^i(\bn)$.  Recall
  that $\partial_{\bn}$ (resp. $\partial'_{\bn}$) has trivial kernel
  in $\bar{M}^i$ (resp. $\bar{N}^i$) by Definition~\ref{splitstrata}
  (resp. by Lemma~\ref{spstratalemma}).  If we identify $\bar{M}^i$
  and $\bar{N}^i$ with the upper and lower off-diagonal components of
  $\bar{\mathfrak{P}}^i$, then $\delta_{\bn}$ preserves each of these
  subspaces and restricts to $\partial_{\bn}$ (resp.
  $\partial'_{\bn}$) on $\bar{M}^i$ (resp.  $\bar{N}^i$).  Since
  $\delta_{\bn}$ also preserves the diagonal blocks, its kernel lies
  in $\bar{\mathfrak{P}}^i_1 \times \bar{\mathfrak{P}}^i_2 $.

  We handle the cases $i >0$ and $i = 0$ separately.  When $i = 0$,
  $Z^0 (\bn) \subset \bar{P} \subset \bar{\mathfrak{P}}$.  It is clear
  that $Z^0 (\bn)$ lies in the kernel of $\delta_{\bn}$.
  Therefore, $Z^0 (\bn)$ is supported on the diagonal blocks, so $Z^0
  (\bn) = Z^0_1 (\bn) \times Z^0_2 (\bn)$.  When $i > 0$, there is an
  isomorphism $\bar{\mathfrak{P}}^i \to \bar{P}^i $ induced by $x
  \mapsto 1 + x$.  Moreover, $\Ad (1+x) (\bn) \in \beta_{\nu} + \ad(x)
  (\beta_\nu) + \mathfrak{P}^{i-r+1}$.  Thus, $Z^i (\bn))$ must lie in
  $1 + \ker(\delta_{\bn})$.  The same argument as above implies
  that $Z^i (\bn)$ is supported on the diagonal blocks and equals
  $Z^i_1 (\bn) \times Z^i_2 (\bn)$.  Similarly, one shows that
  $\mathfrak{z}^i(\bn) = \mathfrak{z}^i_1 (\bn) \times
  \mathfrak{z}^i_2 (\bn)$.

  Let $\mathfrak{u}$ be the Jacobson radical of $\mathfrak{z}^0
  (\bn)$.  If $\epsilon_j \in \mathfrak{P}$ is the idempotent
  corresponding to the identity in $\mathfrak{P}_j$, its image
  $\bar{\epsilon}_j$ in $\bar\fP$ lies in $\mathfrak{z}^0_j(\bn)$.
  Therefore, $\mathfrak{u} = \bar{\epsilon}_1 \mathfrak{u} +
  \bar{\epsilon}_2 \mathfrak{u}$, and the splitting $\mathfrak{z}^0
  (\bn) = \mathfrak{z}^0_1 (\bn) \times \mathfrak{z}^0_2 (\bn)$
  induces a splitting on $\mathfrak{z}^0(\bn) /\mathfrak{u}$.
  Moreover, this splitting is non-trivial, since $\epsilon_j$ has
  non-trivial image in $\mathfrak{z}^0 (\bn) / \mathfrak{u}$.  Lemma
  \ref{410} implies that $\tfl \cong \mathfrak{z}^0
  (\bn)/\mathfrak{u}$, so $\tfl$ is split.  Let $\epsilon_1'$ and
  $\epsilon_2'$ be the idempotents corresponding to the identity in
  each summand of $\tfl$.  The same lemma implies that
  $\bar{\epsilon}_j'\in\bar{\mathfrak{P}}$ is simply the identity
  matrix in the corresponding diagonal block; thus, $\epsilon_j' \in
  \epsilon_j + \mathfrak{P}^1$.  These idempotents determine a
  splitting of $\mathfrak{t}$, and thus of $T$.  Write $T = T_1'
  \times T_2'$.  We claim that $Z^i_j(\bn) = (\bar{T}'_j)^i$.  Since
  $\bar{T}^i=Z^i(\bn)$, it suffices to show that $(\bar{T}'_j)^i$ maps
  to $Z_j^i (\bn)$.  When $i = 0$, this is clear.  In the case $i >
  0$, $\epsilon_j' \mathfrak{t} \cap \mathfrak{P}^i \subset \epsilon_j
  \mathfrak{P}^i \epsilon_j + \mathfrak{P}^{i+1}$.  Since $(T'_j\times
  1) \cap P^i = 1 + \epsilon_j' \mathfrak{t} \cap \mathfrak{P}^i$ and
  $\epsilon_j \mathfrak{P}^i \epsilon_j$ is the image of
  $\mathfrak{P}^i_j$ embedded as a diagonal block, we see that
  $(\bar{T}'_j)^i$ maps to $Z_j^i (\bn)$ as desired.

  Let $(V_1', V_2')$ be the splitting of $V$ determined by $V'_j =
  \epsilon'_j V$.  Let $p = \epsilon_1 \epsilon'_1 + \epsilon_2
  \epsilon'_2$, so $p(V_j') \subset V_j$.  The map $p$ induces the
  identity map on $\bar{L}^i$, and since the kernel of $p$ lies in
  $\cap_{i\in \Z} L^i = \{0\}$, we deduce that $p \in P^1$.  It is
  clear that $p T p^{-1}$ centralizes $(P, r, \beta)$ (indeed, this is true
  for any $p \in P^1$ by remark \ref{Tcentral}), and that $(V_1, V_2)$
  splits $pTp^{-1}$ into a product $T_1 \times T_2$.  Moreover, $y_\beta$
  is semisimple if and only if $y_{\beta_1}$ and $y_{\beta_2}$ are,
  since $y_\beta = t^r (\beta_{1 \nu} + \beta_{2 \nu})^{e_P} = t^r
  \beta_{1 \nu}^{e_P} + t^r \beta_{2 \nu}^{e_P} = y_{\beta_1} +
  y_{\beta_2}$.  The fact that $T_1$ centralizes $(P_1, r, \beta_1)$
  follows from the previous paragraph, so $(P_1, r, \beta_1)$ is
  regular.  The first part of Remark~\ref{Vsplit} implies that $P_2$
  is a uniform parahoric with $e_{P_2} = e_P$.  If $(P_2, r, \beta_2)$
  is fundamental, we conclude in the same way that it is regular and
  centralized by $T_2$.  Finally, if $P_2$ is non-fundamental, $(P_2,
  r, \beta_2)$ satisfies the conditions of Lemma~\ref{nonfund}.  Thus,
  $V_2$ has dimension $1$, and moreover $e_P = e_{P_2} = 1$.
\end{proof}

\begin{lemma}\label{dimstrata} If $(P, r, \beta)$ is a pure stratum, then
$e_P = n$. 
\end{lemma}
\begin{proof}
  Set $m=n/e_P$, and assume that $m>1$.  Let $T=E^\times$ be a torus
  centralizing $(P, r, \beta)$.  By Lemma~\ref{219}, we can find a
  saturation $\mathscr{L}_E=\{L_E^i\}$ of $\mathscr{L}$ that is
  stabilized by $\mathfrak{o}_E$.  We index $\mathscr{L}_E$ so that
  $L_E^{m i} = L^i$, and let $I$ be the Iwahori subgroup that
  stabilizes $\mathscr{L}_E$.  We fix a uniformizer $\varpi_E$ for
  $E$; we can assume that $\varpi_P=\varpi_E^m$.  Recall that
  $\mathfrak{I}_E^1 = \varpi_E \mathfrak{I} = \mathfrak{I} \varpi_E$
  by Proposition~\ref{uniformprop}.  Thus, $\varpi_E^j L^i = L_E^{m i
    + j}$, and furthermore $\varpi_E^j \in \mathfrak{P}^{\lfloor
    \frac{j}{m} \rfloor}$.  By Proposition~\ref{liecen}, $\ad
  (\varpi_E^j) (\beta_\nu) \in \mathfrak{P}^{-r + \lfloor \frac{j}{m}
    \rfloor +1}$.

First, we show that $\beta_\nu \in \fI^{-r m}$.  By
assumption,  $\beta_\nu (L_E^{im}) \subset L_E^{im - rm}$.  Now, take $0 < j
< m$.  We have 
\begin{equation*}
\beta_\nu (L_E^{m i + j}) = \beta_\nu \varpi_E^j (L^i) =
(\varpi_E^j \beta_\nu - \ad(\varpi_E^j) (\beta_\nu)) L^i \subset
L_E^{ im + j - rm},
\end{equation*}
since $\ad (\varpi_E^j) (\beta_\nu) L^i \subset L^{i-r+1} = L_E^{im+
  m-r m}$.  Thus, $\beta_\nu \in \fI^{- r m}$.

We next show that $\ad(\varpi_E)(\bn)\in\fI^{-rm+2}$.
The calculation above actually showed that
$\beta_\nu \varpi_E^j v \equiv \varpi_E^j \beta_\nu v \pmod{L_E^{-rm +j + 1}}$
for any $v \in L^0$ and $0\le j < m$.   In particular, this gives 
\begin{equation*}
\varpi_E \beta_\nu \varpi_E^j v \equiv \varpi_E^{j+1} \beta_\nu v
\equiv \beta_\nu \varpi_E \varpi_E^j v
 \pmod{L_E^{-rm + j + 2}},
\end{equation*}
for $0<j<m-1$, with the first congruence also holding for $j=m-1$.
However, the second congruence is also true when $j=m-1$; in this
case, $\varpi_E^{j+1}=\varpi_P$, and the congruence follows from $\ad
(\varpi_P) (\beta_\nu) L^0 \subset L^{-r + 2} = L_E^{-r m + 2
  m}\subset L_E^{-rm+(m-1)+2}$.  The congruences also hold trivially
for $j=0$, so $\ad (\varpi_E) \beta_\nu \in\fI^{-rm + 2}$ as desired.

By Lemma~\ref{adequation}, $\beta_\nu \in E + \fI^{-rm + 1}$.  Thus,
there exists $\beta'_\nu \in \fI^{- r m} \cap E$ such that
$\beta'_\nu \in \beta_\nu + \fI^{-r m+1}$.  Let
$\beta_{-rm} = \alpha \varpi_E^{-r m}$, with $\alpha \in k^\times$, be
the homogeneous degree $-rm$ term of $\beta'_\nu$, so that $X =
\beta_\nu - \beta_{-rm} \in \fI^{-r m +1}$.  Clearly, $\ad
(\beta_{-rm}) (\beta_\nu) = \ad (\beta_{-rm}) (X)$.  Moreover, since
$E$ centralizes $(P, r, \beta)$, the remark after
Proposition~\ref{liecen} shows that $\ad (\beta_{-rm}) (X) \in
\mathfrak{P}^{-2 r + 1}$.  It follows that $X$ and $\beta_{-rm}$
commute up to a term in $\mathfrak{P}^{-2 r + 1}$, so
\begin{equation}\label{binomial}
t^r \beta_\nu^{e_P} =  \alpha^{e_P} 1 + e_P t^r X \beta_{-rm}^{e_P -1}
+ \text{higher order terms}.
\end{equation}
If $X \in \fI^{-rm + j}$ for $1\le j<e_P$, then $e_P t^r X
\beta_{-rm}^{e_P -1} \in \fI^{j} $ and the higher order terms of
\eqref{binomial} lie in $\fI^{2 j}$.  In particular, if $X \notin
\mathfrak{P}^{-r+1}$, there exists $1 \le j < e_P$ such that
$X\in\fI^{-rm+j}\setminus\fI^{-rm+j+1}$, and it now follows that
$N=t^r\bn^{e_P}-\alpha^{e_P}\in\fI^{-rm+j}\setminus\fI^{-rm+j+1}$.  It
is obvious that $\bar{N}\in\bfP$ is a nonzero nilpotent operator, so
$y_\beta$ has Jordan decomposition $y_\beta= \overline{t^r
  \beta_\nu^{e_P}}=\alpha^{e_P}1+\bar{N}$.  This contradicts the
semisimplicity of $y_\beta$, so $X \in \mathfrak{P}^{-r+1}$.

On the other hand, if $X \in \mathfrak{P}^{-r+1}$, then $\bft^0=\fz^0
(\beta_\nu) = \fz^0(\beta_{-rm})$.  By Lemma~\ref{adequation}, $\fz^0
(\beta_{-rm})$ is one-dimensional if and only if $m=\gcd(-rm,n)=1$,
contradicting our assumption that $m>1$.  Hence, $m=1$ and $e_P=n$.
\end{proof}

\begin{proof}[Proof of Theorem~\ref{thm1}]
  First, assume that $(P, r, \beta)$ is strongly uniform.  Suppose
  that we have a nontrivial splitting $T=T_1\times T_2$, with
  corresponding idempotents $\epsilon_j$.  Setting $V_1 = V^{1 \times
    T_2}$ and $ V_2 = V^{T_1 \times 1}$, we show that $(V_1,V_2)$
  splits $P$ and $\beta$ at level $r$.  Note that $\epsilon_j \in
  \mathfrak{P}$, since $\epsilon_j \in \mathfrak{t} (\mathfrak{o})$.
  Therefore, $L^i_j = \epsilon_j L^i$, and $L^i = L^i_1 \oplus L^i_2$.
  By Remark~\ref{Vsplit}, in order to see that $(V_1, V_2)$ splits $P$
  and $\beta$, it suffices to show that $\epsilon_1 \beta_\nu
  \epsilon_2$ and $\epsilon_2 \beta_\nu \epsilon_1$ are in
  $\mathfrak{P}^{-r+1}$.

  Note that $\epsilon_j\in\Tfl$; indeed, it is a (nonempty) sum of the
  primitive idempotents $\chi_i$ for $\ft$.  By Lemma~\ref{410},
  $a_1\epsilon_1+a_2\epsilon_2\in Z^0(\bn)$ for any $a_1,a_2\in
  k^\times$.  This implies that
  $\Ad(a_1\epsilon_1+a_2\epsilon_2)(\epsilon_1\bn\epsilon_2)=\frac{a_1}{a_2}\epsilon_1
  \beta_\nu \epsilon_2\equiv \epsilon_1 \beta_\nu
  \epsilon_2\pmod\fP^{-r+1}$; accordingly, $\epsilon_1 \beta_\nu
  \epsilon_2\in\mathfrak{P}^{-r+1}$.  Similarly, $\epsilon_2 \beta_\nu
  \epsilon_2\in\mathfrak{P}^{-r+1}$.

  Let $(P_j, r, \beta_j)$ be the stratum corresponding to $V_j$.  By
  Remark~\ref{Vsplit}, each $(P_j, r, \beta_j)$ is strongly uniform.
  We next show that $(P_{1 2}, r, \partial_\beta)$ is uniform.  Using
  notation from the previous section, $\partial_{\beta}$ determines a
  map from $\bar{M}^j \to \bar{M}^{j-r}$.  It has already been
  established in \eqref{P12uniform} that $\mathscr{M} = \{M^j\}$ is
  uniform.  It
  remains to show that $\partial_{\beta}$ has trivial kernel in
  $\bar{M}^j$, say for all $j \ge 0$.  If $x \in \ker(\partial_{\beta})$, then
  $(\beta_1)_\nu x \equiv x (\beta_2)_\nu \pmod{M^{j-r+1}}$.
  Therefore, $\Ad (1+\iota_1 x\pi_2) (\beta_\nu) \equiv \beta_\nu
  \pmod{\mathfrak{P}^{j-r+1}}$, so $1+\iota_1 x\pi_2 \in Z^{j}(\beta_\nu)$.
  However, $(T_1 \times T_2) \cap (1 + \iota_1 M^j \pi_2) = 1$,
  implying that  $x\in M^{j+1}$.  We note that in the case $r = 0$, the
  eigenvalues of $\bar{\beta}_\nu$ are pairwise distinct modulo $\Z$
  by Definition~\ref{regstratum}.  A fortiori, the eigenvalues of
  $(\bar{\beta}_1)_\nu$ are distinct from the eigenvalues of
  $(\bar{\beta}_2)_\nu$ modulo $\Z$.  We have thus shown that
  $(P,r,\beta)$ is the direct sum of two strongly uniform
  strata, and Lemma~\ref{stronglyunif} shows that these strata are
  regular (centralized by the $T_j$'s).

We can iterate this procedure until  $(P,r,\beta)$ is the direct sum
of regular, strongly uniform strata each of which is centralized by a
rank one torus, i.e., by the units of a field.  Therefore, $(P, r, \beta)$
splits into a sum of pure strata.

Finally, suppose that $(P, r, \beta)$ is not strongly uniform.  When
$r > 0$, Corollary~\ref{nilsplitting} implies that $(P, r, \beta)$
splits into a strongly uniform stratum $(P_1, r, \beta_1)$ and a
non-fundamental stratum $(P_2, r, \beta_2)$.  By
Lemma~\ref{stronglyunif}, $V_2$ has dimension one and $e_P = 1$.  When
$r = 0$, Definition~\ref{regstratum} implies that the kernel of
$\bar{\beta}_\nu$ has dimension one and that the non-zero eigenvalues
of $\bar{\beta}_\nu$ are not integers.  Write $\bar{L}^0 = \bar{V}_1
\oplus \bar{V}_2$, where $\bar{V}_2 = \ker(\bar{\beta}_\nu)$ and
$\bar{V}_1$ is the span of the other eigenvectors.  It is easily
checked that any lift of this splitting to $L^0$ induces a splitting
$V = V_1 \oplus V_2$, and $(V_1, V_2)$ splits $(P, 0, \beta)$.
\end{proof}

\begin{proof}[Proof of Corollary~\ref{uniformsplitting}]
  When $(P, r, \beta)$ is strongly uniform, Theorem~\ref{thm1} states
  that $(P, r, \beta)$ splits into a sum of pure strata $(P_i, r,
  \beta_i)$ with $e_{P_i} = e_P$.  Therefore, by
  Lemma~\ref{dimstrata}, $(P_i, r, \beta_i)$ is centralized by a torus
  isomorphic to $E^\times$, and each component $V_i \subset V$ has
  dimension $e_P$.  It follows from Lemma~\ref{stronglyunif} that $T
  \cong (E^\times)^{n / e_P}$.  Otherwise, $e_P = 1$ and $(P, r,
  \beta)$ splits into a strongly uniform stratum $(P_1, r, \beta_1)$
  and a one-dimensional non-fundamental stratum $(P_2, r, \beta_2)$.
  In particular, by Lemma~\ref{stronglyunif}, this gives a splitting
  of a conjugate of $T$ into $T_1 \times T_2$, where $T_2 \cong
  F^\times$.  Since $e_P = 1$, it follows from the theorem that $T_1$
  also splits into rank one factors.
  
We now prove the last statement.  By Lemma~\ref{410}, we know that
$\tfl\cong\bft^0/\fn$, where $\fn$ is the image in $\bfP$ of $\prod
\fp_E$.  However, we have already seen that
$(\varpi_E,\dots,\varpi_E)$ generates $\fP^1$, so $\fn=\{0\}$.  The
proof for $\Tfl$ is similar.

\end{proof}

We can now prove Proposition~\ref{lemrss}.  First, we need a lemma.
\begin{lemma}\label{A1}
  Let $(P, r, \beta)$ be a regular stratum centralized by $T$, and
  suppose that $\bn\in \mathfrak{t} + \fP^{-r+m}$.  Then, $\bn$ is
  conjugate to an element of $\mathfrak{t}$ by an element of $P^m$.
\end{lemma}

\begin{proof}
  Let $E/F$ be a field extension of degree $e_P$, and let $\varpi_E$
  be a uniformizer in $E$.  By Corollary~\ref{uniformsplitting}, $(P,
  r, \beta)$ splits into a sum of pure strata $(P_i, r, \beta_i)$,
  each of which is centralized by a torus isomorphic to $E^\times.$ In
  particular, we can choose a block-diagonal representative
  $\beta'_\nu = (a_1 \varpi_E^{-r}, a_2 \varpi_E^{-r}, \ldots,
  a_{n/e_P} \varpi_{E}^{-r}) \in \prod_{i} \mathfrak{P}^{-r}_i$.
  Denote the summands of $V$ by $V_i$.  We may identify the
  $(\ell-j)^{th}$ off-diagonal block with $\Hom_F (V_j, V_\ell)$.  Let
  $\mathfrak{n} \subset \gl(V)$ be the subalgebra of matrices in the
  $(\ell-j)^{th}$ off-diagonal block corresponding to $\Hom_E(V_j,
  V_\ell)$.  If $a_\ell = a_j$, then $1 + \mathfrak{n}$ centralizes
  $\beta_\nu$.  Since $((1+\mathfrak{n}) \cap P^i)P^{i+1} \nsubseteq
  T^i P^{i+1}$, this is a contradiction.  Thus, the $a_j$'s
  are pairwise distinct.

    By Proposition~\ref{cores}, there exists $X_m \in
    \mathfrak{P}^{m}$ such that $\bn - \pi_{\mathfrak{t}} (\bn) \equiv
    \ad (\beta'_\nu) (X_m) \equiv \ad (\bn) (X_m)
    \pmod{\mathfrak{P}^{2 - r}}$.  Therefore, $\Ad(1+X_m) (\bn) \equiv
    \pi_{\mathfrak{t}}(\bn) \pmod{\mathfrak{P}^{2-r}}$.  Inductively,
    we can find $X_j\in\fP^{j+1}$ so that, setting $p_j = (1 + X_j)
    (1+ X_{j-1}) \ldots (1+ X_m)$, $\Ad(p_j)(\bn) \in \mathfrak{t} +
    \mathfrak{P}^{j+1-r}$ and $p_{j} \equiv p_{j-1} \pmod{P^{j-1}}$.
    If we let $p\in P^1$ be the inductive limit of the $p_j$'s, we see
    that $\Ad(p)(\bn)\in\ft$.
\end{proof}
\begin{rmk}
We note that, in the argument above, it is not necessarily the case that $\bn$
is conjugate to $\pi_{\mathfrak{t}} (\bn)$.
\end{rmk}

\begin{proof}[Proof of Proposition~\ref{lemrss}]  
It was shown in the proof of Lemma~\ref{nonfund} that
  any representative $\bn$ is regular.  To show that $\bn$ is also
  semisimple, it suffices to show that it is conjugate to an element
  of a Cartan subalgebra $\ft$.

  First, suppose $(P, r, \beta)$ is strongly uniform. By
  Theorem~\ref{thm1} and Corollary~\ref{uniformsplitting}, there
  exists a splitting $V=V_1\oplus\dots\oplus V_{n/e_P}$ with $\dim
  V_i=e_P$ for each $i$ and a block diagonal $\beta'_\nu =
  (\beta_{i\nu}) \in \prod_{i = 1}^{n/e_P} \gl(V_i)$ such that
  $\beta_\nu \in \beta'_\nu+ \mathfrak{P}^{-r+1}$ and the $(P_i, r,
  \beta_i)$'s are pure strata.  Moreover, by Lemma~\ref{stronglyunif},
  we can choose a maximal torus $T$ centralizing $(P, r, \beta)$ such
  that the splitting of $V$ induces a splitting
  $T=T_1\times\dots\times T_{n/e_P}$, with $T_i$ centralizing $(P_i, r,
  \beta_i)$.  Since $T_i$ is isomorphic to the units of the field
  extension $E/F$ of degree $e_P$, Lemma~\ref{adequation} implies that
  $\beta_{i\nu} \in T_i\cap \mathfrak{P}_i^{-r} +
  \mathfrak{P}_i^{1-r}$.  Therefore, $\beta_\nu \in \mathfrak{t} +
  \mathfrak{P}^{1-r}$.  By Lemma~\ref{A1}, $\beta_\nu$ is conjugate to
  an element of $\mathfrak{t}.$

If $(P, r, \beta)$ is not strongly uniform, then $e_P = 1 $ by the
second part of Theorem~\ref{thm1}.  As above, we can choose a
splitting $V=V_1\oplus\dots \oplus V_n$, a diagonal representative
$\beta'_\nu\in \prod_{i=1}^n\gl(V_i)$, and a compatibly split torus
$T$ which centralizes $(P, r, \beta)$.  In this case, $\dim V_i=1$ for
all $i$, so $\ft=\prod_{i=1}^n\gl(V_i)$ and $\beta'_\nu\in\ft$.  In
particular, $\beta_\nu\in\mathfrak{t}^{-r} +
\mathfrak{P}^{1-r}$, and Lemma~\ref{A1} again implies that $\beta_\nu$
is conjugate to an element of $\mathfrak{t}$.
\end{proof}

We conclude this section with two lemmas that will be needed in
Section~\ref{modspace}.  We recall from Remark~\ref{repint} that if
$(P, r, \beta)$ is a regular stratum centralized by the maximal torus
$T$, then one can choose $\bn\in\ft$.  We next show that if two such
representatives are conjugate, then they are the same.
\begin{lemma}\label{A0}
  Suppose that $(P, r, \beta)$ is a regular stratum.  Choose
  representatives $\bn, \bn' \in \ft$ for $\beta$.  If $\Ad(g) (\bn) =
  \bn'$ for some $g \in \GL(V)$, then $\bn' = \bn$.\end{lemma}
\begin{proof}

By Proposition~\ref{lemrss}, $\beta_\nu$ is regular semisimple.  Since
$\Ad (g) (\beta_\nu) \in \mathfrak{t}$, $g$ lies in the normalizer of
$T$.  Let $w$ be the image of  $g$ in the relative Weyl
group $W=N(T)/T$.  It suffices to show that $w$ is the identity.

First, we show that $\Ad(g) (\ft^i) \subset \mathfrak{t}^i$.  Recall from Corollary~\ref{uniformsplitting} that $\ft \cong
\prod_{i = 1}^{n/e_P} E$, so $\ft$ splits over $E$. Let $\omega_j$ be
a uniformizer of $E$ supported on the $j^{th}$ summand of $\ft$ as
before; we have $\omega_j \fP_j = \fP_j^1$ in the splitting determined
by Theorem~\ref{thm1}.  Therefore, $\ft^i $ consists of those
$(x_j) \in \prod_{i=1}^{n/e_P} E$ such that $x_j \in \omega_j^i
\fo_E$.  These are precisely the $F$-rational points of $\ft_E$ with
eigenvalues of degree at least $i/e_P$.  Since the action of $W$
permutes the eigenvalues, it follows that $\Ad(g) (\ft^i)
\subset (\ft^i)$.

Suppose that $s \in \ft^{1-r}$.  The previous paragraph shows that
$\Ad(g) (\bn + s) = \bn' + \Ad(g) (s) \in \bn + \ft^{1-r}$.  By
induction on $i$, $\Ad(g^i) (\bn) \in \bn + \ft^{1-r}$, so
$\Ad(g^i)(\bn)$ is a representative for $\beta$.  Let $m$ be the order
of $w$.  Then, $\bn'' = \frac{1}{m} (\sum_{i = 0}^{m-1}
\Ad(g^i)(\bn))$ is a representative for $\beta$ fixed by the action of
$w$.  Since $\bn''$ is regular semisimple, $w$ must be the identity.

\end{proof}

\begin{lemma}\label{isotropy}
  Suppose that $(P, r, \beta)$ is a regular stratum centralized by
  $T$ and that $\bn\in\ft^{-r}$.  Let $A\in\gl(V)^\vee$ be the
  functional determined by $\bn$ and $\nu$.
  Then, $A$ determines an element $A_i \in (\mathfrak{P}^i)^\vee$ by
  restriction, and the stabilizer of $A_i$ under the coadjoint action
  of $P^i$ is given by $T^i P^{r+1-i}$ whenever $r \ge 2 i$,
  and $P^i$ whenever $r < 2 i$.
\end{lemma}
\begin{proof}

  Recall from Proposition~\ref{duality} that $(\mathfrak{P}^i)^\perp =
  \mathfrak{P}^{1-i}$.  Thus, $\Ad^*(p) (A_i) = A_i$ if and only if
  $\Ad (p)(\bn) \in \bn + \mathfrak{P}^{1-i}$.  If $r < 2 i$, then $-r
  + i \ge -i+1$.  Therefore, since $\Ad(p) (\bn)\in\bn +
  \mathfrak{P}^{-r+i}$ for any $p \in P^i$, $P^i$ lies in the
  stabilizer of $A_i$ in this case.

Suppose now that $\Ad^*(p) (A_i) = A_i$ and $r \ge 2 i$.  
The image of  $p$ in $\bar{P}^i$ must lie in
$Z^i (\bn)$; therefore $p  = t p' \in T^i P^{i+1}$.
Assume, inductively, that $p = t q \in T^i P^{j}$ with $j<r+1-i$.

Since $\Ad (t^{-1}) (\bn) = \bn$, $q$ stabilizes $A_i$.  Moreover,
$\Ad (P^j) (\bn) \subset \bn + \mathfrak{P}^{j-r};$ since $j - r <
1-i$, the image of $q$ in $\bar{P}^j$ lies in $Z^j (\bn)$.  Therefore,
$q \in T^j P^{j+1}\subset T^i P^{j+1}$.  We conclude
that $p\in T^i P^{r+1-i}$.  When $j = r+1-i$, $\Ad(P^j) (\bn)
\subset \bn + \mathfrak{P}^{1-i}$, so $P^j$ stabilizes $A_i$.  It is
now clear that $T^i P^{r+1-i}$ stabilizes $A_i$.
\end{proof}

Note that although the functionals $A_i$ depend on the choice of
$\bn\in\ft^{-r}$, their stabilizers do not.

\section{Connections and Strata}\label{strataconnections}

In this section, we describe how to associate a stratum to a formal
connection. 
The local theory of irregular singular point connections is well understood; 
an elegant classification is given in \cite[Theorem III.1.2]{Mal}. 
The geometric theory of strata provides a Lie-theoretic interpretation of elements
in the classical theory.
In particular, the combinatorics of fundamental strata may be used to
determine the slope of a connection, and the theory of
strata makes precise the notion of the leading term of a connection
with noninteger slope.  Moreover, a split stratum induces a splitting
on the level of formal connections.  



First, we recall some notation and basic
facts.  As before, $k$ is an algebraically closed field of
characteristic $0$, $\mathfrak{o}$ is the ring of formal power series
in a parameter $t$, and $F$ is the field of formal Laurent series.
\begin{enumerate}
\item $\diff_F$ (resp. $\diff_{\mathfrak{o}}$) is the ring of formal differential operators
on $F$ (resp. $\mathfrak{o}$).  $\diff_F$ is generated as an $F$-algebra
by $\partial_t = \frac{d}{dt}$ and contains the Lie algebra of
$k$-derivations (i.e., vector fields) on $F$:
$\Der_k(F) = F \partial_t$.
\item $\Omega^\times \subset \Omega^1_{F/k}$ is the $F^\times$-torsor
of non-zero one forms in $\Omega^1_{F/k}$; if $\omega, \nu \in \Omega^\times$,
then $\frac{\omega}{\nu} \in F^\times$ is the unique element such that
$\frac{\omega}{\nu} \nu = \omega$.
\item If $\nu \in \Omega^\times$, there is a unique vector field
$\tau_\nu \in \diff_F$ whose inner derivation takes $\nu$ to $1$, i.e.,
$\iota_{\tau_\nu} (\nu) = 1$.  For example, $\tau_{\frac{dt}{t}} = t \partial_t$,
and $\tau_{d f} = \frac{dt}{df} \partial$.
\item A connection $\nabla$ on an $F$-vector space $V$ is a $k$-linear
  derivation
\begin{equation*}
\nabla : V \to V \otimes_F \Omega^1_{F/k}.
\end{equation*}
The connection $\nabla$ gives $V$ the structure of a $\diff_F$-module:
if $v \in V$, and $\xi \in \Der_k(F)$, then $\xi (v) = \nabla_\xi
(v)\overset{\mathrm{def}}{=} \iota_{\xi} (\nabla (v))$.
\item A connection $\nabla$ is \emph{regular singular} if there exists
  an $\mathfrak{o}$-lattice $L \subset V$ with the property that
  $\nabla(L) \subset L \otimes_{\mathfrak{o}}
  \Omega^1_{\mathfrak{o}/k} (1)$.  Equivalently, if $\nu \in
  \Omega^\times$ has order $-1$, $\nabla_\tau (L) \subset L$.
  Otherwise, $\nabla$ is irregular.
\item Suppose $V$ has dimension $n$.  Let $V^\triv = F^n$ be the
  trivial vector space with standard basis.  If we fix a
  trivialization $\phi : V \overset{\sim}{\to} V^\triv$, then $\nabla$
  has the matrix presentation
\begin{equation}\label{gauge}
\nabla  = d +  [\nabla]_\phi
\end{equation}
where $[\nabla]_\phi \in \gl_n (F) \otimes_{F}\Omega^1_{F/k}$.  The
space of trivializations is a left $\GL_n(F)$-torsor, and
multiplication by $g$ changes the matrix of $[\nabla]_\phi$ by the
usual gauge change formula
\begin{equation}\label{gaugechange}
  g \cdot [\nabla]_\phi :=  g [\nabla]_\phi g^{-1} - (d g ) g^{-1}.
\end{equation}
Thus, $[\nabla]_{g \phi} = g \cdot [\nabla]$.
We note that the matrix  form of $\nabla_\tau$ is given by
$[\nabla_\tau]_\phi = \iota_\tau ([\nabla]_\phi)$, with gauge change formula
$g \cdot [\nabla_\tau]_\phi := g [\nabla_\tau]_\phi g^{-1} - (\tau g) g^{-1}$.
We will drop the subscript $\phi$ whenever there is no ambiguity about
the trivialization.
\end{enumerate}
\subsection{Strata contained in connections}\label{strataex}

Let $\nabla$ be a connection on an $n$-dimensional $F$-vector space.
Fix $\nu\in \Omega^\times$, and set $\tau = \tau_\nu$.  Suppose that
$\mathscr{L}$ is a lattice chain with the property that $\nabla_\tau
(L^i) \subset L^{i-r-(1+\ord(\nu))e_P}$.  We define
$\gr^i(\nabla_\tau)$ to be the following map induced by $\nabla_\tau$:
\begin{equation*}\label{grdef}
\gr^i (\nabla_\tau) : \bigoplus_{j = 0}^{e_P-1} \bar{L}^{i+j} \to 
\bigoplus_{j = 0}^{e_P-1} \bar{L}^{i+j-r-(1+\ord(\nu))e_P}.
\end{equation*}
By Lemma~\ref{subquotient}, this determines an element of
$\bfP^{-r-(1+\ord(\nu))e_P}$.  Equivalently, if we fix a
trivialization $\phi$ of $L^i$ (viewed as a trivialization of $V$
taking $L^i$ to $\fo^n$), then $\gr^i(\nabla_\tau)$ equals
the image of $\phi^{-1}[\nabla_\tau]\phi$ in
$\bfP^{-r-(1+\ord(\nu))e_P}$.

\begin{definition}\label{stratumcontainment}
We say that $(V, \nabla)$ contains the stratum $(P, r, \beta)$ if $P$ stabilizes
the lattice chain $(L^i)$,
$\nabla_\tau (L^i) \subset L^{i-r-(1+\ord(\nu))e_P}$ for all $i$, and
$\gr^j (\nabla_\tau) =\bbn\in \bfP^{-r-(1+\ord(\nu))e_P}$
for some $j$.
\end{definition}

\begin{proposition}\label{stratumindep}
The stratum $(P, r, \beta)$ is independent of $\nu \in \Omega^\times$.
\end{proposition}
\begin{proof}
  Take $\nu' = f \nu$ for some $f \in F^\times$, so that $\tau' =
  \tau_{\nu'} = \frac{1}{f} \tau_\nu$.  Since
  $\nabla_{\tau'}=\frac{1}{f} \nabla_\tau$, it is clear that
  $\nabla_{\tau'}(L^i)\subset L^{i-r-(1+\ord(\nu'))e_P}$ if and only
  if the analogous inclusion holds for $\nu$; in addition,
  $\gr^i(\nabla_{\tau'}) = \frac{1}{f}\gr^i \left(
    \nabla_\tau\right)$.  On the other hand, if $\bn\in
  \fP^{-r-(1+\ord(\nu))e_P}$ is a representative for the functional
  $\beta$ on $\bfP^r$, then $\langle \beta_\nu, X \rangle_\nu =
  \langle f^{-1} \beta_\nu, X \rangle_{\nu'}$ for all $X\in\fP^r$
  implies that one can take $\beta_{\nu'} = \frac{1}{f} \beta_\nu$.
  Hence, $\gr^i(\nabla_{\tau'}) =\bar{\beta_{\nu'}}$.
\end{proof}

For the rest of Section~\ref{strataconnections}, 
we will fix $\nu\in \Omega^\times$ of order $-1$.

\begin{proposition}\label{grindep}
  Suppose that $r \ge 1$.  Then, the coset in $\bfP^{-r}$ determined
  by $\gr^\ell (\nabla_\tau)$ under the isomorphism \eqref{priso}
  is independent of $\ell$.
\end{proposition}
\begin{proof} The maps $\gr^\ell (\nabla_\tau)$ and $\gr^0
  (\nabla_\tau)$ determine the same element on $\bfP^{-r}$ when they
  ``coincide up to homothety''.  More precisely, we must show that if
  $0\le i < e_P$ and $\ell\le i+je_P< \ell +e_P$, then
  $\nabla_\tau(t^j v)\equiv t^j \nabla_\tau(v)\pmod {L^{i+je_P+1}}$
  for all $v\in L^i$.  By the Leibniz rule, $\nabla_\tau (t^j v) =
  \frac{\tau (t^j)}{t^j} (t^j v) + t^j \nabla_\tau (v)$.  However,
  $\frac{\tau (t^j)}{t^j} \in \mathfrak{o}$, so for $r\ge 1$,
  $\frac{\tau (t^j)}{t^j} (t^j v)\in L^{\ell+je_P}\subset
  L^{\ell+je_P-r+1}$ as desired.
\end{proof}

In other words, if $r\ge 1$, the $\gr^\ell(\nabla_\tau)$'s determine a
unique coset $\gr(\nabla_\tau) \in \bar{\fP}^{-r}$; viewed as a degree
$-r$ endomorphism of $\gr(\mathscr{L})$,
$\gr(\nabla_\tau)(\bar{x})=\gr^\ell(\nabla_\tau)(\bar{x})$ for
$x\in\bL^i$ and any $\ell$ with $\ell\le i < \ell+e_P$.

The following lemma will be used in Sections~\ref{formaltypes}
and \ref{modspace}.
\begin{lemma}\label{taup1}
  If $P \subset \GL_n(\fo)$ is a parahoric subgroup, then
  $\tau(\fP^\ell)\subset \fP^\ell$.  Moreover, if $p \in P$, then
  $\tau(p) p^{-1} \in \mathfrak{P}^1$.
\end{lemma}
\begin{proof}

Let $\sL$ be the lattice chain stabilized by $P$ such that $L^0 = \fo^n$.  
We may choose a  basis ${e_1, \ldots, e_n}$ for $\fo^n$
that is compatible with $\sL$ (as in Remark~\ref{compatiblebasis}).  It is clear that
$\tau (L^j) \subset L^j$ for any $j$, since $\tau (e_i) \subset t \fo^n$.  Therefore,
if $v \in L^j$ and $X \in \fP^\ell$, 
$\tau (X) v = \tau (X v)- X \tau(v) \in \fP^{j+\ell}$.  It follows
that $\tau (X) \in \fP^\ell$.

By Lemma~\ref{subquotient}, there exists $H\subset\GL_n(k)$ for which
$P=P^1\ltimes H$.  Hence, it suffices to prove the second statement for $p\in P^1$. Since $P^1$ is topologically unipotent, 
$\exp: \mathfrak{P}^1 \to P^1$ is surjective.  
If $p= \exp(X)$, we obtain 
\begin{equation*}
\tau (\exp(X)) \exp(-X) = \tau(X) \exp(X) \exp(-X) = \tau(X)\in\fP^1.
\end{equation*}

\end{proof}



We now investigate the strata contained in a connection.

\begin{lemma}\label{contstrat}
Every connection $(V, \nabla)$ contains a stratum.
\end{lemma}
\begin{proof}
  Take any lattice $L \subset V$ with stabilizer $P$, and let
  $\mathscr{L}$ be the lattice chain $(L^i=t^i L)$.  It is obvious that $\nabla_\tau (L)
  \subset L^{-r}$ for some $r \ge 0$.  The Leibniz rule calculation
  from the proof of Proposition~\ref{grindep} shows that $\nabla_\tau
  (L^i) \subset L^{i-r}$, so $(V,\nabla)$ contains $(P,r,\beta)$,
  where $\beta$ corresponds to $\gr^0(\nabla_\tau)$.
\end{proof}

One of the standard ways to study irregular singular connections
is to find a good lattice pair.  The theory of good lattice pairs
bears a superficial resemblance to the theory of fundamental strata.  
However, we will see that there are only a few cases in which there is a 
direct relationship between the two theories. 
\begin{lemma}\cite[Lemme II.6.21]{De} \label{goodlattice}
  Given a connection $(V,\nabla)$, there exist two
  $\mathfrak{o}$-lattices $M^1 \subset M^2 \subset V$ with the
  following properties:
\begin{enumerate}
\item $\nabla (M^1) \subset M^2 \otimes_{\mathfrak{o}} \Omega^1_{\mathfrak{o}/k} (1) $
\item
For all $\ell>0$, $\nabla$ induces an isomorphism
\begin{equation}\label{grisom} 
  \bar{\nabla} : M^1 (\ell) / M^1(\ell-1)\cong  \left[M^2
    \otimes_{\mathfrak{o}} \Omega^1_{\mathfrak{o}/k} (\ell+1) \right]
  / \left[M^2 \otimes_{\mathfrak{o}} \Omega^1_{\mathfrak{o}/k} (\ell)
  \right].\end{equation} The connection $\nabla$ is regular singular
if and only if $M^1 = M^2$.
\end{enumerate}
\end{lemma}
We call $M^1$ and $M^2$ a \emph{good lattice pair} for $(V, \nabla)$.
\begin{rmk}
  In the regular singular case, the data of a good lattice pair
  ($\bar{\nabla}$, $M^1$, $M^2$) is
  equivalent to a strongly uniform stratum contained in $(V, \nabla)$:
  take $\mathscr{L} = (t^i M^1)_{i \in \Z}$, and $\beta$ such that the
  image of $\bar{\beta}_\nu$ under the appropriate isomorphism
  \eqref{priso} is $\iota_\tau (\bar{ \nabla}) = \gr^{-1} (\nt) $.  However, it is not
  immediately possible to construct a fundamental stratum from a good
  lattice pair in general.

  Given a good lattice pair $M^1$ and $M^2$, one might naively
  construct a lattice chain $\mathscr{L}$ as follows.  Set $L^0 =
  M^1$.  Choose $s\in \mathbb{Z}_{\ge 0}$ such that $ L^0 (s) \supset
  M^2$ but $L^0 (s-1) \nsupseteq M^2$.  First, we suppose that $M^2 =
  L^0 (s)$.  Define $\mathscr{L}$ to be the chain $(L^i = t^i L^0)$.
  In this case, $e_P = 1$.  Take $\beta$ such that $\bbn = \gr^{-1}
  (\nt)$ as above.  Equation \eqref{grisom} implies that $(V, \nabla)$
  contains $(\GL(L^0), s, \beta)$, and this stratum is fundamental (in
  fact, strongly uniform).

  The naive generalization of the construction above does not
  necessarily produce a fundamental stratum.  Set $L^{1} = M^2 (-s) +
  L^0 (-1)$.  Since $L^0 (s) \supsetneq M^2$, it follows using
  Nakayama's Lemma that the map $M^2 (-s) \to L^0/L^0 (-1)$ is not
  surjective.  We conclude that $L^0 \supsetneq L^{1} \supsetneq L^0
  (-1)$.  This extends to a lattice chain $\mathscr{L}$ with $e_P =
  2$.

Finally, it is clear that there exists a minimal $r\ge 0$ such that
$\nabla_\tau (L^i) \subset L^{i-r}$ for $i=0,1$.  The usual Leibniz
rule argument shows that
$\nabla_\tau (L^i) \subset L^{i-r}$ for all $i$.  Choosing $\beta_\nu
\in \mathfrak{P}^{-r}$ whose coset corresponds to $\gr^j
(\nabla_\tau)$ (for some fixed $j$), we have the data necessary to
give a stratum contained in $(V,\nabla)$.

Notice that the stratum constructed above is not necessarily
fundamental.  For instance, suppose that
\begin{equation*}
[\nabla_\tau] = \begin{pmatrix}
0&  t^{-3} \\ 1 & 0 
\end{pmatrix} 
\end{equation*}
in $V = V^{\triv}$.  Set $M^1 = \mathfrak{o} e_1 + \mathfrak{o} e_2$
and $M^2 = \mathfrak{p}^{-3} e_1 + \mathfrak{o} e_2$.  It is easy to
check that this is a good lattice pair for $(V,\nabla)$, and our
construction gives $L^0 = M^1$ and $L^{1} = \mathfrak{o} e_1 +
\mathfrak{p} e_2$.  However, the coset in $\bfP^{-5}$ corresponding to
$\gr^0(\nabla_\tau)$ contains the nilpotent operator
$\left(\begin{smallmatrix} 0 & t^{-3} \\ 0 & 0
  \end{smallmatrix}\right)$.
\end{rmk}

In Theorem~\ref{theorem:Bu1}, we showed that a stratum is fundamental
if and only if it can not be reduced to a stratum with smaller slope.
The main goal of this section is to show that the slope of a
connection is the same as the slope of any fundamental stratum
contained in it.  First, we define the slope of a connection.

Fix a lattice $L \subset V$.  If $\bfe = \{e_j\}$ is a finite
collection of vectors in $V$, we define $v(\bfe) = m$ if $m$ is the
greatest integer such that $\bfe \subset t^m L$.  Take $\bfe$ to be a
basis for $V$.  An irregular connection $(V, \nabla)$ has slope
$\sigma$, for $\sigma$ a positive rational number, if the subset of
$\Q$ given by
\begin{equation*}
\{ \left|(\nabla_\tau^i \bfe) + \sigma i \right| \mid i > 0\}
\end{equation*}
is bounded.  Here, $\nt^i \bfe = \{ \nt^i (e_j)\}$.  By a theorem of
Katz \cite[Theorem~II.1.9]{De}, every irregular singular connection
has a unique slope, and the slope is independent both of the choice of
basis $\bfe$ and the choice of $\nu$ of order $-1$.  We define the
slope of a regular singular connection to be $0$.

\begin{lemma}\label{slope0} A connection $(V,\nabla)$ contains a stratum of slope
  $0$ if and only if it is regular singular, i.e., has slope $0$.
  Moreover, a regular singular connection contains a fundamental
  stratum with slope $0$ and no fundamental stratum with positive
  slope.
\end{lemma}
\begin{proof}  

  If $(V, \nabla)$ contains $(P, 0, \beta)$, then there exists a
  lattice $L \in \mathscr{L}$ such that $\nabla_\tau (L) \subset L$.
  Thus, $\nabla$ is regular singular.  Now, suppose that $(V, \nabla)$
  is regular singular.  Lemma~\ref{goodlattice} gives a lattice $L =
  M^1 = M^2$ for which $\nabla_\tau (L) \subset L$.  Thus,
  $\nabla_\tau$ preserves the lattice chain $(L^i = t^i L)$, and the
  corresponding stratum has $r = 0$.  Moreover, \eqref{grisom} implies
  that $\gr^{-1}(\nabla_\tau) (L^{-1}/L^0) = L^{-1}/L^0$, so the
  stratum is fundamental.

  Suppose that the same connection $(V, \nabla)$ contains a stratum
  $(P, r, \beta)$ with $r > 0$.  Let $(L^i)$ be the associated lattice
  chain.  We will show that $\bar{\beta}\in\bfP^{-r}$ is a nilpotent
  operator on $\gr(\mathscr{L})$.  With the lattice $L$ as above,
  choose $i$ and $m>0$ such that $L^{i-m r+ e_P} \supset L \supset
  L^i$.  For $0 \le j <e_{P}$, we obtain $(\nabla_\tau)^m
  (L^{i+j})\subset L\subset L^{i+j-mr+1}$.  By
  Proposition~\ref{grindep}, the endomorphism $\bbn$ coincides with
  each $\gr^\ell(\nabla_\tau)$ on the latter's domain, so
  $\bbn^m(\bar{L}^{i+j})=\gr^{i-(m-1)r}(\nabla_\tau)\circ\dots\circ\gr^i(\nabla_\tau)(\bar{L}^{i+j})=0$.
  Therefore, $(P, r, \beta)$ is not fundamental.
\end{proof}

The following corollary describes the relationship between a
fundamental stratum contained in $(V, \nabla)$ and its slope.
\begin{proposition}\label{slope}
  If $(V, \nabla)$ contains the fundamental stratum $(P, r, \beta)$,
  then $\slope(\nabla)=r/e_P$.
\end{proposition}
\begin{proof}
  When $(V, \nabla)$ is regular singular, $r/e_P = 0=\slope(\nabla)$
  by Lemma~\ref{slope0}, so we assume that $(V, \nabla)$ is irregular
  singular.  Let $L = L^0$.  Choose an ordered basis $\bfe$ for $L$
  that is compatible with $\mathscr{L}$.  We use the notation from the
  proof of Lemma~\ref{subquotient}: $\bfe =
  \bigsqcup_{j=0}^{e_P-1}\bfe_j$ with $W_j = \spa(\bfe_j) \subset L^j$
  naturally isomorphic to $L^j/L^{j+1}$.  For $w \in W_j$, $(\nt)^i(w)
  \in \bn^i (w) + L^{-i r +j+1}$ by repeated application of
  Proposition~\ref{grindep}.  Since $(P, r, \beta)$ is fundamental,
  $\bn^i \notin \fP^{-r i+1}$.  Therefore, $\bn^i (\bfe_j) \nsubseteq
  L^{-ri + j + 1}$ for some $\bfe_j$.  It follows that $\nt^i \bfe_j
  \nsubseteq L^{-ri +j+1}$, and thus $\nt^i \bfe \nsubseteq L^{-r i
    +e_P}$.

  Let $\lceil \frac{r i}{e_P} \rceil$ (resp. $\lfloor \frac{r i}{e_P}
  \rfloor$) be the integer ceiling (resp. floor) of $\frac{r i}{e_P}$.
  Then, $\nt^i \bfe \subseteq L^{-r i} \subseteq t^{-\lceil
    \frac{ri}{e_P} \rceil} L^0$.  However, $\nt^i \bfe \nsubseteq
  t^{-\lfloor \frac{ri}{e_P} \rfloor +1} L^0$, since $t^{-\lfloor
    \frac{ri}{e_P} \rfloor +1} L^0 \subseteq L^{-r i+e_P}$.  We
  conclude that $v(\nt^i \bfe)$ is equal to either $-\lfloor \frac{r
    i}{e_P} \rfloor$ or $-\lceil \frac{r i}{e_P} \rceil$.  In
  particular, $\left| v(\nt^i \bfe) + \frac{r i}{e_P} \right| < 1$,
  which proves the proposition.
    \end{proof}

    We are now ready to state our main theorem on the relationship
    between slopes of connections and fundamental strata.  In the
    context of the representation theory of local fields, the
    analogous theorem is due to Bushnell \cite[Theorem 2]{Bu}.

\begin{theorem}\label{fundcontain}  Any stratum $(P,
  r, \beta)$ contained in $(V,\nabla)$ has slope greater or equal to
  $\slope(\nabla)$.  Moreover, the set of strata contained in $\nabla$
  with slope equal to $\slope(\nabla)$ is nonempty and consists
  precisely of the fundamental strata contained in $\nabla$.  In
  particular, every connection contains a fundamental stratum.
\end{theorem}

\begin{proof}
  The regular singular case is dealt with in Lemma~\ref{slope0}, so assume
  that $ (V, \nabla)$ is irregular singular.  By
  Proposition~\ref{slope}, any fundamental stratum contained in
  $\nabla$ has slope equal to $\slope(\nabla)$.  Now, let $(P_0, r_0,
  \beta_0)$ be any stratum contained in $\nabla$.  Assume that this
  stratum is not fundamental.  We show that the stratum has a
  reduction $(P_1, r_1, \beta_1)$ with strictly smaller slope which is
  also contained in $\nabla$.

  By Theorem~\ref{theorem:Bu1}, there is a reduction of $(P_0, r_0,
  \beta_0)$ to $(P_1, r_1, \beta_1')$ such that $r_1/e_{P_1} < r_0/
  e_{P_0}$.  Let $\mathscr{L}_0 = (L_0^i)$ and $\mathscr{L}_1=
  (L_1^i)$ be the lattice chains corresponding to $P_0$ and $P_1$
  respectively.  By definition, there is a lattice $L \in
  \mathscr{L}_0 \cap \mathscr{L}_1$; reindexing the lattice chains if
  necessary, we can assume without loss of generality that $L = L_0^0
  = L_1^0$.

  Choose a basis $\{e_1, \ldots, e_n\}$ for $L$, and write
  $\nabla_\tau (v) = [\nabla_\tau] (v) + \tau (v)$.  In particular, if
  $v \in L$, then $\tau(v) \in t L$, since $\tau (f e_j) = tf' e_j\in
  t L$ for $f\in\fo$.  Thus, for any $v \in L_0^i$, $0 \le i <
  e_{P_0}$ (resp. $w \in L_1^\ell$, $0 \le \ell < e_{P_1}$),
  $\nabla_\tau(v) - [\nabla_\tau](v) \in L_0^{e_{P_0}}$ (resp.
  $\nabla_\tau (w) - [\nabla_\tau](w) \in L_1^{e_{P_1}}$).  Therefore,
  $[\nabla_\tau]$ is a representative for both $\gr_0^0 (\nabla_\tau)$
  and $\gr_1^0 (\nabla_\tau)$.  By Proposition~\ref{grindep}, we
  conclude that $[\nabla_\tau]$ is a representative for $\beta_0$.

  By definition of a reduction, $[\nabla_\tau] \in \fP_1^{-r_1}$.  If
  we let $\beta_1\in(\bfP_1^r)^\vee$ be the corresponding functional,
  it is immediate that $(P_1, r_1, \beta_1)$ is also a reduction of $(P_0,
  r_0, \beta_0)$ with strictly smaller slope.  

Finally, consider the collection of all strata contained in $\nabla$; this
set is nonempty by Lemma~\ref{contstrat}.  Since $e_P\le n$ (where
$n=\dim V$, these slopes are all contained in  $\frac{1}{n!}
\mathbb{Z}$.  Thus, the set of these slopes has a minimum value $s>0$.
(This value is positive by Lemma~\ref{slope0}.)  The argument given above
shows that any stratum with slope $s$ is fundamental, so $\nabla$
contains a fundamental stratum and $s=\slope(\nabla)$.

\end{proof}

\subsection{Splittings}

Recall that the connection $(V,\nabla)$ is split by the direct sum
decomposition $V=V_1\oplus V_2$ if $\nt(V_1)\subset V_1$ and
$\nt(V_2)\subset V_2$.  In this section, we will show that any time a
connection $(V, \nabla)$ contains a fundamental stratum $(P, r,
\beta)$ that splits, then there is an associated splitting of the
connection itself.  We note that, in the language of flat
$\GL_n$-bundles, this corresponds to a reduction of structure of $(V,
\nabla)$ to a Levi subgroup.
\begin{lemma}\label{case0}
Suppose that $(V,\nabla)$ contains a fundamental stratum $(P, 0, \beta)$ that is split
by $(V_1,V_2)$.  Then, there exists $q\in P^1$ and a fundamental
stratum $(P', 0, \beta')$ contained in $\nabla$ such that $e_{P'}=1$,
$P'^1\subset P^1$, and $(qV_1,qV_2)$ splits both strata
\end{lemma}

\begin{proof}

  Let $\mathscr{L}$ be the lattice chain stabilized by $P$, and,
  without loss of generality assume that $\beta$ is determined by
  $\gr^0 (\nabla_\tau)$.  Choose a trivialization for $L^0$, and let
  $[\nt]\in\fP$ be the corresponding $F$-endomorphism.  Now let
  $\mathscr{L}' = (L'_i=t^i L^0)$ with corresponding stabilizer
  $P'=\GL(L^0)$.  It is obvious that $P'^1\subset P^1$.  Setting
  $\beta'\in(\bfP')^\vee$ equal to the functional induced by $[\nt]$,
  we see that $(P', 0, \beta')$ is a stratum contained in $(V,\nabla)$
  with $e_{P'}=1$.  It will be convenient to denote $[\nt]$ by $\bn$
  or $\bn'$ depending on whether it is being viewed as a
  representative of $\beta$ or $\beta'$.  With this convention, it
  follows from Lemma~\ref{subquotient} that there is a parabolic
  subalgebra $\fq\subset\gl(\bar{L}^0)$ with Levi subalgebra $\fh$
  such that $\bbn'\in\fq$ and $\bbn$ is the projection of $\bbn'$ onto
  $\fh$.

  Since $(P, 0, \beta)$ is fundamental, there exists $0 \le j < e_P$
  such that $\bbn\in\End(\bL^j)$ is non-nilpotent.  It follows that
  $(P', 0, \beta')$ is fundamental.  The splitting $V = V_1 \oplus
  V_2$ for $\beta$ does not necessarily split $\beta'$ at level $0$:
  using the notation of Section~\ref{section:splittings}, it is
  possible that $\epsilon_1 \beta'_\nu \epsilon_2 \in \mathfrak{P}^1$
  has non-trivial image in $\bfP'$.  Identify $\bM^j$ with $\epsilon_1
  \mathfrak{P}^j \epsilon_2 / \epsilon_1 \mathfrak{P}^{j+1}
  \epsilon_2$ and $\bN^j$ with $\epsilon_2 \mathfrak{P}^j \epsilon_1 /
  \epsilon_2 \mathfrak{P}^{j+1} \epsilon_1$.  By
  Definition~\ref{splitstrata} and Lemma~\ref{splitstrata},
  $\bar{\partial}_{\beta'_\nu}$ (resp. $\bar{\partial}'_{\beta'_\nu}$)
  is an automorphism of each $\bM^j$ (resp. $\bM^j$).   In particular, there
  exists $X_1\in\epsilon_1 \mathfrak{P}^1 \epsilon_2 $ and $Y_1 \in
  \epsilon_2 \mathfrak{P}^1 \epsilon_1$ such that $\ad (X_1)
  (\beta'_\nu) \in - \epsilon_1 \beta'_\nu \epsilon_2 +
  \mathfrak{P}^2$ and $\ad (Y_1) (\beta'_\nu) \in - \epsilon_2
  \beta'_\nu \epsilon_1 + \mathfrak{P}^2$.  It follows that
  $\epsilon_i \Ad (1+Y_1)\Ad (1+X_1) (\beta'_\nu) \epsilon_j \in
  \mathfrak{P}^2$ for $i\ne j$.  Continuing this process, we construct
  $p =(1+Y_{e_P-1}) (1+X_{e_P-1} ) \ldots (1+Y_1) (1+X_1) \in P^1$
  such that $\epsilon_i \Ad (p) (\beta'_\nu) \epsilon_j\in
  \mathfrak{P}^{e_P}$ for $i\ne j$.  Since $\mathfrak{P}^{e_P} \subset
  \mathfrak{P}'$, this implies that $(V_1, V_2)$ splits $\Ad (p)
  (\beta'_\nu)$ at level $0$, so $(p^{-1} V_1, p^{-1} V_2)$ splits
  $\beta'_\nu$ at level $0$.  Clearly, $(p^{-1} V_2, p^{-1} V_2)$
  still splits $(P, 0, \beta_\nu)$; $q=p^{-1}$ will be the desired
  element of $P^1$.

  Without loss of generality, we may assume that $V_1$ and $V_2$ split
  $\mathscr{L}'$ and $\beta'_\nu$ at level $0$.  Since
  $\bar{\beta}'_\nu$ projects to $\bar{\beta}_\nu \in \mathfrak{h}$,
  it is clear that these matrices have the same eigenvalues (as do
  $\bar{\beta}_{j\nu}$ and $\bar{\beta}'_{j\nu}$ for $j = 1, 2$), so
  $(P',0, \beta')$ satisfies conditions (1) and (3) of
  Definition~\ref{splitstrata}.  By assumption $(P_1, 0, \beta_1)$ and
  $(P_{1 2}, 0, \partial_{\beta})$ are strongly uniform.  Since $(V_i
  \cap \mathscr{L}')$ is a sub-lattice chain of $\mathscr{L}_i$, it is
  clear that $(P_1',0, \beta'_1)$ is strongly uniform.  It remains to
  show that the induced map $\bar{\partial}_{\beta'_\nu}$ is an
  automorphism of $\End(\bM')$.

  Let $F^i$ be the image of $M^i \cap (M')^0$ in $\bM'$.  It is easy
  to see that $(M')^0\subset M^{1-e_P}$ and $M^{e_P}\subset (M')^1$,
  implying that $F^{e_P} = \{0\}$ and $\bM^0 = F^{-e_P+1}$.  Moreover,
  $\partial_{\beta'_\nu}$ preserves the flag $\{F^i\}$, so
  $F^i/F^{i+1}$ is a $\bar{\partial}_{\beta'_\nu}$-invariant subspace
  of $\bM^i$.  Since $(P_{1 2}, 0, \partial_{\beta})$ is strongly
  uniform, $\bar{\partial}_{\beta'_\nu}\in\Aut(\bM^i)$ for all $i$,
  hence the restrictions to $F^i/F^{i+1}$ are also automorphisms.  It
  follows that $\bar{\partial}_{\beta'_\nu}$ gives an automorphism of $\bM'$.

\end{proof}

\begin{theorem}\label{thm2}
  Suppose that $(V,\nabla)$ contains a fundamental stratum $(P, r, \beta)$.
  Let $(V_1, V_2)$ split $(P, r, \beta)$.  Then, there exists $p\in
  P^1$ such that $(pV_1,pV_2)$ splits both  $(P, r, \beta)$ and $\nabla$.
\end{theorem}
The case when $e_P = 1$ is well known (see \cite[Lemma 2]{Le}). 
\begin{proof}

We first recall from Remark~\ref{splitp1} that  $(pV_1,pV_2)$ splits
$(P, r, \beta)$ for any $p\in P^1$.  

Let $V'=F^n$, and let $V'=V'_1\oplus V'_2$ be the standard splitting
of $F^n$ into subspaces with $\dim V'_j=\dim V_j$.  Let
$\ep_j'\in\gl_n(F)$ be the corresponding idempotents.  By
\eqref{gauge}, $(W_1,W_2)$ splits $\nabla$ with $\dim W_j=\dim V_j$ if
and only if there exists a trivialization $\psi:V\to V'$ such that
\begin{equation}\label{projsplitting} \epsilon'_1 [\nt]_\psi \epsilon'_2 =
  \epsilon'_2 [\nt]_\psi \epsilon'_1 = 0,
\end{equation}
and $W_j=\psi^{-1}(V'_j)$.

Since $\nabla$ contains $(P,r,\b)$, there exist $j$ such that
$\gr^j(\nt)=\bbn$; reindexing the lattice chain if necessary, we can
assume that $j=0$.  Fix a trivialization $\phi$ such that
$\phi(L^0)=\fo^n$ and $\phi(V_j)=V'_j$.  We set $\fP'^m=
{}^\phi\fP^m\subset\gl_n(F)$ and similarly for $P'^m$.  Setting
$\bn'={}^\phi \bn$, we have $[\nt]_\phi\equiv \bn'\pmod{\fP'^{-r+1}}$.
By \eqref{gaugechange}, it suffices to find $h \in P'^1$ such that $h
\cdot [\nt]_\phi$ satisfies \eqref{projsplitting}; then $(pV_1,pV_2)$
splits $\nabla$, where $p=\phi^{-1}h^{-1}\phi\in P^1$.

Inductively, we construct $h_m\in P'^1$ such that $h_m\equiv
h_{m-1}\pmod{\fP'^{m-1}}$ and $\ep'_i(h_m\cdot[\nt]_\phi)\ep'_j\equiv
0\pmod{\fP'^{-r+m+1}}$ for $i\ne j$.  The limit $h=\lim h_m\in P^1$
will then satisfy \eqref{projsplitting}.  For $m=0$, we can take
$h_0=I$, since $(V_1,V_2)$ splits $(P,r,\b)$.  Now, suppose $m\ge 1$
and we have
already constructed $h_{m-1}$.  Let $Q
= h_{m-1}\cdot [\nt]_\phi$ and $Q_{ij} = \ep'_i Q \ep'_j$, so that
$Q_{12},Q_{21}\in P'^{-r+m}$.  We
will find $g=I+\ep'_1 X \ep'_2 + \ep'_2 Y \ep'_1\in P'^{m}$ with $X \in V_{1 2}$
and $Y \in V'_{2 1}$ satisfying 
\begin{equation}\label{gcong}
(g \cdot Q)_{1 2} \equiv (g \cdot Q)_{2 1} \equiv 0 \pmod{\fP'^{-r+m+1}}.
\end{equation}
The element $h_m=g h_{m-1}\in P^1$ will then have the desired properties.

Given $g$ as above, the gauge change formula $g \cdot Q = g Q g^{-1} -
\tau (g) g^{-1}$ immediately leads to the equation
\begin{small}
\begin{equation*}
\begin{pmatrix}
\id & X \\ Y & \id
\end{pmatrix}
\begin{pmatrix}
Q_{11} & Q_{1 2} \\ Q_{2 1} & Q_{22}
\end{pmatrix}
-
\tau\begin{pmatrix}
0 & X \\ Y & 0
\end{pmatrix}
=
\begin{pmatrix}
(g \cdot Q)_{11} & (g \cdot Q)_{12} \\
(g \cdot Q)_{21} & (g \cdot Q)_{22} 
\end{pmatrix}
\begin{pmatrix}
\id & X \\ Y & \id
\end{pmatrix}.
\end{equation*}
\end{small}

Since $XQ_{21}$ and $YQ_{12}$ lie in
$\fP'^{-r+2m}\subset\fP'^{-r+m+1}$, the congruences \eqref{gcong} are
equivalent to the system of congruences
\begin{equation*}\begin{aligned}
Q_{11} - (g \cdot Q)_{11} & \equiv 0 \\
Q_{22} - (g \cdot Q)_{22} & \equiv 0 \\
-\tau X + X Q_{2 2}  -  (g \cdot Q)_{11} X + Q_{12} 
& \equiv 0 \\
-\tau Y + Y Q_{11}  -  (g \cdot Q)_{22}  Y + Q_{21} 
& \equiv 0 ,
\end{aligned}\pmod{\fP'^{-r+m+1}}
\end{equation*}
where the first two automatically hold for any $g$ of the given form.
Suppose that $r\ge 1$.  In this case, 
$\tau X$ and $\tau Y$ are in $\fP'^{m}\subset\fP'{-r+m+1}$, so these
terms drop out of the congruences.   Substituting using the first two
congruences, the problem is reduced to finding $X$ and $Y$ such that
\begin{equation*}\begin{aligned}
Q_{11} X - X Q_{22} & \equiv  Q_{12} \pmod{\fP'^{-r+m+1}} \\
Q_{22} Y - Y Q_{11} & \equiv Q_{21} \pmod{\fP'^{-r+m+1}}.
\end{aligned}
\end{equation*}
However, since $Q\equiv \bn'\pmod{\fP'^{-r+1}}$, the first equation is
  equivalent to $\partial_{\beta'} (X) \equiv Q_{1 2}
  \pmod{\fP'^{-r+m+1}}$, and a solution $X$ exists since
  $(P_{12},r,\partial_\b)$ is strongly uniform.  Similarly,
  Lemma~\ref{spstratalemma} guarantees the existence of a solution $Y$
  to the second equation.

When $r = 0$, Lemma~\ref{case0} shows that there
exists $q\in P^1$ such that $(qV_1,qV_2)$ splits a fundamental stratum
$(\hat{P},0,\hat\b)$ with $\hat{P}^1\subset P^1$ and $e_{\hat{P}}=1$.
We are thus in the classical case of lattice chains with period $1$,
and there exists $q'\in\hat{P}^1$ such that $(q'qV_1,q'qV_2)$ splits
$\nabla$ by \cite[Lemma 2]{Le}.  The desired element of $P^1$ is thus
given by $p=q'q$.
\end{proof}

\subsection{Formal Types}\label{formaltypes}

Suppose that $(V, \nabla)$ is a formal connection which contains
a regular stratum.  We fix a trivialization $\phi : V \to F^n$.
In this section, we will show that the matrix of $(V, \nabla)$ in this
trivialization can be diagonalized by a gauge transformation into a 
uniform torus $\ft \subset \gl_n(F)$. The diagonalization of $[\nt]$
determines a functional $A \in (\ft^0)^\vee$ called a \emph{formal type},
and any two connections on $V$ with the same formal type are
isomorphic.

In the following, let $(P, r, \beta)$ be a regular stratum in
$\GL_n(F)$ with $P \subset \GL_n(\fo)$, and let $T \subset \GL_n(F)$
centralize $(P, r, \beta)$.  We denote $\phi^{-1} P \phi \subset
\GL(V)$ by $P^\phi$, and write $\beta^\phi$ for the pullback of
$\beta$ to $(\bfP^\phi)^{-r}$.  Suppose that $(V, \nabla)$ contains
$(P^\phi, r, \beta^\phi)$, and that $\beta^\phi$ is determined by
$\gr^0 (\nt)$.  (By Lemma~\ref{grindep}, the second condition is
superfluous when $r >0$.)  The goal of this section is the following
theorem:
\begin{theorem}\label{thm4}   Fix $\nu$ of order $-1$.  There
  exists $p \in P^1$ and a regular element $A_\nu \in \ft^{-r}$ such
  that $p \cdot [\nt]_\phi = A_\nu$ and $A_\nu$ is a representative
  for $\beta$.  Furthermore, the orbit of $A_\nu $ under $P^1$-gauge
  transformations contains $A_\nu +\fP^1$, and $A_\nu$ is unique
  modulo $\ft^1$.
\end{theorem}
The obvious analogue of this theorem holds for an arbitrary $\nu$.
\begin{rmk}
  The above theorem implies that after passing to a ramified cover
  (specifically, the splitting field for $T$), any connection
  containing a regular stratum is formally gauge equivalent to a
  direct sum of line bundles of slope less than or equal to $r$ (with
  equality in all but at most one factor, with inequality only
  possible when $e=1$).  Moreover, the associated rank one connections
  have pairwise distinct leading terms.  These properties could be
  used as an ad hoc way of defining the class of connections which are
  the primary topic of this paper.  However, the perspective gained
  from our intrinsic approach via regular strata will prove essential
  below.  The Lie-theoretic nature of this approach also suggests that
  it can be adapted to study flat $G$-bundles for $G$ a reductive
  group.

\end{rmk}

We also remark that $A_\nu$ satisfies a stronger condition than
regular semisimplicity.  Suppose that $A_\nu = (a_1, \ldots, a_{n/d})
\in E^{n/d}$, where $d=e_P$.  Then, $a_j = a_{j,-r} \varpi_E^{-r} +
a_{j,-r +1} \varpi_E^{-r+1} + \dots$, with $a_{j} \ne 0$ except
possibly for a single $j$ when $d=1$.  The leading term $A_\nu'=
(a_1^{-r} \varpi_E^{-r+1}, \ldots a_{n/d}^{-r} \varpi_E^{-r+1})$ is a
representative for $\bn$, since the higher order terms lie in
$\fP^{-r+1}$.  By Proposition~\ref{lemrss}, $A'_\nu$ is regular
semisimple, and we see that $\An$ has regular leading term.

In the following definition, let $T\subset \GL_n(F)$ be a uniform
maximal torus such that $T(\fo) \subset\GL_n(\fo)$.  We set
$P=P_{T,\fo^n}$ as defined before Proposition~\ref{uniquedet}.  We also allow
$\nu$ to have arbitrary order. 
\begin{definition}\label{def:formaltype}  A functional
  $A\in(\ft^0)^\vee$ is called a \emph{$T$-formal type of depth $r$}
  if
\begin{enumerate}
\item $\ft^{r+1}$ is the
  smallest congruence ideal contained in $A^\perp$; and 
 \item the stratum $(P, r, \beta)$ is regular and centralized by  $T$,
   where $\beta\in(\bfP^r)^\vee$ is the functional induced by
   $\pi_\ft^*(A)\in \fP^\vee$.
 \end{enumerate}
We denote the  space of $T$-formal types of depth $r$ by $\A(T,
r)\subset (\ft^0 / \ft^{r+1})^\vee$.  A $T$-formal type is any element
of $\A(T)=\cup_{r\ge 0}\A(T,r)$.
\end{definition}
We will always use the notation $A_\nu$ for a representative of $A$ in ${\ft}^{-r-(1+\ord(\nu))e_P}$.

\begin{rmk}\label{blockdiagonal} There is an embedding of $\A(T, r)$
  into ${\ft}^{-r-(1+\ord(\nu))e_P}/{\ft}^{1-(1+\ord(\nu))e_P}$
  determined by the pairing $\left< , \right>_\nu$.  For simplicity,
  we only describe it when $\nu=\frac{dt}{t}$.
  First, recall that $\ft$ has a natural grading so that
  $\ft^{-r}/\ft^1\cong\oplus_{i=-r}^0 \bft^i$.  If $r>0$, then
  $\A(T,r)$ is isomorphic to the open subspace of $\ft^{-r}/\ft^1$
  with degree $-r$ term regular.  If $r=0$, then $\ft$ is split and
  $\ft^{0}/\ft^1\cong\tfl\cong k^n$.  In this case, $\A(T,r)$
  corresponds to $\sum a_i\chi_i\in\tfl$ with the $a_i$'s distinct
  modulo $\Z$.  This is not a Zariski-open subset of $\tfl$.  However,
  if $k=\C$, it is open in the complex topology.

  To be even more explicit, assume that $\ft$ is the block-diagonal
  Cartan subalgebra of $\gl_n(F)$ as in Remark~\ref{standardtorus}.
  If $\ft$ is split (and $r>0$), there is a bijection between formal
  types $A$ and representatives of the form $A_\nu=\sum_{i=0}^r t^{-i}
  D_i$ with $D_i\in\gl_n(k)$ diagonal and $D_r$ regular.  In the pure
  case, there is a similar bijection between formal types and
  representatives $A_\nu=q(\varpi_I^{-1})$ where $q\in k[X]$ has
  degree $r$.  Throughout Section~\ref{modspace}, we will assume that
  $\ft$ has such a block diagonal embedding into $\gl_n(F)$.
\end{rmk}

\begin{rmk} An element of $\A(T,r)$ may also be viewed as a functional
  on $\fP/\fP^{r+1}$ (resp. $\gl_n(\fo)^\vee$) for which the corresponding stratum
  $(P,r,\beta)$ is regular and all of whose representatives lie in $\ft^{-r-(1+\ord(\nu))e_P} + \fP^{1-(1+\ord(\nu))e_P}$.
\end{rmk}

The notion of a $T$-formal type actually depends only on the conjugacy
class of $T$.  Indeed, set $L = \fo^n \subset F^n$.  If $T$ and $S$
are conjugate tori with $T(\fo),S(\fo)\subset\GL_n(\fo)$, then
Lemma~\ref{derp} below states that there exists $h \in \GL(L)$ such
that ${}^h T = S$.  It is evident that ${}^h \fP^j_{T, L} = \fP^j_{S,
  L}$ and ${}^h \ft^j = \fs^j$ for all $j$.  Applying
Remark~\ref{regularaction}, we conclude that $\Ad(h) (A_\nu) \in
\fs^{-r}$ determines a regular stratum $(P_{S, L}, r, \b')$
centralized by $S$.

We say that a lattice $L \subset V$ is \emph{compatible}
with $T$ if $T(\fo) \subset \GL (L)$.
\begin{lemma}\label{derp}
Suppose that $T$ is a uniform maximal torus.  
\begin{enumerate}
\item The set of lattices $L$ that are compatible with $T$ is a single $N(T)$-orbit.  
\item If $S$ is conjugate to $T$ in $\GL(V)$, and $L$ is compatible with
both $S$ and $T$, then $S$ is conjugate to $T$ in $\GL(L)$. 
\end{enumerate}
\end{lemma}
\begin{proof}
Suppose that $L$ and $L'$ are compatible with $T$, and let
$g \in \GL(V)$ satisfy $g L = L'$.  In particular, this implies
that $S =  g^{-1}Tg$ is compatible with $L$.
Choose 
$x \in \ft^{-r}$ with regular leading term.
By Proposition~\ref{uniquedet},
there exist parahoric subgroups $P_{T,L},P_{S,L} \subset \GL(L)$
such that $x$ and $\Ad(g^{-1})(x)$ 
determine regular strata $(P_T, r, \b)$ and $(P_S, r, \b')$.   
By Theorem~\ref{thm1}, $e_{P_T} = e_{P_S}$.  
  
The same theorem
states that  $(P_T, r, \b)$ and $(P_S, r, \b')$ induce splittings of $L$, and it is easily checked
that there exists an element of $h \in \GL(L)$ taking the components of the $T$-splitting
to the $S$-splitting.  Replacing $g $ with $gh$, we may assume that the
splittings induced by $S$ and $T$ are the same.  Thus, we may reduce
to the pure case.

Suppose that $(P_T, r, \b)$ and $(P_S, r, \b')$ are pure.  We may choose
$h' \in \GL(L)$ such that $h' P_T (h')^{-1} = P_S$, so by a similar
argument, we may assume $P_T = P_S = P$.
By \eqref{Adcalc},
there exists $p \in P$ such that $\Ad(p) (\Ad(g^{-1})(x)) \in \ft + \fP^{-r+1}$.  Finally,
Lemma~\ref{A1} implies that there exists $p' \in P^1$ such that 
$\Ad(p'p g^{-1}) (x)) \in \ft$.  It follows that $p' p g^{-1}  = n^{-1} \in N(T)$.  
It is now clear that $n L = L'$, since $p'$ and $p$ are in $\GL(L)$.

We now prove the second statement.
Suppose that $S = g T g^{-1}$ for $g \in \GL_n(F)$.  
Then, $L$ and $g L$ are compatible with $S$.  By
the first part, $g L = n L$ for some $n \in N(S)$.
It follows that there exists $h = n^{-1} g \in \GL(L)$
such that $S = h T h^{-1}$.
\end{proof}

We continue to fix $T\subset \GL_n(F)$ as in Definition~\ref{def:formaltype}.

\begin{definition} The set $\A_T^{(V,\n)}\subset\A(T)$ of $T$-formal types
  associated to $(V, \nabla)$ consists of those $A$ for which there is
  a trivialization $\phi : V \to F^n$ such that $(V, \nabla)$ contains
  the stratum $(P^\phi, r, \beta^\phi)$ and the matrix
  $[\nabla_\tau]_\phi$ is formally gauge equivalent to an element of
  $A_\nu + \fP^{1-(1+\ord(\nu))e_P}$ by an element of $P^1$.
\end{definition}
By Theorem~\ref{thm4}, the last statement is equivalent to the 
condition $[\nabla_\tau]_\phi$ is formally gauge equivalent to
$A_\nu$.

\begin{proposition}\label{pywit} Let $\nabla$ be a connection containing a 
  $T$-formal type $A_T$.  If $S$ is a maximal torus with
  $S(\fo)\subset\GL_n(\fo)$, then $\nabla$ has an $S$-formal type if
  and only if $S$ is $\GL_n(F)$-conjugate to $T$.  Moreover, if
  $h\in\GL_n(\fo)$ conjugates $T$ to $S$, then $\Ad^*(h^{-1})$ gives a
  bijection from $\A_T^{(V,\n)}$ to $\A_S^{(V,\n)}$.
\end{proposition}
\begin{proof}
Set $L = \fo^n$.  By Lemma~\ref{derp}, there exists
$h \in \GL_n(\fo)$ such that $S = h T h^{-1}$.
We may choose
a trivialization $\phi : V \to F^n$ such that $[\n_\tau]_\phi = A_\nu \in \ft$
by Theorem~\ref{thm4}.  
We now observe that, by Lemma~\ref{taup1},
\begin{equation*}
h \cdot [\n_\tau]_{\phi} \in \Ad(h) (A_\nu) + t \gl_n(\fo) \subset \fs + t \gl_n(\fo).
\end{equation*}
After dualizing, $A_\nu$ and $\Ad(h) (A_\nu)$ determine functionals
$A_T$ and $A_S$ (respectively) in $\gl_n(\fo)^\vee$ and
$\Ad^*(h^{-1}) (A_T) = A_S$.  Moreover,  $A_S$ is an $S$ formal type corresponding
to $(V, \n)$:  it is clear that any representative for $A_S$ lies in
$S = \Ad(h)(T) + \fP_S^1$, since $t \gl_n(\fo) \subset \fP_S^1$,
and  $(V, \n)$ contains the regular stratum
$(P_S^{h \phi}, r, \b')$.  Here, $\b'$ is the functional on $\fP_S^{r}/ \fP_S^{r+1}$
determined by $A_S$.
 \end{proof}

Before proving Theorem~\ref{thm4}, we give two corollaries.

\begin{corollary}\label{newlemma}
  Suppose that $A_\nu \in \ft$
is a representative of a $T$-formal type.  If $g \in \GL_n(F)$ satisfies
$g \cdot A_\nu = A_\nu$, then $g \in \Tfl$.
\end{corollary}
\begin{proof} We assume without loss of generality that $\nu = \frac{dt}{t}$, so
 $\tau = t\frac{d}{dt}.$
Since $g$ is invertible, it will suffice to show that $g\in\tfl$.
  Choose a regular stratum $(P, r, \b)$ corresponding to $A$, and
  consider the exhaustive filtration $\tfl+\fP^i$ of $\gl_n(F)$.
  Suppose that
  $g\notin \tfl$, and let $\ell$ be the largest integer such that $g\in\tfl+\fP^\ell$. By assumption, $[g, A_\nu] = \tau (g) \in \fP^\ell$.

  First, assume that $r>0$.  Note that $\ell\ne 0$, since $g\in\fP$
  implies that $g \in \tfl + \fP^1$ by Proposition~\ref{liecen}.
  Suppose $\ell<0$.  Corollary~\ref{liecencor} gives $g=s+h$ with
  $s\in\ft^{\ell}$ and $h\in \fP^{\ell+r}$.  Since $\pi_\ft ([g,
  A_\nu]) \in \ft^{\ell+1}$ by Proposition~\ref{cores3}, we also get
  $\pi_\ft(\tau(s))\in\ft^{\ell+1}$.  Lemma~\ref{Hlemma} now gives $s
  \in \fP^{\ell+1}$, a contradiction.  Hence, $\ell\ge 1$, and we have
  $g=s_0+x$ with $s_0\in\tfl$ and $x\in\fP^\ell$.  Since $[x, A_\nu] =
  \tau (x) \in \fP^\ell$, we get the contradiction $x\in\fP^{\ell+1}$
  by the same argument as in the $\ell<0$ case.

  When $r = 0$, we may assume $A_\nu$ is a regular diagonal matrix in
  $\gl_n(k)$ satisfying the last condition of
  Definition~\ref{regstratum}.  In other words, $-\ad(A_\nu)$ has no
  non-zero integer eigenvectors in $\gl_n(k)$.  Write $g = t^\ell
  g_\ell + t^{\ell+1} g_{\ell+1}+ \ldots$ with $g_j \in \gl_n(k)$.
  Then, $[g, A_\nu] = \tau(g)$ implies that $-\ad(A_\nu) (g_j) = j
  g_j$.  We deduce that $g_j = 0$ except when $j = 0$.  Moreover,
  since $A_\nu$ is regular, $[g_0, A_\nu] = 0$ implies that $g_0 \in
  \tfl$.  Thus, $g\in\tfl$.

\end{proof}

\begin{corollary}\label{ftgauge}
  Let $A$ be a formal type.  Any two connections with formal type $A$
  are formally isomorphic.  Furthermore, the set of  formal types associated to a
  connection is independent of choice of $\nu \in \Omega^\times$.
\end{corollary}
\begin{proof}

  Independence of $\nu$ follows by the argument given in
  Proposition~\ref{stratumindep} and the remark above.  Thus, fix
  $\nu$ with order $-1$.  Suppose $(V, \nabla)$ and $(V', \nabla')$
  have formal type $A$, and let $\phi$ (resp. $\phi'$) be the given
  trivialization for $V$ (resp. $V'$).  By Theorem \ref{thm4}, there
  exists $A_\nu \in \ft^{-r}$ and $p, p' \in P^1$ such that $p \cdot
  [\nt]_\phi = A_\nu = p' \cdot [\nt']_{\phi'}$.  It is easily checked
  that the composition $ (\phi')^{-1} \circ (p')^{-1} \circ p \circ
  \phi : V \to V'$ takes $\nabla$ to $\nabla'$.
  
\end{proof}
We begin the proof of Theorem~\ref{thm4}.  Throughout, we will
suppress the fixed trivialization $\phi$ from the notation.  We may
assume that $\nu = \frac{dt}{t}$, so $\tau=t\frac{d}{d t}$.  First, we
show that if the result holds for the trivialization $\phi$ and the
regular stratum $(P,r,\b)$ centralized by $T$, then, for any
$g\in\GL_n(\fo)$, it holds for the trivialization $g\phi$, the regular
stratum $({}^g P,r,{}^g \b)$, and its centralizing torus ${}^g T$.

Suppose that $g \in \GL_n(\fo)$.  By Lemma~\ref{taup1}, $\tau(g)
g^{-1} \in t \gl_n(\fo)$.  In particular, $\tau (g) g^{-1} \in \fP^1$.
Therefore, if the theorem holds for $(P,r,\b)$, then there exists $p
\in P^1$ such that $p \cdot [\nt] = A_\nu + \Ad(g^{-1}) (\tau(g)
g^{-1})$.  It follows that ${}^g p \cdot (g \cdot [\nt]) = g \cdot (p
\cdot [\nt]) = \Ad(g) (A_\nu)$.  Thus, the first part of the theorem
still holds after changing the trivialization by $g$. The second and third parts
follows from a similar argument.

Without loss of generality, we henceforth assume that $\ft$ embeds
into the $d \times d$ diagonal blocks of $\gl_n (F)$ and in each
diagonal block the matrix $\varpi_I$ from \eqref{varpip} is a
uniformizer for the corresponding copy of $E$.

By Theorem~\ref{thm1}, this splitting of $V$ splits $(P, r, \beta)$
into pure strata, plus at most one non-fundamental stratum in the case
$e_P = 1$.  Therefore, Theorem~\ref{thm2} shows that $\nabla$ splits
into a direct sum of connections containing a pure stratum when $e_P
>1$ and into a direct sum of connections in dimension $1$ when $e_P =
1$.  Moreover, the splitting for $\nabla$ maps to the splitting
determined by $T$ by an automorphism $p \in P^1$.  In other words, $p \cdot
[\nabla_\tau]$ lies in $\bigoplus_{j = 1}^{n/d} \gl(V_j)$.

First, we consider the case $e_P = 1$ (which includes the case $r =
0$).  By the above discussion, we may reduce to the case where $\dim
V=1$.  In this case, $[\nt] \in F$ and $g \cdot [\nt] \in [\nt] +
\fp^1$ for all $g \in 1+\fp^1$.  This proves the first statement and
the statement about uniqueness.  It suffices to show that the orbit of
$[\nt]$ under gauge transformations contains $[\nt] + \fp^1$.  Suppose
$X \in \fp^1$.  Since $\tau : \fp^1 \to \fp^1$ and $\log : (1+\fp^1)
\to \fp^1$ are surjective, there exists $g \in 1 + \fp^1$ such that
$\tau (\log (g)) = X$.  Therefore, $g \cdot ([\nt]+X) = [\nt]$, and
the assertion follows.

When $e_P>1$, it suffices to prove the theorem in the case when $(P,
r, \beta)$ is pure.  In particular, $P= I$ is an Iwahori subgroup and
$T \cong E^\times$.  Take $\bn = [\nt]$.  By Remark~\ref{repint}, $\bn
\in \ft \cap \fI^{-r} + \fI^{1-r}$.

The following two lemmas prove Theorem~\ref{thm4} in the pure case with
$e_P >1$ and thus complete the proof of the theorem.

\begin{lemma}\label{dlog}
Let $\psi_\ell$ be defined as in Section~\ref{sec:cores}.
When $\ell \ge 1$,
$\tau (\fI^\ell) \subset
\fI^{\ell}$ and $\psi_{\ell} (\tau (\varpi^\ell_{I})) \ne 0$.
Furthermore,
\begin{equation*}
\left[\tau (1 + \alpha \varpi_I^\ell)\right]  (1 + \alpha \varpi_{I}^\ell)^{-1} 
\equiv  \alpha \tau (\varpi_{I}^\ell) \pmod{ \mathfrak{I}^{\ell + 1}}
\end{equation*}
for any $\alpha \in k$.
\end{lemma}
\begin{proof}
  Suppose that $\ell = q n + z,$ for $0 \le z <n$.  The matrix
  coefficients of $\varpi^\ell_{I}$ are
\begin{equation*}
(\varpi^\ell_{I})_{ij} = 
\begin{cases}
t^{q+1} & \text{ if $j = i +z-n$;}\\
t^{q} &  \text{ if $j  = i+z$;} \\
0 & \text{ if $j \not\equiv i + z \pmod{n}$.}

\end{cases}
\end{equation*}
Let $x$ be the diagonal matrix with $x_{jj} = q$ when $j \le n-z$ and
$q+1$ otherwise.  Then, $\tau (\varpi_{I}^\ell) = x
\varpi_{I}^{\ell}$.  Moreover, $\tau(\fI)\subset\tau(\gl_n(\fo))\subset
t\gl_n(\fo)\subset \fI$.
The Leibniz rule and the fact that $\fI^\ell = \varpi_I^\ell \fI$  now
imply that $\tau (\fI^\ell) \subset \fI^{\ell }$ for all $\ell\ge 1$.
The first assertion of the lemma follows, since
$\psi_{\ell} (\tau (\varpi_{I}^\ell))$ is the trace of $x$, which is non-zero for $\ell \ne 0$.

To see the second statement, observe that $(1+ \alpha
\varpi_{I}^\ell)^{-1} = 1 -\alpha \varpi_{I}^\ell + y$, with $y \in
\mathfrak{I}^{\ell + 1}$.  Therefore,
\begin{equation*}\begin{aligned}
\left[\tau (1 + \alpha \varpi_I^\ell)\right]  (1 + \alpha \varpi_{I}^\ell)^{-1}   &=
\left[\tau (1 + \alpha \varpi_I^\ell)\right]  (1 - \alpha \varpi_{I}^\ell + y)  \\
& = \alpha \tau (\varpi_{I}^\ell) (1 -\alpha  \varpi^\ell_I+ y) \\
& \equiv \alpha \tau (\varpi^\ell_{I}) \pmod{\mathfrak{I}^{\ell + 1}}.
\end{aligned}
\end{equation*}
\end{proof}
\begin{lemma}\label{formaltype}
  Suppose that $(V,\nabla)$ contains the pure stratum $(I, r, \beta)$
  with $n\ge 2$ (so $r\ge 1$).  Then, there is a unique
  $q(x) \in k[x]$ such that $[\nt]$ is formally gauge equivalent to
  $q(\varpi_{I}^{-1})$ by an element of $I^1$. 
   If  $B_\nu \in  q(\varpi_{I}^{-1})+\fI^1$, then
$B_\nu$ is formally gauge equivalent to  $[\nt]$ by an element of $I^1$.
  \end{lemma}
\begin{proof}
By the remarks made before Lemma~\ref{dlog},
$[\nt] = \bn = q_r \varpi_{I}^{-r} + y$ with $y \in \fI^{-r+1}$.
Moreover, since $\bn \notin \fI^{-r+1}$, $q_r\ne 0$.
We need to find 
$p \in I^1$ with the property 
\begin{equation}\label{gq}
p \cdot \beta_\nu=  q(\varpi_{I}^{-1}),
\end{equation}
for $q\in k[x]$ as in the statement of the lemma.

Inductively, we construct $g_\ell\in I^1$ and $q^\ell\in k[x]$ of
degree $r$ such that $g_\ell\equiv g_{\ell-1}\pmod{\fI^{\ell-1}}$,
$\deg(q^\ell-q^{\ell-1})\le r-\ell+1$, and $ g_\ell \cdot \beta_\nu
\in q^\ell (\varpi_{I}^{-1}) + \fI^{\ell-r}$.  Moreover, we will show
that $q^\ell(\varpi_I^{-1})$ is unique modulo $\ft^{\ell - r}$.  Note
that $q^\ell$ is independent of $\ell$ for $\ell>r+1$.  If we set
$p=\lim g_\ell$ and $q=q^\ell$ for large $\ell$, \eqref{gq} is
satisfied.

We start by taking $g_1=1$ and $q_1=q_r x^r$.  Suppose that we have
constructed $g_{\ell}$ and $q^\ell$; note that $q^\ell_r=q_r$.  We
will find $g=1+X\in I^{\ell}$ such that $g_{\ell+1}=g g_\ell\in I^1$
has the required properties.  Obviously, $g_{\ell+1}\equiv
g_{\ell}\pmod{\fI^{\ell}}$.

To construct $g$, first, note that $\tau (g) g^{-1} = \tau (X)g^{-1} \in \fI^\ell$ by
Lemma~\ref{dlog}.  Moreover, $g^{-1} \equiv 1-X \pmod{\fI^{\ell+1}}$.
If $\ell -r \le 0$, it suffices to find $g \in I^\ell$ such that
\begin{equation*}
\Ad(g) (g_\ell \cdot \beta_\nu) \in \mathfrak{t}  + \fI^{-r+\ell+1}.
\end{equation*}
We see that
\begin{equation*}\begin{aligned}
\Ad(g) (g_\ell \cdot \beta_\nu) &\equiv  (1+ X) (g_\ell \cdot  \beta_\nu) (1-X) 
\pmod{\fI^{\ell-r+1}} \\
& \equiv g_\ell \cdot  \beta_\nu  + 
{q}_r \delta_{X} (\varpi_I^{-r}) \pmod{\fI^{\ell-r+1}}.
\end{aligned}
\end{equation*} 
Thus, we need to solve the equation
\begin{equation*} g_\ell \cdot  \beta_\nu  + q_r \delta_{X} (\varpi_I^{-r}) 
\equiv Y \pmod{ \mathfrak{I}^{\ell-r+1}}.
\end{equation*}
for $Y \in \ft$.
Since ${q}_r \ne 0$, 
Proposition~\ref{cores3} implies that a solution for $X$
exists if and only if $Y \in \pi_{\ft} (g_\ell \cdot \beta_\nu)  + \fI^{\ell-r+1}$.
Letting ${q}^{\ell+1}(\varpi_{I}^{-1})$ denote the terms of
nonpositive degree in $\pi_{\mathfrak{t}} (g_\ell \cdot \beta_\nu)$
(where $q^{\ell+1}\in k[x]$), we see that $\deg(q^{\ell+1}-q^\ell)\le
r-l$.  Moreover, $q^{\ell+1}$ is uniquely determined.  


Now, suppose $\ell-r > 0$.  The first part of Lemma~\ref{dlog} implies that
$\pi_{\mathfrak{t}} (\tau(\varpi_{I}^{\ell-r})) \notin \fI^{\ell-r+1}.$
The argument above implies that we may choose $s \in I^\ell$ with the property
\begin{equation*}
\Ad(s) (g_\ell \cdot \beta_\nu)  \equiv {q}^\ell(\varpi^{-1}_{I}) + \alpha \pi_{\mathfrak{t}} \tau (\varpi_{I}^{\ell-r})
\pmod{\fI^{\ell-r+1}}.
\end{equation*}
for some $\alpha \in k$. Again, Proposition~\ref{cores} implies that
there exists $h \in I^\ell$ such that 
\begin{equation*}
\Ad(h) ( {q}^\ell(\varpi^{-1}_{I}) + \alpha \tau (\varpi_{I}^{\ell-r})) \equiv
 {q}^\ell(\varpi^{-1}_{I}) + \alpha \pi_{\mathfrak{t}} \tau (\varpi_{I}^{\ell-r}) 
 \pmod{ \fI^{\ell-r+1}}.
\end{equation*}
Thus, by the second part of Lemma~\ref{dlog},
\begin{equation*}
  \Ad(h^{-1} s)  (g_\ell \cdot \beta_\nu)  \equiv {q}^\ell(\varpi^{-1}_{I}) + \left[\tau (1 + \alpha \varpi_{I}^{\ell-r})\right]  (1 + \alpha \varpi_{I}^{\ell-r})^{-1} 
  \pmod{\fI^{\ell-r+1}}. 
\end{equation*}

Since $1 + \alpha \varpi_{I}^{\ell-r}$ commutes with
${q}^\ell(\varpi_{I}^{-1})$ and $\tau (h^{-1} s) s^{-1} h \in
\fI^\ell\subset\fI^{\ell-r+1}$, it follows that
\begin{equation*}
(1 + \alpha \varpi_{I}^\ell) \cdot \left[ (h^{-1} s) \cdot (g_\ell \cdot
\beta_\nu)  \right]
\equiv {q}^\ell(\varpi^{-1}_{I})
\pmod{\mathfrak{I}^{\ell-r+1}}.
\end{equation*}
Setting $g_{\ell+1} = (1+\alpha \varpi_I^\ell) h^{-1} s$ and
$q^{\ell+1}=q^\ell$ completes the induction.

The same inductive argument (beginning with $\ell = r+1$) 
shows that for any $B_\nu \in q(\varpi_I^{-1})+\fI^1$, there exists 
$h \in I^\ell$ such that $h \cdot B_\nu = q(\varpi_I^{-1})$. 
This completes the proof of the second statement of the lemma.

\end{proof}

\subsection{Formal Types and Formal Isomorphism Classes}\label{formaliso}

In this section, we describe the relationship between formal types and
isomorphism classes of formal connections.  In particular, we show
that formal types are the isomorphism classes in the category of
\emph{framed formal connections}.  This category is the disjoint union
of the categories of \emph{$T$-framed formal connections} as $T$ runs
over conjugacy classes of uniform maximal tori. Moreover, there is an
action of the relative affine Weyl group of $T$ on the set of
$T$-formal types, and the forgetful functor to the category of formal
connections sets up a bijection between orbits of $T$-formal types and
isomorphism classes of formal connections containing a regular stratum
of the form $(P_T,r,\b)$.  We also exhibit an intermediate category
whose isomorphism classes correspond to relative Weyl group
orbits.

Given a conjugacy class of uniform maximal tori and a fixed lattice
$L$, we can choose a representative $T$ such that $T(\fo)\subset
\GL(L)$.  Setting $P=P_{T,L}$, we will further have $T(\fo)\subset
P\subset \GL(L)$ and $T \cong (E^\times)^{n/e_P}$ with $E/F$ a degree
$e_P$ ramified extension.  Upon choosing a basis for $L$, we can
assume without loss of generality that $T(\fo)\subset P\subset
\GL_n(\fo)$ and that $T$ is the standard block diagonal
torus described in Remark~\ref{standardtorus}.  Throughout this
section, we will fix a form $\nu = \frac{dt}{t}$ and the corresponding
derivation $\tau = t \frac{d}{dt}$.

Let $W_T = N(T)/ T$ and $\Waff_T = N(T) / T(\fo)$ be the relative
Weyl group and the relative affine Weyl group associated to $T$.
 Note that $\Waff_T$ is a semi-direct product of
$W_T$ with the free abelian group $T / T(\fo)$, i.e., $\Waff_T \cong
W_T \ltimes T / T(\fo)$.  Furthermore, if we
write $\varSigma_{n/e_P}$ for the group of permutations on the
$E^\times$-factors of $T$ and $C_{e_P}$ for the Galois group of $E/F$,
then $W_T \cong \varSigma_{n/e_P} \ltimes (C_{e_P}^{n/e_P})$.  Here,
$\varSigma_{n/e_P}$ acts on $C_{e_P}^{n/e_P}$ by permuting the
factors.  We note that $N(T) \cap \GL_n(\fo) \subset P_{T, \fo^n}$,
since $C_{e_P}$ and $\varSigma_{n/e_P}$ both preserve the
filtration determined in Proposition~\ref{uniquedet}.

Any element of $W_T$ has a representative in
$\GL_n(k)\subset\GL_n(F)$.  Therefore,
 $W_T \cong (N(T) \cap \GL_n(k)) / \Tfl$.  
 In fact, $W_T$ can be embedded as a subgroup of  $\GL_n(k)$ as follows.
The centralizer of $\Tfl$ in $\GL_n(k)$ is a Levi subgroup isomorphic
to $\prod_{i =1}^{n/e_P} \GL_{e_P} (k)$.

Let $D_i$ (resp. $\fd_i$) denote the diagonal subgroup (resp.
subalgebra) in each component.  Fix a primitive $e_P^{th}$ root of
unity $\xi$. We view $\varSigma_{n/e_p}$ as the subgroup of
permutation matrices that permute the factors of this Levi subgroup
while the $i^{th}$  copy of
$C_{e_P}$ maps to the cyclic subgroup of $D_i$ generated by $\diag(1,
\xi, \xi^2, \ldots, \xi^{e_P-1})$.

We now define an action $\varrho$ of $\Waff_T$ on $\A(T, r)$.  
Taking $w\in\GL_n(k)$ a representative for $wT\in W_T$, $s = (s_1, \ldots, s_{n/e_P})  \in
T$, and $A\in (\ft^0/\ft^{r+1})^\vee$, we obtain actions of $W_T$ and
$T(F)/T(\fo)$ on $(\ft^0/\ft^{r+1})^\vee$ via 
\begin{align*}
\varrho (wT) (A) &= \Ad^*(w) (A) \\
\varrho (sT(\fo)) (A)  & = A -  \sum_{i = 1}^{n/e_P} \frac{\deg_E{s_i}}{e_P} \chi^\vee_i.
\end{align*}
Here,  $\chi_i^\vee$ is the functional induced
by $\chi_i \frac{dt}{t}$, where $\chi_i$ is the identity of the
$i^{th}$ component of $\ft$.  It is easy to see that these two actions
give rise to a unique action of  $\Waff_T$.  

To check that this action restricts to an action on $\A(T,r)$,
consider the action of $\Waff_T$ on $\ft$ defined by the similar formulas  $\varrho_\nu (wT) (x) = \Ad(w) (x)$
and $\varrho_\nu (s) (x) = x - \sum_{i = 1}^{n/e_P}
\frac{\deg_E{s_i}}{ e_P} \chi_i$.  The induced action on
$\ft^{-r}/\ft^1$ corresponds to $\varrho$ under the isomorphism
$\ft^{-r}/\ft^1\cong (\ft^0/\ft^{r+1})^\vee$ determined by
$\nu=\frac{dt}{t}$, and for any $\hat{w}\in\Waff_T$, $\varrho_\nu(\hat{w})(A_\nu)$ is a representative for
$\rho(\hat{w})(A)$.  If $A\in\A(T,r)$, then the leading term of
$A_\nu$ is regular, with distinct eigenvalues modulo $\Z$ when $r=0$.  If $r>0$, it is clear that
$\varrho_\nu(\hat{w})(A_\nu)$ also has regular leading term.  If
$r=0$, then the action permutes and adds integers to the eigenvalues,
so again the condition for being a formal type is preserved.




 \begin{proposition}\label{waffact}
   Suppose that $g \in N(T) $ and $A \in \A(T, r)$.  If $A_\nu \in \ft$ is a
   representative for $A$, then $g \cdot A_\nu$ is formally gauge
   equivalent to $\varrho_\nu (g T(\fo)) (A_\nu)$ by an element of $P^1$.
   Furthermore, $\varrho _\nu(g T(\fo)) (A_\nu) \in \pi_{\ft} (g
   \cdot A_\nu) + \ft^1$.
 \end{proposition}
\begin{proof}
  The case when $r = 0$ is easily checked since $T$ is the usual split
  torus, so we assume that $r >0$. First, consider
  the case $g=s\in T$, so $s \cdot A_\nu = A_\nu - (\tau s ) s^{-1}$.
  Recall that the intersection of $P$ with the block-diagonal Levi
  subgroup is a product of Iwahori subgroups $I_i \subset \GL_{e_P}
  (F)$.   Each $P_i$ determines an
  ordering on the roots of $D_i$.  We take $H_i \in \fd_i$ to be the
  half sum of positive coroots, and $H = (H_1, \ldots,
  H_{n/e_P})\in\fP$.
\begin{lemma}\label{Hlemma}
Suppose that  $s \in \ft^r$.  If $H$ is defined as above, then
\begin{equation*}
\tau (s) + \frac{1}{e_P} \ad (H) (s)- \frac{r}{e_P} s \in \fP^{1+r}.
\end{equation*}
Moreover, $\pi_{\ft} (\tau (s)) \in \frac{r}{e_P} s + \fP^{1+r}$.
\end{lemma}

\begin{proof}
  The first statement follows from the observation $\tau (\varpi_E^i)
  = \frac{i}{e_P} \varpi_E^i - \frac{1}{e_P} \ad (H_j) (\varpi^i_E).$
  We then obtain the second statement from Proposition~\ref{cores3}.

\end{proof}

Setting $s = (s_1, \ldots, s_{n/e_P})$, the lemma gives
\begin{equation*}
  (\tau s) s^{-1} \in \sum_{i = 1}^{n/e_P} \frac{\deg_E{s_i}}{e_P} \chi_i -\frac{1}{e_P}  \ad (H) (s) s^{-1} + \fP^1.
\end{equation*}
Observe that each term on the right of this expression lies in $\fP$.
Applying Proposition~\ref{cores}, we obtain $X \in \fP^r$ such that
$\ad (X) (A_\nu) \in \pi_{\ft} ((\tau s) s^{-1})-(\tau s) s^{-1} +
\fP^1$.  Taking $h = 1-X$, we see that $h \cdot (s \cdot A_\nu) \in
\pi_{\ft} (s \cdot A_\nu) + \fP^1$.  By Theorem~\ref{thm4}, it follows
that $s \cdot A_\nu$ is gauge equivalent to $\pi_{\ft} (s \cdot
A_\nu)$.

Since $\pi_\ft$ is a $\ft$-bimodule map, we deduce from
Lemma~\ref{adequation} that $\pi_{\ft} ( \ad (H) (s) s^{-1}) \in
\ft^1$.  Therefore, $\pi_{\ft} (A_\nu -(\tau s) s^{-1})
\in\varrho_\nu(sT(\fo)) (A_\nu)+\ft^1$.  Note that if $s \in T(\fo)$,
Lemma~\ref{taup1} implies that $(\tau s) s^{-1} \in \fP^1$; thus,
$T(\fo)$ does not affect the $P^1$-gauge equivalence class.

For the general case, take $g=sn$ with $n\in N(T)\cap\GL_n(k)$ and
$s\in T$.  The result now follows by applying the case above to the
formal type $\varrho_\nu(nT(\fo)(A_\nu)=\Ad(n)(A_\nu)=n\cdot
A_\nu=\pi_\ft(n\cdot A_\nu)\in\A(T,r)$.

\end{proof}

\begin{lemma}\label{Waff}
  Suppose that $A,A'\in\A(T,r)$ with representatives $A_\nu, A'_\nu
  \in \ft$.  If $A_\nu$ and $A'_\nu$ are $\GL_n(F)$-gauge equivalent
  modulo $\ft^1$, then there exists a unique $\hat{w} \in \Waff_T$ such that
  $\varrho(\hat{w}) ( A ) = A'$.
 \end{lemma}
\begin{proof}

  Since $\Waff_T$ acts freely on $\A(T,r)$, it suffices to show
  existence of $\hat{w}$.  First, take $r =0$, so $\ft$ is the space
  of diagonal matrices.  Write $A_\nu(0)$ for the image of $A_\nu$
  under the evaluation map $t \mapsto 0$.  Then, $A_\nu$ and $A'_\nu$
  are gauge equivalent modulo $t$ if and only if there exists $f \in
  \GL_n(\C)$ such that $\Ad(f) \exp(2 \pi i A_\nu (0)) = \exp (2 \pi i
  A'_\nu (0))$.  (This follows from the Riemann-Hilbert correspondence
  and \cite[Theorem 5.5 and Section 17.1]{Wa}).  Therefore, $f$ lies
  in the normalizer of $T$, and $\Ad(f) (A_\nu (0))$ differs from
  $A'_\nu(0)$ by a diagonal matrix with integer entries.  In
  particular, $A$ and $A'$ lie in the same $\Waff_T$ orbit.

  Now, assume $r > 0$.  Suppose that $h \cdot A'_\nu = A_\nu+x$ for
  some $x\in\fP^1$.  Fix a
  split torus $D$ with $D (\fo) \subset P$.  Using the affine Bruhat
  decomposition, we may write $h = p_1 n p_2$, where $p_1, p_2 \in P$
  and $n \in N(D)$.  We see that
\begin{equation*}
A_\nu + x = \Ad(h) (A'_\nu- p_2^{-1} \tau p_2) - (\tau p_1) p_1^{-1} - \Ad(p_1) ((\tau n) n^{-1}).
\end{equation*}
By Lemma~\ref{taup1}, $p_2^{-1} \tau p_2$ and $(\tau p_1) p_1^{-1}$
both lie in $\fP^1$.  Moreover, there exists $d \in D$ and $\sigma \in
N(D) \cap \GL_n(\C)$ such that $n = d \sigma$, so $(\tau n) n^{-1} =
(\tau d) d^{-1} \in \fd(\fo)\subset\fP$.  In particular, $\Ad(h)
(A'_\nu - p_2^{-1} \tau p_2) \in A_\nu + \Ad(p_1)((\tau d) d^{-1}) +
\fP^1$.  By Lemma~\ref{A1}, there exist $q_1 \in P^{r}$ and $q_2 \in
P^{r+1}$ such that $\ad(q_2) (A'_\nu - p_2^{-1} \tau p_2)$ and
$\ad(q_1^{-1}) (A_\nu + x + (\tau p_1) p_1^{-1} + \Ad(p_1) ((\tau n)
n^{-1}))$ lie in $\ft$.  Since $A'_\nu - p_2^{-1} \tau p_2$ is regular
by Proposition~\ref{lemrss}, it follows that $h \in q_1 N(T) q_2$.
 
Set $g = q_1^{-1} h q_2^{-1}$ and $\hat{w} = g T(\fo)$.  We will show that
$\varrho(\hat{w}) (A) = A'$.  The element $A^1_\nu = \pi_{\ft} ( q_2 \cdot A_\nu)\in \ft$ is a valid representative for $A$, since
$\pi_{\ft} (\Ad(q_2) (A_\nu))\in A_\nu + \fP^1$ and $(d q_2) q_2^{-1}
\in \fP^1$.  Moreover, the fact that $\pi_\ft$ is a $N(T)$-map implies
that $\pi_{\ft} ((g q_2) \cdot A_\nu) =\pi_{\ft}
(g \cdot A^1_\nu )$. 

By Proposition~\ref{waffact}, $ \pi_{\ft} (g \cdot A^1_\nu ) \in
\varrho_\nu (\hat{w}) (A^1_\nu) + \ft^1$, so $\pi_{\ft} ((g q_2 )
\cdot A_\nu) \in \varrho_\nu (\hat{w}) (A^1_\nu)+\ft_1$. 
Write $q_1 = 1 + X$ for $X \in \fP^r$ so that $(g q_2) \cdot A_\nu \in
A'_\nu - \ad(X) (A'_\nu) + \fP^1$.  It follows that $\pi_{\ft} ((g
q_2) \cdot A_\nu) \in A'_\nu + \ft^1$ by Proposition~\ref{cores3}.
Thus, $\hat{w} = g T(\fo)$ satisfies the Lemma.
\end{proof}

\begin{theorem}\label{waffthm}
  Suppose $(V,\nabla)$ is a formal connection.  If $A \in \A^{(V,
    \n)}_T\cap\A(T,r)$ and $A' \in \A^{(V, \n)}_{T'}\cap \A(T',r')$,
  then $T$ and $T'$ are $\GL_n(\fo)$-conjugate and $r=r'$.  Moreover,
  if $h\in\GL_n(\fo)$ satisfies ${}^hT=T'$, then there exists a unique
  $\hat{w} \in \Waff_T$ such that $A'=\Ad^*(h^{-1})\varrho(\hat{w}) (A)$.
\end{theorem}
\begin{proof} Proposition~\ref{pywit}
  shows that $r=r'$ and allows us to assume without loss of generality
  that $T=T'$.
  By definition of formal types, any choice of representatives $A_\nu$
  and $A'_\nu$ are formally gauge equivalent modulo $\fP^1$.  The
  theorem now follows from Lemma~\ref{Waff}. 

\end{proof}
\begin{corollary}
Suppose that $A \in \A^{(V, \n)}_T$.
Let $\phi$ be an associated trivialization, let $(P^\phi, r, \b^\phi)$
be the associated stratum (in $\GL(V)$), and let $L = \phi^{-1} (\fo^n)$.
Suppose $A' \in \A^{(V, \n)}_T$ has associated trivialization $\phi'$,
and choose $\hat{w} \in \Waff_T$  such that $A' = \varrho(\hat{w}) (A)$.
\begin{enumerate}
\item If $\phi'^{-1} (\fo^n) = L$, then $\hat{w} \in W_T$.
\item If, in addition, $(P^\phi, r, \b^\phi) = (P^{\phi'}, r, \b^{\phi'})$, 
then $\hat{w}$ is the identity.
\end{enumerate}
\end{corollary}
\begin{proof}
By Lemma~\ref{Waff}, there exists $\hat{w} \in \Waff_T$
such that $\varrho(\hat{w}) (A) = A'$.  Recall that
$\Waff_T \cong T/T(\fo) \ltimes W_T$, and $W_T \cong (N(T) \cap \GL_n(k)) / \Tfl$.
Observe that the set of trivializations of satisfying $\phi'^{-1} (\fo^n) = L$
is a single $\GL_n(\fo)$-orbit.  
Moreover,
$P^\phi = P^{\phi'}$ by Proposition~\ref{uniquedet}.
Since $P$ is its own normalizer in $\GL_n(\fo)$, it follows that 
$A_\nu$ and $A'_\nu$ are gauge equivalent by an element $p \in P$
in both cases.

In the second case, the stabilizer of $\b + \fP^{1-r}$ in $P$ is equal to
$P^1 \Tfl$.  Without loss of generality, take $p \in P^1$.  Now, the uniqueness
statement in Theorem~\ref{thm4} implies that $A = A'$.

In the first case, as long as $r>0$, we have $A'_\nu \in  \Ad(n) (A_\nu) + \fP^{-r+1}$, where
$n \in N(T)$ is a representative for $\hat{w}$.   Since $n\in mT(\fo)$ for some $m\in N(T) \cap \GL_n(k) \subset P$,  we also $A'_\nu \in  \Ad(m) (A_\nu) + \fP^{-r+1}$.  We deduce that 
$p \in m P^1$, and the same uniqueness argument shows that 
$A'_\nu = \Ad(m) (A_\nu)$.  This implies that
$A' = \varrho(m T(\fo)) (A)$, and simple transitivity of the $\Waff_T$-action gives $\hat{w}=\varrho(mT(\fo))\in W_T$.   If $r=0$, then let $m\in\GL_n(k)$ be the image of $p$ modulo $t$.  Since $A'_\nu\in\Ad(p)(A_\nu)+t\gl_n(\fo)$, we have $A'_\nu=\Ad(m)(A_\nu)$.  We conclude that $\hat{w}=\varrho(mT(\fo))\in W_T$ as before.

\end{proof}

We can now make precise the relationship between formal types and
formal isomorphism classes in terms of moduli spaces of certain
categories of formal connections.
In the following, we consider three related categories of formal
connections.  
We define $\sC$ to be the full subcategory of formal
connections $(V, \n)$ of rank $n$ such that $(V, \n)$ contains
a regular stratum. 
Let
$\sC^{lat}$ be the category of triples $(  V, \n, L)$, where $V \in \sC$ and $L
\subset V$ is a distinguished $\fo$-lattice such such that $P \subset
\GL(L)$ for some regular stratum $(P, r, \b)$ contained in $V$.  Morphisms
in $\Hom_{\sC^{lat}}( (V, \n, L), ( V', \n', L'))$ in $\sC^{lat}$ consist of
homomorphisms $\phi : V \to V'$ (in the category $\sC$) such that $L'
\cap \phi (V) = \phi (L)$.  Note that if $\phi$ is an isomorphism,
this implies that $\phi (L) = L'$.   Finally, $\sC^{fr}$ is the
category of \emph{framed} connections.  This consists of objects in $\sC^{lat}$
where $ \Hom_{\sC^{fr}} ( (V, \n, L), ( V', \n', L'))$ is the set of
isomorphisms $\phi \in \Hom_{\sC^{lat}} ( ( V, \n, L), ( V', \n', L'))$ such that
$(\phi^{-1} (P'), r, \phi^*(\b')) = (P, r, \b)$.

Fix a uniform torus $T \subset \GL_n(F)$ satisfying $T(\fo)\subset\GL_n(\fo)$ and an integer $r \ge 0$.
We denote the full subcategory of $\sC$ (resp. $ \sC^{lat}$, $\sC^{fr}$)
of connections that have formal type in $\A(T, r)$ by
$\sC(T, r)$ (resp. $\sC^{lat} (T,r)$, $\sC^{fr} (T, r)$).  Proposition~\ref{pywit}
implies that the subcategory $\sC(T,r)$ only depends on the conjugacy 
class of $T$ in $\GL_n(F)$.  It follows from Theorem~\ref{waffthm} that the set of objects in $\sC$  is the disjoint union of 
objects in $\sC(T_i,r)$, taken over a set of representatives $T_i$ as above
for the conjugacy classes of uniform tori.  The analogous statement
holds for $\sC^{lat}$ and $\sC^{fr}$.

\begin{corollary}\label{catcor}
  Fix a uniform torus $T\subset \GL_n(F)$ with $T(\fo)\subset\GL_n(\fo)$
  and $r \in \Z^{\ge 0}$.  Then, $\A(T,r)$ is the moduli space for
  $\sC^{fr}(T, r)$, $\A(T, r) / W_T$ is the moduli space for
  $\sC^{lat}(T,r)$, and $\A(T, r) / \Waff_T$ is the moduli space for
  the category $\sC(T,r)$.
\end{corollary}

\section{Moduli Spaces}\label{modspace}

In this section, we will describe the moduli space $\mathscr{M} (A^1,
\ldots, A^m)$ of `framed' connections on $C = \proj^1(\cplx)$ with
singular points $\{ x_1, \ldots, x_m \}$ and formal type $A^i$ at
$x_i$.  In our explicit construction, we show that this moduli space is
the Hamiltonian reduction of a symplectic manifold via a torus
action.  

Set $k = \cplx$. We denote by $F_x \cong F$ the field of Laurent
series at $x \in C$ and $\mathfrak{o}_x \subset F_x$ the ring of power
series.  Let $V$ be a trivializable rank $n$ vector bundle on
$\proj^1$; thus, there is a noncanonical identification of $V$ with
the trivial rank $n$ vector bundle $V^\triv\cong\struct_{\proj^1}^n$.
The space of global trivializations of $V$ is a $\GL_n(\cplx)$-torsor,
so we will fix a base point and identify each trivialization $\phi$
with an element $g \in \GL_n(\cplx)$.  Thus, we will write $[\nabla]$
for the matrix of $\nabla$ in the fixed trivialization, and $g \cdot
[\nabla]$ for $[\nabla]_\phi$.

Define $V_x = V \otimes_{\struct_\glo} F_x$ and $L_x = V
\otimes_{\struct_\glo} \fo_x$.  The inclusion $V_\cplx = \Gamma
(\proj^1; V) \subset V_x$ gives $V_x$ a natural $\cplx$-structure.
Furthermore, $L_x$ determines a unique maximal parahoric $G_x =
\GL(L_x) \cong \GL_n(\fo)$.  In particular, by the remarks preceding
Lemma~\ref{subquotient}, there is a one-to-one correspondence between
parahoric subgroups $P \subset G_x$ and parabolic subgroups $Q \subset
\GL(V_\cplx)$, where $Q = P / G_x^1$.

  Let  $T_x$ be a uniform torus in $\GL_n(F_x)$
such that $T_x(\fo_x) \subset \GL_n(\fo_x)$, and set $P_x = P_{T_x, \fo_x^n}$.
  In the following,  $(V,\nabla)$ is a
connection on $C$, and  $A_x \in \A(T_x, r)$ is a formal type
associated to $(V, \n)$ at $x$.  This means that the
formal completion $(V_x,\nabla_x)$ at $x$ has formal type $A_x$.  We
denote the corresponding $\GL_n(F_x)$-stratum by $(P_x, r, \beta_x)$.
We may assume, by Proposition~\ref{pywit}, that $T_x$
has a block-diagonal embedding in $\GL_n(F_x)$ as in Remark~\ref{blockdiagonal}.
We write $U_x = P^1_x / G^1_x$.  Furthermore, if $g \in
\GL_n(\cplx)$, $P^g_x \subset \GL(V)$ and $\beta^g_x$ are the
pullbacks of $P$ and $\beta$, respectively, under the corresponding
trivialization (as in Section~\ref{formaltypes}).

\begin{definition}\label{framingdef}
A \emph{compatible framing} for $\nabla$ at $x$ is an element  $g \in \GL_n (\cplx)$
with the property that $\nabla$ contains the $\GL(V_x)$-stratum $(P^g_x, r, \beta^g_x )$
defined above.
We say that $\nabla$ is \emph{framable} at $x$ if there exists such a $g$.
\end{definition}
For example, suppose that $e_{P_0} = 1$. Choose
$\nu \in \Omega_0^\times$ of order $-1$.  By
Remark~\ref{blockdiagonal}, $A_{0\nu} = \frac{1}{t^r} D_r +
\frac{1}{t^{r-1}} D_{r-1} + \dots$ where $D_j\in\GL_n(\cplx)$ are
diagonal matrices and $D_r$ is regular.  It follows that $g$ is a
compatible framing for $(V,\nabla)$ at $0$ if and only if
\begin{equation*}
g \cdot [\nabla] = \frac{1}{t^r} D_r \nu + \frac{1}{t^{r-1}} M_{r-1} \nu + \dots,
\end{equation*}
with $M_j \in \gl_n(\cplx)$.

Now, let $\mathbf{A} = (A_1, \dots, A_m)$ be a collection of formal
types $A_i$ at points $x_i \in \proj^1$.
\begin{definition}\label{mod}
The category $\mathscr{C}^* (\mathbf{A})$ of framable connections
with formal types $\mathbf{A}$ is the category whose objects are
connections $(V, \nabla)$, where
\begin{itemize}
\item $V$ is a trivializable rank $n$ vector bundle on $\proj^1$;
\item $\nabla$ is a meromorphic connection on $V$ with singular points
  $\{x_i\}$;
\item $\nabla$ is framable and has formal type $A^i$ at $x_i$;
\end{itemize}
and whose morphisms are vector bundle maps compatible with the
connections.
The moduli space of this category is denoted  by $\mathscr{M}^* (\mathbf{A})$.
\end{definition}

By Corollary~\ref{ftgauge}, any two objects in $\mathscr{C}^*
(\mathbf{A})$ correspond to connections that are formally isomorphic
at each $x_i$. Note that $\mathscr{C}^* (\mathbf{A})$ is not a full
subcategory of the category of meromorphic connections.  However, the
next proposition show that the moduli space of this full subcategory
coincides with $\mathscr{M}^* (\mathbf{A})$, so this moduli space may
be viewed as a well-behaved subspace of the moduli stack of
meromorphic connections.


\begin{proposition}
  Suppose that $(V, \n)$ and $(V', \n')$ are framable connections in
  $\mathscr{C}^*(\bfA)$.  If they are isomorphic as
  meromorphic connections, then they are isomorphic as framable
  connections.
\end{proposition}
\begin{proof}
  Choose trivializations for $V$ and $V'$.  Then, $(V, \n)$ and $(V',
  \n')$ are isomorphic as meromorphic connections if and only if there
  exists a meromorphic section $g$ of the trivial $\GL_n(\C)$-bundle
  such that $g \cdot [\n] = [\n']$.  Moreover, $g$ is necessarily
  regular at all points of $\proj^1 \backslash \{x_1, \ldots, x_m\}$.
  It suffices to show that $g$ is regular at each of the the singular
  points of $\n$.  Thus, we may reduce to the following local problem:
  if $\n$ and $\n'$ are formal framed connections, $g \cdot [\n] =
  [\n']$, and $\n$ and $\n'$ have the same formal
  type,  then $g \in \GL_n(\fo)$.
  
  Fix $\nu = \frac{dt}{t}$.  By Theorem~\ref{thm4}, there exist
  $g_1,g_2\in\GL_n(\fo)$ such that $g_1 \cdot [\nt] = g_2 \cdot[\nt']
  = A_\nu$. Therefore, $(g_2 g g_1^{-1}) \cdot A_\nu = A_\nu$.  By
  Corollary~\ref{newlemma}, this implies that $g_2 g g_1^{-1} \in
  \Tfl$.  It follows that $g \in \GL_n(\fo)$.

\end{proof}

We will construct $\mathscr{M}^* (\mathbf{A})$ using symplectic
reduction, so in general $\mathscr{M}^* (\mathbf{A})$ will not be a
manifold.  Following Section 2 of \cite{Boa}, we define an extended
moduli space $\widetilde{\mathscr{M}}^*(\mathbf{A})$ that resolves
$\mathscr{M}^* (\mathbf{A})$.

\begin{definition}\label{xmod}
  The category $\widetilde{\mathscr{C}}^*(\mathbf{A})$ of
  \emph{framed} connections with formal types $\mathbf{A}$ has objects
  consisting of triples $(V, \nabla, \mathbf{g})$, where 
\begin{itemize} \item  $(V, \nabla)$
  satisfies the first two conditions of Definition~\ref{mod};
\item $\mathbf{g} = (U_{x_1}g_1, \ldots, U_{x_m}g_m)$, where $g_i$ is a
  compatible framing for $\nabla$ at $x_i$;
\item the formal type $(A')^i$ of $\nabla$ at $x_i$ satisfies
  $(A')^i|_{\ft^1}=A^i|_{\ft^1}$. 
\end{itemize}
A morphism between $(V, \nabla, \mathbf{g})$ and $(V', \nabla',
\mathbf{g}')$ is a vector bundle isomorphism $\phi : V \to V'$ that is
compatible with $\n$ and $\n'$, with the added condition that
$(\phi_{x_i}^{-1}(P'^{g'_i}_{x_i}), r, \phi_{x_i}^*((\b'_{x_i})^{g'_i})) =
(P^{g_i}_{x_i}, r, \b^{g_i}_{x_i})$ for all $i$.  We let
$\widetilde{\mathscr{M}}^*(\mathbf{A})$ denote the corresponding
moduli space.
\end{definition}


\begin{rmk}
  Define $\Waff_{T_{x_i}}$ and $W_{T_{x_i}}$ as in
  Section~\ref{formaliso}.  The groups $\bfW = \prod_{x_i}
  W_{T_{x_i}}$ and $\bfWa = \prod_{x_i} \Waff_{T_{x_i}}$ act
  componentwise on $\prod_{x_i} \A(T_{x_i}, r_{x_i})$.  We note that a
  global connection $(V, \n)$ lies in $\sC^*(\bfA)$ if $(V_{x_i},
  \n_{x_i})$ is isomorphic to the diagonalized connection $(F_{x_i}^n,
  d + A^i_\nu \nu)$ in $\sC^{lat}$.  It follows from
  Corollary~\ref{catcor} that the categories $\sC^* ( \bfA')$ and
  $\sC^*(\bfA)$ have the same objects if and only if $\bfA' = \bfw
  \bfA$ for some $\bfw \in \bfW$.    In particular, $\sM^*(\bfw \bfA) \cong \sM^* (\bfA)$.
  If we let $j$ denote the injection of these spaces into the moduli
  space of meromorphic connections, then $j(\sM^*(\bfw \bfA))=j(\sM^*
  (\bfA))$.  On the other hand, if $\bfA'$ is not in the $\bfWa$-orbit
  of $\bfA$, then $j(\sM^*(\bfA'))$ and  $j(\sM^* (\bfA))$ are
  disjoint.  This is because connections in the corresponding
  categories are not even formally isomorphic by
  Theorem~\ref{waffthm}.

  Now, suppose that $\bfs = (s_1, \ldots, s_m) \in \bfWa$ and $ s_i
  \in T_{x_i} (F) / T_{x_i} (\fo)$.  In this case, $\sC^*( \bfs \bfA)
  \ne \sC^* (\bfA)$ unless $s_i$ is the identity.  However, it is
  clear that $\varrho(s_i) (A^i)|_{\ft^1} = A^i |_{\ft^1}$.  We deduce
  that $\widetilde{\sC}^*(\bhw \bfA) = \widetilde{\sC}^*( \bfA)$, and
  $\tM^*(\bhw \bfA) \cong \tM^* (\bfA)$ for all $\bhw \in \bfWa$.
(Indeed, $\widetilde{\sC}^*(\bfA') = \widetilde{\sC}^*(\bfA)$
  if and only if for every $x_i$ 
there exists $\hat{w}_i \in \Waff_{T_{x_i}}$ such that 
$\varrho(\hat{w}_i)((A')^i) |_{\ft^1} = A^i|_{\ft^1}$.)


\end{rmk}

Let $X$ be a symplectic variety with a Hamiltonian action of 
a group $G$.  There is a moment map $\mu_G : X \to \mathfrak{g}^\vee$.
If $\alpha \in \mathfrak{g}^\vee$ lies in $[\mathfrak{g}, \mathfrak{g}]^\perp$, 
so that the coadjoint orbit of $\alpha$ is a singleton, then
the symplectic reduction $X \sslash_\alpha G$  is defined to be the quotient
$\mu^{-1} (\alpha) / G$.  

In Section~\ref{extended orbits}, we will use the formal type $A_i$ at
$x_i$ to define \emph{extended orbits} $\mathscr{M}_i$ and
$\widetilde{\mathscr{M}}_i$.  These are smooth symplectic manifolds
with a Hamiltonian action of $\GL_n(\cplx)$.  The following theorem
generalizes \cite[Proposition 2.1]{Boa}:
\begin{theorem}\label{modthm} 
  Let   $\mathscr{M}^* (\mathbf{A})$, $\widetilde{\mathscr{M}}^* (\mathbf{A})$
  be the moduli spaces defined above.
  \begin{enumerate}
  \item The moduli space  $\mathscr{M}^* (\mathbf{A})$ is a symplectic reduction of $\prod_{i} \mathscr{M}_i$:  
\begin{equation*}
\mathscr{M}^*(\mathbf{A}) \cong (\prod_{i} \mathscr{M}_i) \sslash_0 \GL_n(\cplx).
\end{equation*}
\item Similarly,
\begin{equation*}
\widetilde{\mathscr{M}}^*(\mathbf{A}) \cong (\prod_{i} \widetilde{\mathscr{M}}_i) \sslash_0 \GL_n(\cplx).
\end{equation*}
Moreover, $\widetilde{\mathscr{M}}^* (\mathbf{A})$ is a symplectic manifold.
\item  Let $T_i = T_{x_i}$. 
There is a Hamiltonian action of $\Tfl_i$
on $\widetilde{\mathscr{M}}^*( \mathbf{A})$, and $\mathscr{M}^*(\mathbf{A})$
is naturally a symplectic reduction of $\widetilde{\mathscr{M}}^* (\mathbf{A})$ 
by the group $\prod_i \Tfl_i$.
\end{enumerate}
\end{theorem}
This theorem will be proved in Section~\ref{proof}.

\begin{rmk} We also obtain a version of this theorem when additional
  singularities corresponding to regular singular points are allowed
  (Theorem~\ref{thm3}).  In the case of the Katz-Frenkel-Gross
  connection~\cite{Katz,FGr}, the moduli space reduces to a point,
  consistent with the rigidity of this connection.  Theorem~\ref{thm3}
  also allows one to construct many other examples of connections with
  singleton moduli spaces, which are thus plausible candidates for
  rigidity.

\end{rmk}

\begin{rmk}



It is not surprising that these moduli spaces  are
symplectic: it is conceivable that this fact might be proved independently
using the abstract methods of \cite[Section 6]{Fed}.
The advantage
of Theorem~\ref{modthm} is that it gives explicit constructions of
$\sM^*(\bfA)$ and $\tM^* (\bfA)$ in a number of important, novel cases (including
connections with `supercuspidal' type singularities).  Moreover, the
fact that $\tM^*(\bfA)$ is smooth allows one to generalize the work of
Jimbo, Miwa, and Ueno \cite{JMU} and explicitly calculate the
isomondromy equations in these cases (see \cite{BrSa2}).
\end{rmk}

\subsection{Extended Orbits}\label{extended orbits}

In this section, we will construct symplectic manifolds, called
extended orbits, which will be ``local pieces'' of the moduli spaces
$\mathscr{M}^* (\mathbf{A})$ and $\widetilde{\mathscr{M}}^*
(\mathbf{A})$.   Without loss of generality, we will take our singular
point to be $x=0$, and we will suppress the subscript $x$ from $F_x$,
$P_x$, $A_x$, etc. 

Our study of extended orbits is motivated by the relationship between
coadjoint orbits and gauge transformations.  In the following, fix $\nu \in \Omega^\times_0$.
\begin{proposition}\label{gaugecoadjoint}
The map $\gl_n(F) \to \fP^\vee$ determined by $\nu$
intertwines the gauge action of $P$ on $\gl_n(F)$ with the coadjoint action
of $P$ on $\fP^\vee$.
\end{proposition}
\begin{proof}
  Recall from Lemma~\ref{subquotient} that $P=H\ltimes P^1$ for
  $H\cong Q/U$ a Levi subgroup of $\GL_n(\cplx)$.  Thus, we may write
  any element of $P$ as $p = h u$, for $h \in H$ and $u \in P^1$.
  Without loss of generality, we may assume that $\nu$ has order $-1$.
  Thus,
\begin{equation*}
p \cdot X = h\cdot (\Ad(u) (X) - \tau (u) u^{-1}) = \Ad(p) (X) - \Ad(h) (\tau(u) u^{-1}).
\end{equation*}
Lemma~\ref{taup1} shows that $\Ad(h) (\tau(u) u^{-1}) \in
\mathfrak{P}^1 = \mathfrak{P}^\perp$.  Applying
Proposition~\ref{duality}, we see that
$\Ad^*(p)(\langle X,\cdot\rangle_\nu\vert_{\fP})=\langle
\Ad(p)(X),\cdot\rangle_\nu\vert_{\fP}= \langle p\cdot
X,\cdot\rangle_\nu\vert_{\fP}$.
\end{proof}

From now on, we assume $\nu$ has order $-1$.  We suppose that $A$ is a
formal type at $0$ stabilized by a torus $T$ with $T(\mathfrak{o})
\subset P$ and that the corresponding regular stratum $(P, r, \b)$ has
$r > 0$.  In particular, any connection with formal type $A$ is
irregular singular.  Denote the projection of $\pi^*_\ft(A)$ onto
$(\mathfrak{P}^1)^\vee$ by $A^1$.  Let $G = \GL_n(\mathfrak{o})$ be
the maximal standard parahoric subgroup at $0$ with congruence
subgroups (resp. fractional ideals) $G^i$ (resp.  $\fg^i$).  Then,
$G^1 \subset P \subset G$, and $P / G^1 \cong Q$.
For any subgroup $H\subset G$ with Lie algebra $\fh$, there is a
natural projection $\pi_{\fh} : \mathfrak{g}^\vee \to (\fh)^\vee$
obtained by restricting functionals to $\fh\subset \mathfrak{g}$.
Denote the $P$-coadjoint orbit of $\pi^*_\ft(A)$
by $\orbit$, and the $P^1$-coadjoint orbit of ${A}^1$ by $\orbit^1$.

\begin{definition}
  Let $A$ be a formal type at $0$ with irregular singularity, and
  let $U$ be the unipotent radical of $Q$. We
  define the extended orbits $\mathscr{M}(A)$ and
  $\widetilde{\mathscr{M}} (A)$ by
\begin{itemize}
\item $\mathscr{M}(A) \subset (Q \backslash \GL_n(\cplx)) \times \mathfrak{g}^\vee$
is the subvariety defined by
\begin{equation}\label{madef}
\mathscr{M}(A) = \{ (Q g, \alpha) \mid \pi_{\mathfrak{P}} (\Ad^*(g) (\alpha)) \in \orbit) \};
\end{equation}
\item $\widetilde{\mathscr{M}} (A) \subset (U \backslash
  \GL_n(\cplx)) \times \mathfrak{g}^\vee$ is defined by
\begin{equation*}\label{matwiddledef}
\widetilde{\mathscr{M}}(A) = \{ (U g, \alpha) \mid \pi_{\mathfrak{P}^1} (\Ad^*(g) (\alpha)) \in \orbit^1) \};
\end{equation*}
\end{itemize}
\end{definition}

\begin{proposition}\label{framing}
The extended orbits $\mathscr{M} (A)$ and $\widetilde{\mathscr{M}}(A)$ are isomorphic to
symplectic reductions of $T^* G \times \orbit$ and $T^* G \times
\orbit^1$ respectively:
\begin{equation*}\begin{aligned}
\mathscr{M}(A) & \cong T^*G \times \orbit \sslash_0 P \\
\widetilde{\mathscr{M}}(A) & \cong T^*G \times \orbit^1 \sslash_0 P^1.
\end{aligned}
\end{equation*}
In particular, the natural symplectic form on $T^*G \times \orbit$
descends to both  $\mathscr{M}(A)$ and $\widetilde{\mathscr{M}}(A)$.
Moreover, 
$\mathscr{M}(A)$ and $\widetilde{\mathscr{M}}(A)$ are smooth
symplectic manifolds.
\end{proposition}
\begin{rmk}
  Note that $T^*G$ is not finite dimensional.  However, for $\ell$
  sufficiently large, $A \in (\mathfrak{g}^\ell)^\perp$.  Since
  $G^\ell \subset P^1$, we see that $G/P^1 \cong
  (G/G^\ell)/(P^1/G^\ell)$.  Thus, in Proposition~\ref{framing}, it
  suffices to consider $T^*(G/G^\ell) \times \orbit \sslash_0 P$
  (resp.  $T^*(G/G^\ell) \times \orbit^1 \sslash_0 P^1$).  This fact,
  although concealed in our notation, ensures that we are always
  applying results from algebraic and symplectic geometry to
  finite-dimensional varieties.
\end{rmk}
\begin{proof}
  The proof in each case is similar, so we will prove the second isomorphism.
  The group $P^1$ acts on $T^* G$ by the usual left action $p (g,
  \alpha) = (p g, \alpha)$ and on $\orbit^1$ by the coadjoint
  action.  Moreover, on each factor, the action of $P^1$ is
  Hamiltonian with respect to the standard symplectic form.  The
  moment map for the diagonal action of $P^1$ is the sum of the two
  moment maps:
\begin{equation*}\begin{aligned}
&\mu_{P^1} : T^* G \times \orbit^1  \to (\mathfrak{P}^1)^\vee \\
&\mu_{P^1} (g, \alpha, \beta)  =  \pi_{\mathfrak{P}^1} (-\Ad^*(g) (\alpha)) + \beta.
\end{aligned}
\end{equation*}
In particular, 
\begin{equation*}
\mu_{P^1}^{-1} (0) = \{(g, \alpha, \beta) \mid \pi_{\mathfrak{P}^1}(\Ad^*(g) (\alpha)) = \beta\}.
\end{equation*}
We will show that $\mu_{P^1}^{-1}(0)$ is smooth.
Let $\varphi: \mu_{P^1}^{-1} (0) \to G \times \orbit^1$ be defined by
$\varphi (g, \alpha, \beta) = (g, \beta)$.  Choose a local section 
$ f: \orbit^1 \to \orbit^1 + (\fP^1)^\perp \subset \fg^\vee$.
Then, $\varphi^{-1} (g, \beta) = 
\{(g, \Ad^*(g^{-1})(f(\beta)+X) , \beta) \mid X \in  (\fP^1)^\perp \}$.
Therefore, $\mu_{P^1}^{-1} (0)$ is an affine bundle over $G \times \orbit^1$
with fibers isomorphic to $\fP/\fg^1$.  It follows that $\mu_{P^1}^{-1} (0)$
is smooth.

Since $U \cong P^1 / G^1$, $U \backslash \GL_n(\cplx) \cong
P^1 \backslash G$.  Therefore, the map $(g, \alpha, \beta) \mapsto
(P^1 g, \alpha)$ takes $\mu_{P^1}^{-1} (0)$ to
$\widetilde{\mathscr{M}}(A)$; moreover, the fibers are $P^1$ orbits,
so the map identifies $P^1 \backslash \mu_{P^1}^{-1} (0) \cong
\widetilde{\mathscr{M}}(A)$.

Finally, choose a local section $\zeta : P^1 \backslash G \to G$ with
domain $W$, and let $W' = \widetilde{\mathscr{M}} (A) \cap (W \times
\fg^\vee)$.  There is a section $\zeta' : W' \to \mu_{P^1}^{-1} (0)$
given by $\zeta' (P^1 g, \alpha) = (\zeta'(P^1 g), \alpha, \pi_{\fP^1}
(\Ad^*(\zeta(P^1 g)) (\alpha)))$.  This shows that $\mu_{P^1}^{-1} (0)
\to \widetilde{\mathscr{M}}(A)$ is a principal $P^1$-bundle, since
$P^1$ acts freely on the fibers.  In particular,
$\widetilde{\mathscr{M}}(A)$ is smooth, and the symplectic form on
$T^* G \times \orbit^1$ descends to $\widetilde{\mathscr{M}}(A)$.
\end{proof}

Let $\res : \mathfrak{g}^\vee \to \gl_n(\cplx)^\vee$ be the
restriction map dual to the inclusion $\gl_n(\cplx) \to \mathfrak{g}$.
Notice that if we fix a representative $\alpha_\nu \in \gl_n(F)$ for
$\alpha \in \mathfrak{g}^\vee$, then $\gl_n(\cplx)^\vee \cong
\gl_n(\cplx)$ under the trace pairing and $\res(\alpha)$ corresponds
to the ordinary residue of $\alpha_\nu \nu$.

There is a Hamiltonian left action of $\GL_n(\cplx)$ on $T^*G$ defined
by
\begin{equation}\label{rtact}
\rho (h) (g, \alpha) = (g h^{-1}, \Ad^*(h) \alpha).
\end{equation}
The moment map $\mu_\rho$ is given by $\mu_\rho (g, \alpha) =
\res(\alpha)$.  To see this, observe that $\rho$ is the restriction to
$\GL_n(\cplx)$ of the usual left action of $G$ on $T^*G$ (via
inversion composed with right multiplication).  Hence, the map
$\mu_\rho$ is just the composition of the moment map for right
multiplication $\mu (g, \alpha) = \alpha$ with $\res$.

The action $\rho$ defines left actions of $\GL_n(\cplx)$ on the 
first components of 
$T^*G \times \orbit$ and $T^*G \times \orbit^1$ respectively.  These
actions commute with the left actions of $P$ and $P^1$, and 
it is clear that $\mu_P$ and $\mu_{P^1}$ are
$\GL_n(\cplx)$-equivariant.
\begin{lemma}\label{moeq}
Let $G_1$ and $G_2$ act on a symplectic manifold $X$ via Hamiltonian
actions, and let $\mu_1$ and $\mu_2$ be the corresponding moment maps.  
If $\mu_2$ is $G_1$-invariant on $\mu_1^{-1} (\lambda)$, 
then there is a natural Hamiltonian action of $G_2$ on $X \sslash_\lambda G_1$.
Furthermore, if $\iota_\lambda : \mu_{1}^{-1}  (\lambda) \to X$ and 
$\pi_\lambda : \mu_{1}^{-1} (\lambda) \to X \sslash_\lambda G_1$ are
the natural maps, then the induced moment map $\bar{\mu}_{2}$ on $X \sslash_\lambda G_1$
is the unique map satisfying $\mu_{2} \circ \iota_\lambda = \bar{\mu}_{2} \circ \pi_\lambda$.
\end{lemma}
This follows from \cite[Theorem 4.3.5]{AM}.  Thus, $\rho$ descends to
natural Hamiltonian actions on $\mathscr{M}(A)$ and
$\widetilde{\mathscr{M}}(A)$.  For example, if $(Q g, \alpha) \in
\mathscr{M}(A)$ and $h\in\GL_n(\cplx)$, then
\begin{equation*}
h(Q g, \alpha)  = (Q g h^{-1}, \Ad^*(h) (\alpha)).
\end{equation*}

\begin{proposition}\label{glmoment}
The moment map for the action of $\GL_n(\cplx)$ on 
$\mathscr{M}(A)$ is given by
\begin{equation*}
\mu_{\GL_n} (Q g, \alpha)) = 
\res (\alpha).
\end{equation*}
The action of $\GL_n(\cplx)$ on $\widetilde{\mathscr{M}}(A)$ 
has the analogous moment map $\tmug$.
\end{proposition}
\begin{proof}
This follows directly from Lemma~\ref{moeq}.
\end{proof}
\begin{lemma}\label{submersion}
  The moment map $\tmug:\widetilde{\mathscr{M}}(A) \to \gl_n(\cplx)$ is
  a submersion.
\end{lemma}
\begin{proof}
  By Proposition~\ref{framing}, $\widetilde{\mathscr{M}}(A)$ is
  smooth.  We will show that the differential map $d \tmug$ on
  tangent spaces is surjective.  Note that $\tmug(Ug g',
  \Ad^*(g'^{-1}) \alpha) = g'^{-1} \tmug (U g, \alpha)$.
  Therefore, it suffices to show that the tangent map is surjective at
  points $s=(U,\a)$ in the subvariety $S$ defined by taking $g$ to be the
  identity.

  Let $\fu = \Lie(U)$, so $\fu^\perp \subset \gl_n(\cplx)^\vee$.
  Indeed, $\fu^\perp \cong (\fP^1)^\perp \subset \fg^\vee$.  If we
  choose a section $f : (\fP^1)^\vee \to \fg^\vee,$ we see that
  $\orbit^1 \times \fu^\perp \cong S$ by the map $(\gamma, y) \mapsto
  (U, f(\gamma) +y)$.  Here, the image of $y$ is identified with its
  image in $\fg^\vee$.  In particular, the image of $d \tmug (T_s
  \widetilde{\mathscr{M}}(A))$ contains $\fu^\perp \subset
  \gl_n(\cplx)^\vee$.  Therefore, it suffices to show that the
  composition of $d \tmug$ with the quotient $\gl_n(\cplx)^\vee \to
  \fu^\vee$ is surjective.  Observe that tangent vectors to $\orbit^1$
  are of the form $\ad^*(X) (\alpha)$ for $X \in \fP^1$.

  First, suppose that $r>e_P$.  In this case, we will show that
  $\ad^*(\fP^{1-e_P+r})(\a)\subset T_s\orbit^1$ surjects onto
  $\fu^\vee$.  More precisely, we will construct a filtration
  $\fu^\vee=\fu^1\supset\fu^2\supset\dots\supset\fu^{e_P}=\{0\}$ such
  that the map $X\mapsto \ad^*(X)(\a)$ induces a surjection
  $\bpi^j_\fu:\bfP^{j-e_P+r}\to\fu^j/\fu^{j+1}$ for each $j$.  Since
  $\fu\cong \fP^1/\fg^1$, we see that $\fu^\vee \cong \fg/ \fP$ under
  the duality isomorphism.  We now obtain the desired filtration on
  $\fu^\vee$ by subspaces of the form $\fu^j \cong (\fP^{j-e_P} \cap
  \fg) / \fP$.  More explicitly, $\fu^j$ is the restriction of
  $(\fP^{e_P-j+1})^\perp\subset (\fP^1)^\vee$ to $\fu$.  Note that the
  map $\bpi_\fu^j=\tau^j\circ(-\delta_{\an})$, where
  $\tau^j:\bfP^{j-e_P}\to\fu^j/\fu^{j+1}$ is the surjection defined by
  $\tau^j(X)=(\langle X,\cdot\rangle_\nu|_\fu)+\fu^{j+1}$.  
  Furthermore,  $\bpi_\fu^j$ depends only on the coset 
  $\an+ \fP^{-r+1}$.

  By assumption, $\alpha_{\nu} \in A_{ \nu} + \fP^{-r+1}$.
  Proposition~\ref{cores} shows that
  $\delta_{\an}(\bfP^{j-e_P+r})=\ker(\bar{\pi}_\ft)$.  Since
  $\bar{\pi}_\ft:\bfP^{j-e_P}\to \bft^{j-e_P}$ is a surjection, a
  diagram chase shows that $\tau^j|_{\ker(\bar{\pi}_\ft)}$ is
  surjective if and only if $\bar{\pi}_\ft|_{\ker(\tau^j)}$ is
  surjective.  We now verify this last statement.  In the case $e_P =
  n$, recall the description of $\varpi_E$ from \eqref{varpip}.  It is
  easily checked that $Y_{j\nu} = t^{-1} \Res (\varpi_E^{j-e_P} dt)$
  corresponds to a non-zero element $Y_j$ of $\ker(\bar{\pi}^j_\fu)$
  (since $1 \le j <e_P$) and $\pi_\ft (Y_{j\nu}) = \frac{e_P- j}{e_P}
  \varpi_E^{j-e_P}$.  Therefore, the span of $Y_{j\nu}$ surjects onto the
  one-dimensional space $\bft^{j-e_P}$.  A similar proof works for
  $e_P < n$, using the observation in Corollary~\ref{uniformsplitting}
  that $\ft \cong E^{n/e_P}$.  

  Now, assume that $1\le r\le e_P$.  The above argument shows that
  $\bpi^j_\fu(\bfP^{j-e_P+r})=\fu^j/\fu^{j+1}$; however, in this case,
  we can only conclude that the image of $d \tmug|_{T_s\orbit^1}$
  contains $\fu^{e_{P}- r + 1}$.  Let $\fw^j = \fP^{j-e_P+r}\cap
  \gl_n(\cplx)$; it follows that $\fw^j / \fw^{j+1}$ determines a
  well-defined subset of  $\bfP^{j-e_P + r}$.  We claim that
  $\bar{\pi}^j_{\fu} (\fw^j/ \fw^{j+1}) = \fu^j/\fu^{j+1}$ for $1 \le
  j \le e_P - r$.  Observe that $t^{-1} \fg \supset \fP^{\ell} \supset
  \fP^{\ell+1} \supset t \fg$ for $-e_P \le \ell \le 0$.  Therefore,
  we may take a representative $\bn \in t^{-1} \gl_n(\cplx) +
  \gl_n(\cplx)$ for $\an + \fP^{-r+1}$.  Similarly, choose a
  representative $X + t^{-1} X'$ for $\bar{X} \in \bfP^{j-e_P+r}$,
  where $X, X' \in \gl_n(\cplx)$.  It follows that $\langle \ad(X+
  t^{-1} X') \bn, Y \rangle_\nu = \langle \ad(X) \bn, Y \rangle_\nu$
  whenever $Y \in \fu$.  This proves the claim, and we conclude that
  $\ad^* (\fw^1) (\alpha)$ surjects onto $\fu^\vee / \fu^{e_P - r +
    1}$ .
  
  Finally, let $X \in \fw^1$.  The action of $\GL_n(\cplx)$ on
  $\widetilde{\mathscr{M}}(A))$ gives rise to a map $\gl_n(\cplx)\to
  T_s\widetilde{\mathscr{M}}(A)\subset
  \fu\backslash\gl_n(\cplx)\times\fg^\vee$; explicitly, $X\mapsto
  (-X,\ad^*(X)\a)$, which is sent to $\res(\ad^*(X) (\alpha))$ by
  $d\tmug$.  Therefore, $d \tmug$ maps tangent vectors coming from
  $\fw^1\subset \gl_n(\cplx)$ surjectively onto $\fu^\vee/\fu^{e_P - r
    + 1}$.  It follows that the image of $d \tmug$ contains
  $\fu^\vee$, so $\tmug$ is a submersion.

\end{proof}

\begin{lemma}\label{freeact}
$\GL_n(\cplx)$ acts freely on $\widetilde{\mathscr{M}}(A)$.
\end{lemma}
\begin{proof}
  Suppose that $h\in \GL_n(\cplx)$ fixes $(U g, \alpha)$.  In
  particular, $U g h = U g$, so $ {}^g h\in
  U$.
  To show that $h=1$, it suffices to show that ${}^g h=1$, so without
  loss of generality, we may assume that $g = 1$ and $h \in U$.
  By Proposition~\ref{duality}, there exists a representative
  $\alpha_\nu \in \gl_n(F)$ for $\alpha$ with terms only in
  nonpositive degrees.  The fact that $\Ad^*(h^{-1})(\alpha)=\alpha$
  implies that $\Ad(h^{-1}) (\alpha_\nu) = \alpha_\nu + X$ for $X \in
  \mathfrak{g}^1$.  Since $h \in \GL_n(\cplx)$, $X = 0$, and we see
  that $\Ad(h^{-1})(\alpha_\nu) = \alpha_\nu$.

  We will show that $h$ is $P^1$-conjugate to an element of
  $T(\mathfrak{o})$.  In particular, since $P^1 \cong U\ltimes
  G^1$,  we
  see that $h$ is $U$-conjugate to an element of $T(\mathfrak{o}) G^1
  \cap \GL_n(\cplx) = \Tfl$.  Since $\Tfl \cap P^1$ is trivial,
  Corollary~\ref{uniformsplitting} implies that $h = 1$.

  Take $p \in P^1$ such that $\Ad^*(p) (\alpha) = A^1$; thus, $\Ad(p)
  (\alpha_\nu) \in \ft +\mathfrak{P}$. By
  Lemma~\ref{A1}, there exists $p' \in P^1$ and a representative
  $A^1_\nu \in \ft^{-r}$ of $A^1$ such that
  $\Ad(p') (\Ad(p) (\alpha_\nu)) = A^1_\nu$.  Therefore, setting $q =
  p'p\in P^1$, $\Ad(({}^q h)^{-1}) A^1_\nu = A^1_\nu$.  By
  Lemma~\ref{lemrss}, ${}^q h \in T\cap G= T(\mathfrak{o})$.
\end{proof}

\begin{lemma}\label{412}
  If $(Q g_1, \alpha)$ and $(Q g_2, \alpha)$ both lie in
  $\mathscr{M}(A)$, then $g_2 = p g_1$ for some $p \in Q$.
  Moreover, if $(U g_1, \alpha)$ and $(U g_2, \alpha)$ both
  lie in $\widetilde{\mathscr{M}}(A)$, then $g_2= u s g_1$ for some
  $u \in U$ and $s \in \Tfl$.
  \end{lemma}
\begin{proof}
  Notice that $(Q , \Ad^*(g_1) \alpha)$ and $(Q g_2
  (g_1^{-1}), \Ad^*(g_1) \alpha)$ satisfy the conditions of the first
  statement.  There is a similar reformulation of the second
  statement.  Thus, we may assume without loss of generality that
  $g_1$ is the identity; we set $g_2 = g$.

  In the first case, note that by Lemmas~\ref{A1} and \ref{A0}, there
  exist $p_1, p_2 \in P$ such that $\Ad(p_1 g ) (\alpha_\nu) = A_\nu =
  \Ad(p_2) (\alpha_\nu)$ for some $A_\nu \in
  \mathfrak{t}$.  Since $p_1 g p_2^{-1}$ centralizes the regular
  semisimple element $A_\nu$, $p_1 g p_2^{-1} \in T \cap G =
  T(\mathfrak{o})$.  We conclude that $g \in P \cap
  \GL_n(\cplx)=Q$.

  In the second case, the same argument shows that whenever $\Ad^*(g)
  (\alpha) \in \orbit^1$, $g = p_1^{-1} s p_2 $ for some $s \in
  T(\mathfrak{o})$ and $p_i \in P^1$.  Since $P^1$ is normal in $P$,
  $g = u s$ for some $u \in P^1$.  By
  Corollary~\ref{uniformsplitting}, we may assume that $s \in \Tfl$.
  This implies that $u\in \GL_n(\cplx) \cap P^1 = U$ as desired.
\end{proof}

\begin{lemma}\label{short}
Let $\alpha \in \mathfrak{g}^\vee$ be a functional such that
$\pi_{\mathfrak{P}^1} (\alpha) = A^1$.  Then, if $s \in T(\mathfrak{o})$,
$\pi_{\mathfrak{P}^1} (\Ad^*(s ) \alpha)  = A^1$.
\end{lemma}
\begin{proof}
  Since any representative of $A^1$ lies in $\ft(\fo)$, $\an \in \ft
  + \mathfrak{P}$.  The lemma is now clear, since $T(\mathfrak{o})$
  preserves $\mathfrak{P}$ and stabilizes $\ft$.
  \end{proof}

We are now ready to describe the relationship between
$\mathscr{M}(A)$ and $\widetilde{\mathscr{M}}(A)$.  Recall, from
Lemma~\ref{410}, that $\Tfl = T(\mathfrak{o}) \cap \GL_n(\cplx)$.
There is a left action of $\Tfl$ on $\widetilde{\mathscr{M}} (A)$
given by $s (U g, \alpha) = (U s g, \alpha)$.  To see this,
note that by assumption,
$\pi_{\mathfrak{P}^1} (\Ad^*(g) (\alpha)) \in \orbit^1$, so there
exists $u \in P^1$ such that $\Ad^*(u) (\pi_{\mathfrak{P}^1} (\Ad^*(g)
(\alpha))) = A^1$.  We wish to show that there exists $u' \in
P^1$ such that $\Ad^*(u') (\pi_{\mathfrak{P}^1} (\Ad^*(s g) (\alpha))) =
A^1$.  However,
\begin{multline}\label{Tact}
\Ad^*(u^s) (\pi_{\mathfrak{P}^1}( \Ad^*(s g) (\alpha))) =
\pi_{\mathfrak{P}^1} (\Ad^*({}^s u)  \Ad^*(s g) (\alpha)) = \\
\pi_{\mathfrak{P}^1} (\Ad^*(s)  \Ad^*(u g) (\alpha)) =
A^1,
\end{multline}
where the last equality follows from Lemma~\ref{short}.
In particular, $s (U g, \alpha) \in \widetilde{\mathscr{M}}(A)$.

We will show that this action is Hamiltonian with moment map
$\mu_{\Tfl}$
defined as follows.  Take $(U g, \alpha) \in
\widetilde{\mathscr{M}}(A)$.  There exists $u \in P^1$ such that
\begin{equation} \label{ucondition}
\pi_{\mathfrak{P}^1} (\Ad^*(u g) (\alpha)) = A^1.
\end{equation}
Define a map 
\begin{equation*}\label{mut}
\mu_{\Tfl} (U g, \alpha) = - (\Ad^*(u g) (\alpha))\vert_{\Tfl}.
\end{equation*}

We need to show that this map is well-defined.  Let $\tilde{A} =
\Ad^*(u g) (\alpha)$.  Suppose that $u' \in P^1$ satisfies
\eqref{ucondition}.  Observe that $\Ad^*(u' u^{-1}) (A^1) = A^1$.  By
Lemma~\ref{isotropy}, $u' u^{-1} \in T(\mathfrak{o}) P^{r}$.  It
suffices to show that whenever $s \in T(\mathfrak{o})$ and $p \in
P^r$, $(\Ad^*(s p)(\tilde{A}))\vert_{\Tfl}=(\tilde{A})\vert_{\Tfl}$.
In fact, we will prove the stronger statement:
\begin{equation} \label{tpres}\pi_{\ft\cap \fP}(\Ad^*(s p)
(\tilde{A}))=\pi_{\ft\cap \fP}(\tilde{A}).
\end{equation}

Fix a representative $\tilde{A}_\nu\in\mathfrak{P}^{-r}$.  By
Proposition \ref{cores4}, the projection $(\gl(F))^\vee \to
\mathfrak{t}^\vee$ corresponds to tame corestriction
$\pi_{\mathfrak{t}} : \gl_n(F) \to \mathfrak{t}$ after dualizing.
Thus, $\pi_\ft (\Ad(s p) (\tilde{A}_\nu))$ is a representative of
$\pi_{\ft\cap\fP}(\Ad^*(s p) \tilde{A})$.  Since $\pi_{\mathfrak{t}}$
commutes with the action of $\mathfrak{t}$, $\pi_{\mathfrak{t}} (\Ad(s
p) (\tilde{A}_\nu) )= \pi_{\mathfrak{t}} (\Ad(p) (\tilde{A}_\nu))$.
By Proposition~\ref{cores3}, $\pi_{\mathfrak{t}} (\Ad(p)
(\tilde{A}_\nu)) - \pi_{\mathfrak{t}}
(\tilde{A}_\nu)\in\mathfrak{P}^1=\fP^\perp$, so $\pi_{\mathfrak{t}}
(\tilde{A}_\nu)$ is a representative for both functionals in
\eqref{tpres} as desired.

The following lemma generalizes \cite[Lemma 2.3]{Boa}.  
The proof is more complicated, due to the absence of a 
`decoupling' map in the general case.
\begin{proposition}\label{411}  Let $\Lambda= A\vert_{\tfl}$.
  The action of $\Tfl$ on $\widetilde{\mathscr{M}}(A)$ is Hamiltonian
  with moment map $\mu_{\Tfl}$.  Moreover,
\begin{equation*}
\mathscr{M}(A) \cong \widetilde{\mathscr{M}}(A) \sslash_{-\Lambda} \Tfl.
\end{equation*}
\end{proposition}
\begin{proof}
  Recall that $\orbit^1$ be the $P^1$-coadjoint orbit containing
  $A^1$.  If $\beta \in \orbit^1$, we may take $\alpha \in
  \fg^\vee$ such that $\pi_{\fP^1} (\alpha) = \beta$.  The torus
  $\Tfl$ acts on $\orbit^1$ by $s \cdot \beta = \pi_{\fP^1}( (\Ad^*(s)
  (\alpha)))$.  (One sees that this element is in $\orbit^1$ by an
  argument similar to the one used to show \eqref{Tact}, and it is
  easily checked that it is independent of the choice of $\alpha$.)

  We construct a moment map for this action.  Consider the semi-direct
  product $\Tfl \ltimes P^1 \subset P$, and lift $A^1 \in \orbit^1$ to
  $\tilde{A} \in (\tfl)^\perp \subset (\tfl \times
  \mathfrak{P}^1)^\vee$.  Let $\tilde{\orbit} \subset (\tfl \times
  \mathfrak{P}^1)^\vee$ be the coadjoint orbit of $\tilde{A}$.  Since
  $\Tfl$ stabilizes $A^1$ by Lemma~\ref{short}, it is clear that it
  stabilizes $\tilde{A}$ as well.  In particular, $P^1$ acts
  transitively on $\tilde{\orbit}$.  We will prove in
  Lemma~\ref{wrapup} that the natural map $\tilde{\pi} :
  \tilde{\orbit} \to \orbit^1$ is a $\Tfl$-equivariant symplectic
  isomorphism.  Therefore, the moment map $\tilde{\mu} : \orbit^1 \to
  (\tfl)^\vee$ is given by
\begin{equation*}
\tilde{\mu} (\beta) = \pi_{\Tfl} (\tilde{\pi}^{-1} (\beta)),
\end{equation*}
where $\pi_{\Tfl}$ is the projection $(\tfl\times \mathfrak{P}^1)^\vee
\to (\tfl)^\vee$.

We remark that if a different lift of $A^1$ is chosen, say
$\tilde{A} + \gamma$ for $\gamma \in (\mathfrak{P}^1)^\perp \cong (\tfl)^\vee$, 
then 
\begin{equation}\label{altlift}
(\Ad^*(u) (\tilde{A} + \gamma)) (z) = 
\Ad^*(u) (\tilde{A}) (z) + \gamma(z)
\end{equation}
for $u\in P^1$ and $z \in \tfl$.
In particular, this changes $\widetilde{\mu}$ by a constant $\gamma$.

The action of $\Tfl$ on $\widetilde{\mathscr{M}}(A)$ descends from a
Hamiltonian action of $\Tfl$ on $T^* G \times \orbit^1$.  Indeed, if
$(g, \alpha, \beta) \in T^*G \times \orbit^1$, then $s (g, \alpha,
\beta) = (s g, \alpha, s\cdot\beta)$ defines a Hamiltonian action; the
moment map $\mu'$ is given by the sum of the natural moment
map on $T^* G$ and $\tilde{\mu}$ .
Moreover, $\Tfl$ preserves $\mu_{P^1}^{-1} (0)$, and the map from
$\mu_{P^1}^{-1} (0) \to \widetilde{\mathscr{M}}(A)$ is
$\Tfl$-equivariant.

We will show that the restriction of   $\mu'$ to $\mu_{P^1}^{-1} (0)$ is
$P^1$-invariant.  Let $(g, \alpha, \beta) \in \widetilde{\mathscr{M}}(A)$,
and define $\phi(g, \alpha)$ to be the projection of 
$\Ad^*(g) (\alpha)$ onto $(\tfl \times \mathfrak{P}^1)^\vee$.
Then, if $u \in P^1$,
\begin{equation*}\begin{aligned}
\mu' (u (g, \alpha, \beta)) & = 
\mu' (u g, \alpha, \Ad^*(u) \beta)  \\
& =\pi_{\Tfl} \left(-\Ad^*(u) \phi(g ,\alpha) +
\Ad^*(u) \tilde{\pi}^{-1} (\beta) \right) 
\end{aligned}
\end{equation*}
However, $\beta = \pi_{\mathfrak{P}^1} (\Ad^*(g) (\alpha))$
lies in $\orbit^1$,  so $\phi(g, \alpha)$ must lie in 
a coadjoint orbit containing $\tilde{A} - \gamma$
for some $\gamma \in (\tfl)^\vee$.
Equation~\eqref{altlift} implies that 
\begin{equation*}
\pi_{\Tfl} (- \phi(g, \alpha)+\tilde{\pi}^{-1}(\beta)) = \pi_{\Tfl} \left(
-\Ad^*(u) \phi(g, \alpha) +\Ad^*(u) \tilde{\pi}^{-1} (\beta) \right)
= \gamma.
\end{equation*}
Thus, $\mu'$ is $P^1$-invariant.  By Lemma~\ref{moeq}, 
the action of $\Tfl$ on $\widetilde{\mathscr{M}}(A)$
is Hamiltonian, and the moment map descends from $\mu'$.

It remains to show
that $\gamma = \mu_{\Tfl} (U g, \alpha)$.  By $P^1$-invariance,
it suffices to consider the case where $\pi_{\mathfrak{P}^1} (\Ad^*(g) (\alpha)) = A^1$.
By construction, $\tilde{\mu}(A^1) = 0$, so 
$\gamma = -\Ad^*(g) (\alpha) |_{\tfl} = \mu_{\Tfl}(U g, \alpha)$.

We now prove that $\mathscr{M}(A) \cong \widetilde{\mathscr{M}}(A)
\sslash_{-\Lambda} \Tfl$.  First, we show that if $\mu_{\Tfl}(U
g,\alpha)= -\Lambda$, then $(Q g,\a)\in\orbit$.  Let $u \in P^1$
satisfy $\pi_{\mathfrak{P}^1} (\Ad^* (u g) \alpha) = A^1$.  Choosing a
representative $\an$ for $\a$, we have $\Ad(ug)(\an)\in \An+\fP$.
Applying $\pi_\ft$, we see that $\pi_\ft(\Ad(ug)(\an))=\An+z$ for some
$z\in\ft(\fo)$.  In fact, $z\in\ft\cap \fP^1$ because the restrictions
of $\Ad^* (u g) \alpha$ and $A$ to $(\tfl)^\vee$ agree.  By
Proposition~\ref{cores3}, there exists $X \in \mathfrak{P}^r$ such
that $\Ad(1+X) \Ad (u g) \alpha_\nu) \in A_\nu + \mathfrak{P}^1$, so
$\pi_{\mathfrak{P}} (\Ad^*((1+X)u g) \alpha) =
\pi^*_\ft(A)$, i.e., $\Ad^*(g) (\alpha) \in
\orbit$.  Thus, we have a map $\mu_{\Tfl}^{-1} (-\Lambda) \to
\mathscr{M}(A)$ given by $(U g, \alpha) \mapsto (Q g, \alpha)$.
Lemma~\ref{412} shows that the fibers of the map are $\Tfl$-orbits,
and we obtain the desired isomorphism.
\end{proof}

\begin{lemma}\label{wrapup}
In the notation from the previous proof, the map $\tilde{\pi} : \tilde{\orbit} \to \orbit^1$
is a $\Tfl$-equivariant symplectic isomorphism.
\end{lemma}
\begin{proof}
  First, we show $\Tfl$-equivariance.  We have already observed that
  there is a transitive action of $P^1$ on $\tilde{\orbit}$ and that
  $\Tfl$ stabilizes $\tilde{A}$.  For any $s\in\Tfl$ and $u\in P^1$,
  we calculate
\begin{equation*}\begin{aligned}
\tilde{\pi} (\Ad^*(s) \Ad^*(u) \tilde{A})& = 
\tilde{\pi} (\Ad^*({}^su) \tilde{A}) = \Ad^*({}^su) \tilde{\pi} (\tilde{A}) \\
& = \pi_{\fP^1}(\Ad^*(s) \Ad^*(u) \Ad^*(s^{-1}) (\pi^*_\ft(A))) \\ 
& = \pi_{\fP^1}(\Ad^*(s) \Ad^*(u) (\pi^*_\ft(A)))=s\cdot (\Ad^*(u) (\pi^*_\ft(A))).
\end{aligned}
\end{equation*}

Next, we show that the stabilizer of $\tilde{A}$ in $P^1$ is the same
as the stabilizer of $\pi^*_\ft(A)$.  Let $\tilde{A}_\nu \in \mathfrak{P}^{-r}$
be a representative for $\tilde{A}$.  In fact, $\tilde{A}_\nu \in
(\mathfrak{t} + \mathfrak{P}) \cap (\tfl)^\perp$.  By
Lemma~\ref{isotropy}, the stabilizer of $\pi^*_\ft(A)$ is precisely
$(T(\mathfrak{o}) \cap P^1) P^r$ if $r\ge 2$ or $P^1$ if $r=1$ (in
which case $\pi^*_\ft(A)$ is a singleton orbit).  It suffices to show that this
group stabilizes $\tilde{A}$, as the stabilizer of $\tilde{A}$ is a
subgroup of the stabilizer of $\pi^*_\ft(A)$.  Since $\tilde{A}_\nu \in
\mathfrak{t}+\mathfrak{P}$, $T\cap P^1$ stabilizes $\tilde{A}$.  Now,
take $u \in P^r$, $z \in \tfl$ and $X \in \mathfrak{P}^1$.  We see
that
\begin{equation*}
\Ad^*(u) (\tilde{A}) (z + X) = \tilde{A} (\Ad(u^{-1}) z + \Ad(u^{-1}) X) =
\tilde{A} (\Ad(u^{-1}) z) + \tilde{A} (X),
\end{equation*}
so we need only check that $\Ad^*(u) \tilde{A} (z) = \tilde{A}(z)$.
However, by Proposition~\ref{cores4}, $\Ad^*(u) \tilde{A} (z) =
\langle \Ad (u) \tilde{A}_\nu, z \rangle_\nu = \langle
\pi_{\mathfrak{t}} (\Ad (u) \tilde{A}_\nu), z \rangle_\nu$.
Proposition~\ref{cores3} implies that $\pi_{\mathfrak{t}} (\Ad(u)
\tilde{A}_\nu) \equiv \pi_{\mathfrak{t}} (\tilde{A}_\nu)
\pmod{\mathfrak{P}^1}$.  It follows that the stabilizer of $\tilde{A}$
is, indeed, the same as the stabilizer of $\pi^*_\ft(A)$; moreover, since
$\tilde{\pi}$ is a $P^1$-map, it follows that $\tilde{\pi}$ is an
isomorphism.

Finally, we need to show that $\tilde{\pi}$ preserves the natural symplectic
form on each coadjoint orbit.  In particular, since $\tilde{\orbit}$ and
$\orbit^1$ are $P^1$ orbits, it suffices by transitivity to show that
the symplectic forms are the same at 
$\tilde{A}$ and $A^1$.  In other words,
we need to show that
$\tilde{A} ([ X_1+z_1, X_2+z_2]) =A^1 ([X_1, X_2])$
for $z_j \in \tfl$ and $X_j \in \fP^1$.
This is clear, since the restriction of $\tilde{A}$ to
$\mathfrak{P}^1$ is exactly $A^1$, and $\tfl$ lies
in the kernel of the symplectic form at $\tilde{A}$.
\end{proof}

\subsection{Proof of the theorem}\label{proof}

Let $V$ be a trivializable vector bundle on $\proj^1$, and let
$\nabla$ be a meromorphic connection with singularities at $\{x_1,
\ldots, x_m \}$.  We assume that $\nabla$ has compatible framings
$\{g_1, \ldots, g_m\}$ at each of the singular points and that
$\nabla$ has formal type $A_i\in\ft_i^\vee$ at $x_i$.  We define
$\orbit_i \subset \mathfrak{P}_i$ (resp. $\orbit^1_i \subset
\mathfrak{P}_i^1$) to be the coadjoint orbit of $\pi^*_\ft(A_i)$ under $P_i$
(resp. $P^1_i$).  We fix a global trivialization as in the beginning
of the section; as usual, we will use this fixed trivialization to
identify subgroups of $\GL(V_x)$ and $\GL_n(F_x)$, etc.

\begin{definition}
  The \emph{principal part} $[\nabla_x]^{pp}$ of $\nabla$ at $x$ is
  the image of $ [\nabla_x]$ in $\fg_x^\vee$ by the residue-trace
  pairing.
\end{definition}
 To give an example, if
$[\nabla_0] = M_{-1} \frac{dt}{t^2} + M_{0} \frac{dt}{t} + M_1 dt +
M_2 t dt + \ldots$, with the $M_i\in\GL_n(\cplx)$, then
$[\nabla_0]^{pp} (X) = \Res (\Tr ((M_{-1} \frac{dt}{t^2} + M_{0} \frac{dt}{t}) X))$
for any $X \in \fg_0$.  

We set $[\nabla_i]^{pp}$ to be the principal part of
$\nabla$ at $x_i$.  It is a consequence of
the duality theorem (\cite[Theorem II.2]{Se}) that
$\nabla$ is uniquely determined by the collection
$\{[\nabla_i]^{pp}\}$.  Moreover, the residue theorem (\cite[Proposition
II.6]{Se}) shows that $\sum_{i} \res([\nabla_i]^{pp}) = 0$.

If $g_i$ is a compatible framing for $\nabla$ at $x_i$, 
\begin{equation*}
  \pi_{\mathfrak{P}_i} ( (\Ad^*(g_i) [\nabla_i]^{pp})  \in \orbit_i \qquad
  \text{and} \qquad \pi_{\mathfrak{P}^1_i} (  \Ad^*(g_i) [\nabla_i])^{pp})  \in \orbit^1_i.
\end{equation*}
This follows from the observation that $g_i \cdot [\nabla_i]^{pp} =
\Ad^*(g_i) [\nabla_i]^{pp}$ and Proposition~\ref{gaugecoadjoint}.  In
particular, since $g_i \cdot [\nabla_i]$ is formally gauge equivalent
to $\pi^*_\ft(A)$ by an element of $p_i \in
P_i^1$, it follows that $\Ad^*(p_i) \Ad^*(g_i) [\nabla_i]^{pp} =
\pi^*_\ft(A)$.

Finally, we need to define extended orbits $\mathscr{M}(A)$ and
$\widetilde{\mathscr{M}}(A)$ in the case where $A$ is a regular
singular formal type.  In particular, the corresponding uniform
stratum is of the form $(G, 0, \b)$.  Since $\pi^*_\ft(A)$ is a functional on
$\fg$ that kills $\fg^1$, we may think of $\pi^*_\ft(A)$ as an element of
$\gl_n(\cplx)^\vee$.  We define $(\tfl)' \subset \gl_n(\cplx)^\vee$ to
be the set of functionals of the form $\phi(X) =
\Tr(DX)$,  where $D\in
\tfl $ is a diagonal matrix with distinct eigenvalues modulo $\Z$.

The following definition comes from Section 2 of \cite{Boa}.
\begin{definition} Let $A$ be a formal type corresponding to a stratum
  $(G, 0, \b)$.  Define $\mathscr{M} (A) = \orbit_A$, the coadjoint
  orbit of $\pi^*_\ft(A)$ in $\mathfrak{g}^\vee$.
  Moreover, let
\begin{equation*}
\widetilde{\mathscr{M}}(A) := \{ (g, \a) \in \GL_n(\cplx) \times \gl_n(\cplx)^\vee \mid 
\Ad^*(g) \a \in (\tfl)' \} \subset G \times \mathfrak{g}^\vee.
\end{equation*}
\end{definition}
\begin{rmk}
  This definition of $\mathscr{M}(A)$ coincides with the definition
  for $r>0$ given in \eqref{madef} (where now $Q=\GL_n(\cplx)$), but
  this is not true for $\widetilde{\mathscr{M}}(A)$.  Indeed,
  $\widetilde{\mathscr{M}}(A)$ is independent of formal type when
  $r=0$.

  However, the essential results of Section~\ref{extended orbits}
  remain true in the regular singular case.  By \cite[Theorem
  26.7]{GuSt}, $\widetilde{\mathscr{M}} (A)$ is a symplectic
  submanifold of $T^* \GL_n(\cplx)$.  Moreover, $\Tfl$ (resp.
  $\GL_n(\cplx)$) acts on $\widetilde{\mathscr{M}} (A)$ by left
  multiplication (resp. inversion composed with right multiplication).
  The moment map for $\Tfl$ is simply $(g, X) \mapsto -\Ad^*(g) (X)$,
  and the map $(g, \a) \mapsto \a$ induces an isomorphism
  $\widetilde{\mathscr{M}} (A) \sslash_{-A}\Tfl \cong \mathscr{M}(A)$.
\end{rmk}

\begin{proof}[Proof of Theorem~\ref{modthm}]
  For each $x_i$, set $\mathscr{M}_i = \mathscr{M}(A^i)$, and
  $\widetilde{\mathscr{M}}_i = \widetilde{\mathscr{M}}(A^i)$.  As  above, a
  meromorphic connection $\nabla$ on $\proj^1$ is uniquely determined
  by the principal parts at its singular points $\{x_i\}$.  
  Moreover,
  any collection $\{ M_i \}$, where $M_i\in \fg_i^\vee$, that also
  satisfies the residue condition 
  \begin{equation}\label{rescondition}
\sum_i \res (M_i ) = 0
\end{equation}
  determines a unique connection with singularities only at the
  $x_i$'s and with principal part at $x_i$ given by $M_i$.

  There is a map $\mathscr{M}_i \to \fg_i^\vee$ obtained by taking
  $(Q_i g, \alpha_i)$ to $\a_i$.  Lemma~\ref{412}
  implies that this map is one-to-one, and it identifies elements of
  $\mathscr{M}_i$ with the principal part of a framed connection at
  $x_i$ with formal type $A_i$.  We conclude that any element of
  $\prod_i \mathscr{M}_i$ satisfying \eqref{rescondition} uniquely
  determines a connection $\nabla$ with framed formal type $A_i$ at
  $x_i$.

  The action of $\GL_n(\cplx)$ on $\prod_i \mathscr{M}_i$ induced by
  its action on global trivializations of $V$ is the product of the
  left actions on $\mathscr{M}_i$ given in \eqref{rtact}. Therefore,
  it follows from Proposition~\ref{glmoment} that the moment map of
  this action is simply
\begin{equation*}
\mu : \prod_i (Q_i g_i, \alpha_i) \mapsto \sum_i \res(\alpha_i).
\end{equation*}  
The moment map $\tmu:\prod_i
\widetilde{\mathscr{M}}_i\to\gl_n(\cplx)^\vee$ is defined similarly.
The residue condition \eqref{rescondition} now translates into an
$m$-tuple lying in $\mu^{-1}(0)$, so
\begin{equation*}
\mathscr{M}^*(\mathbf{A}) \cong \left(\prod_i \mathscr{M}_i\right) \sslash_0 \GL_n(\cplx).
\end{equation*}
Similarly, $\widetilde{\mathscr{M}}^*(\mathbf{A}) \cong \left(\prod_i
  \widetilde{\mathscr{M}}_i\right) \sslash_0 \GL_n(\cplx)$: the map
$\tmu^{-1}(0)\to \widetilde{\mathscr{M}}^*(\mathbf{A})$ takes
$(U_ig_i,\a_i)$ to the data $(V,\nabla,\mathbf{g})$, where $\nabla$
has principal part $\a_i$ at $x_i$ and $\mathbf{g}=(U_ig_i)$.

By Lemma~\ref{freeact}, $\GL_n(\cplx)$ acts freely on $\widetilde{\mathscr{M}}_i$,
so the action on $\prod_i \widetilde{\mathscr{M}}_i$ is free.  Moreover,
Lemma~\ref{submersion} states  that $\tmu$ is a submersion on each factor,
so $\tmu$ is a submersion.  Therefore, $\tmu^{-1} (0)$ is smooth.
It follows
that $\widetilde{\mathscr{M}}^*(\mathbf{A})$ is a smooth symplectic
variety.

Finally, let $\Lambda_i = A_i\vert_{\tfl_i}.$ The action of $\prod_i
\Tfl_i$ on $\prod_i \widetilde{\mathscr{M}}_i$ commutes with the
action of $\GL_n(\cplx)$, so by Lemma~\ref{moeq}, there is a natural
Hamiltonian action of $\prod_i
\Tfl_i$ on $\widetilde{\mathscr{M}}^*(\mathbf{A})$.  Similarly, there
is a Hamiltonian action of $\GL_n(\cplx)$ on 
\begin{equation*}
\prod_i (\widetilde{\mathscr{M}}_i \sslash_{- \Lambda_i}
\Tfl_i )\cong (\prod_i
\widetilde{\mathscr{M}}_i) \sslash_{\prod_i (- \Lambda_i)} \prod
\Tfl_i.
\end{equation*}
We now see that taking the iterated symplectic reduction of the
product of local data by $\GL_n(\cplx)$ and the product of the local
tori is independent of order:
\begin{equation*}
\left( (\prod_i \widetilde{\mathscr{M}}_i) \sslash_0 \GL_n(\cplx) \right) \sslash_{\prod_i (-\Lambda_i)}
\prod_i \Tfl_i \cong 
\prod_i (\widetilde{\mathscr{M}}_i \sslash_{- \Lambda_i}
\Tfl_i ) \sslash_0 \GL_n(\cplx);
\end{equation*}
indeed, both are isomorphic to the symplectic reduction via the
product action: $\prod \widetilde{\mathscr{M}}_i \sslash_{(0,\prod_i
  (- \Lambda_i))} (\GL_n(\cplx)\times\prod \Tfl_i)$.  By
Proposition~\ref{411}, it follows that
\begin{equation*}
\widetilde{\mathscr{M}}^*(\mathbf{A}) \sslash_{\prod_i (-\Lambda_i)} \prod_i \Tfl_i
\cong \mathscr{M}^*(\mathbf{A}).
\end{equation*}
\end{proof}
\begin{rmk}
  In the case $m >1$ above, we only require $\mu$ to be a submersion
  on one factor in $\prod_{i = 1}^m \widetilde{\mathscr{M}}(A_i)$.  In
  particular, the residue map on $\widetilde{\mathscr{M}}(A_1) \times
  \prod_{i = 2}^m \mathscr{M}(A_i)$ is submersive.  Moreover, by
  Lemma~\ref{freeact}, the action of $\GL_n(\cplx)$ on
  $\widetilde{\mathscr{M}}(A_1) \times \prod_{i = 2}^m
  \mathscr{M}(A_i)$ is free.  Therefore, $\mathscr{M}' (\mathbf{A}) =
  \left( \widetilde{\mathscr{M}}(A_1) \times \prod_{i = 2}^m
    \mathscr{M}(A_i)\right) \sslash_0 \GL_n(\cplx)$ is smooth, and
  $\mathscr{M}' (\mathbf{A}) \sslash_{-\Lambda_1} \Tfl_1 \cong
  \mathscr{M}(\mathbf{A})$.
\end{rmk}

We state here a more general version of Theorem~\ref{modthm}.  The
proof is essentially the same; however, it allows us to consider
regular singular points with arbitrary monodromy.  In particular, this
construction includes the $\GL_n$ case of the flat $G$-bundle
constructed in \cite{FGr}.

Let $\{\hat{\orbit}_j\}$ be a collection of `non-resonant' adjoint
orbits in $\gl_n(\cplx)$; this means that the distinct eigenvalues of
elements $\hat{\orbit}_j$ do not differ by nonzero integers.  Using
the trace pairing, we may identify $\hat{\orbit}_j$ with a coadjoint
orbit $\orbit_j \subset \gl_n(\cplx)^\vee$.  Thus, we say that a
connection $\nabla$ on the trivial bundle $V$ over $C$ has residue in
$\orbit_j$ at $y_j \in C$ if the principal part at $y_j$ corresponds
to an element of $\orbit_j$ in $\gl_n(\cplx)^\vee$.  Equivalently,
$[\nabla_{y_j}]^{pp} = X \frac{dt}{t}$ for some $X \in
\hat{\orbit}_j$.  By the
standard theory of regular singular point connections (see, for
example, \cite[Chapter II]{Wa}), if a connection $(V, \nabla)$ has
non-resonant residue $X \in \hat{\orbit}_{j}$, then $(V, \nabla)$ is
formally equivalent to $d + X \frac{dt}{t}$.

Let $\mathbf{B} = \{\orbit_j\}$ be a finite collection of non-resonant
adjoint orbits corresponding to a collection of regular singular
points $\{y_j\} \subset C$, and let $\mathbf{A} = \{A_i\}$ be a finite
collection of formal types at $\{x_i\} \subset C$, disjoint
from $\{y_j\}$.
\begin{definition}
Define $\mathscr{M} (\mathbf{A}, \mathbf{B})$ to be the moduli space of
connections $\nabla$ on the trivial bundle $V$ with the following properties:
\begin{enumerate}
\item $(V, \nabla)$ is compatibly framed at each $x_i$, with formal type $A_i$;
\item $(V, \nabla)$ is regular singular and has residue in
  ${\orbit}_j$ at each $y_j$.
\end{enumerate}
If $\mathbf{A}$ is nonempty, we define the extended moduli space  $\mathscr{M}
(\mathbf{A}, \mathbf{B})$ of isomorphism classes of data
$(V,\nabla,\mathbf{g})$, where $(V,\nabla)$ satisfy the conditions
above and $\mathbf{g}=(g_i)$ is a collection of local compatible
framings at the $x_i$'s.
\end{definition}

We omit the proof of the following theorem, since it is almost
identical to the proof of Theorem~\ref{modthm}.  We note that a
similar construction is used in \cite{Boa2} and \cite{Boa1} in the
case where the $A_i$ are totally split.
\begin{theorem}\label{thm3} \mbox{}
\begin{enumerate}\item The moduli space $\mathscr{M} (\mathbf{A},
  \mathbf{B})$ is a symplectic reduction of the product of local data:
  \begin{equation*}\mathscr{M} (\mathbf{A}, \mathbf{B}) \cong \left[ \left( \prod_i
      \mathscr{M}(A_i) \right) \times \left( \prod_j \orbit_j \right)
  \right] \sslash_0 \GL_n(\cplx).
\end{equation*}
\item If $\mathbf{A}$ is nonempty, then $\widetilde{\mathscr{M}}
  (\mathbf{A}, \mathbf{B})$ is a symplectic manifold, and \begin{equation*}\widetilde{\mathscr{M}} (\mathbf{A}, \mathbf{B}) = 
\left[ \left( \prod_i \widetilde{\mathscr{M}}(A_i) \right) \times
\left( \prod_j \orbit_j \right) \right] \sslash_0 \GL_n(\cplx).
\end{equation*}
\item 
\begin{equation*}
\mathscr{M} (\mathbf{A}, \mathbf{B}) \cong
\widetilde{\mathscr{M}} (\mathbf{A}, \mathbf{B}) \sslash_{\prod_i (-\Lambda_i)} \left(\prod_i \Tfl_i\right).
\end{equation*}
\end{enumerate}
\end{theorem}

\end{document}